\documentclass[reqno, 12pt]{amsart}
\pdfoutput=1
\makeatletter
\let\origsection=\section \def\section{\@ifstar{\origsection*}{\mysection}}
\def\mysection{\@startsection{section}{1}\z@{.7\linespacing\@plus\linespacing}{.5\linespacing}{\normalfont\scshape\centering\S}}
\makeatother

\usepackage{amsmath,amssymb,amsthm}
\usepackage{mathrsfs}
\usepackage{mathabx}\changenotsign
\usepackage{dsfont}

\usepackage{graphicx}
\usepackage{float}

\usepackage[dvipsnames]{xcolor}
\usepackage{hyperref}
\hypersetup{
	colorlinks,
	linkcolor={red!60!black},
	citecolor={green!60!black},
	urlcolor={blue!60!black}
}

\definecolor{codelightgray}{gray}{0.8}
\definecolor{codeverylightgray}{gray}{0.9}

\usepackage{array,multirow,colortbl,enumerate}
\usepackage{caption} 

\usepackage{graphicx}

\usepackage[open,openlevel=2,atend]{bookmark}

\usepackage[abbrev,msc-links,backrefs]{amsrefs}
\usepackage{amsrefs}
\usepackage{doi}

\renewcommand{\PrintDOI}[1]{\doi{#1}}

\usepackage[OT2, T1]{fontenc}
\usepackage{lmodern}
\usepackage[babel]{microtype}
\usepackage[english]{babel}

\DeclareRobustCommand{\rn}[1]{  {\fontencoding{OT2}\selectfont#1}}

\usepackage{geometry}
\numberwithin{equation}{section}
\numberwithin{figure}{section}

\usepackage{enumitem}
\def\rmlabel{\upshape({\itshape \roman*\,})}

\def\nlabel{\upshape({\itshape \arabic*\,})}

\let\polishlcross=\l
\def\l{\ifmmode\ell\else\polishlcross\fi}

\def\paragraph#1{	\noindent\textbf{#1.}\enspace}

\let\sm=\setminus

\makeatletter
\def\moverlay{\mathpalette\mov@rlay}
\def\mov@rlay#1#2{\leavevmode\vtop{   \baselineskip\z@skip \lineskiplimit-\maxdimen
		\ialign{\hfil$\m@th#1##$\hfil\cr#2\crcr}}}
\newcommand{\charfusion}[3][\mathord]{
	#1{\ifx#1\mathop\vphantom{#2}\fi
		\mathpalette\mov@rlay{#2\cr#3}
	}
	\ifx#1\mathop\expandafter\displaylimits\fi}
\makeatother

\newcommand{\dcup}{\charfusion[\mathbin]{\cup}{\cdot}}
\newcommand{\bigdcup}{\charfusion[\mathop]{\bigcup}{\cdot}}

\DeclareFontFamily{U}  {MnSymbolC}{}
\DeclareSymbolFont{MnSyC}         {U}  {MnSymbolC}{m}{n}
\DeclareFontShape{U}{MnSymbolC}{m}{n}{
	<-6>  MnSymbolC5
	<6-7>  MnSymbolC6
	<7-8>  MnSymbolC7
	<8-9>  MnSymbolC8
	<9-10> MnSymbolC9
	<10-12> MnSymbolC10
	<12->   MnSymbolC12}{}
\DeclareMathSymbol{\powerset}{\mathord}{MnSyC}{180}

\usepackage{tikz}
\usetikzlibrary{calc,decorations.pathmorphing,decorations.pathreplacing}
\usetikzlibrary{intersections}
\usetikzlibrary {arrows.meta} 
\pgfdeclarelayer{background}
\pgfdeclarelayer{foreground}
\pgfdeclarelayer{front}
\pgfsetlayers{background,main,foreground,front}

\usepackage{subcaption}
\newcommand{\qedge}[7]{

	\ifx\relax#4\relax
		\def\qoffs{0pt}
	\else
		\def\qoffs{#4}
	\fi

	\def\qhedge{
		($#1+#3!\qoffs!-90:#2-#3$) --
		($#2+#1!\qoffs!-90:#3-#1$) --
		($#3+#2!\qoffs!-90:#1-#2$) -- cycle}

	\coordinate (12) at ($#1!\qoffs!90:#2$);
	\coordinate (13) at ($#1!\qoffs!-90:#3$);
	\coordinate (23) at ($#2!\qoffs!90:#3$);
	\coordinate (21) at ($#2!\qoffs!-90:#1$);
	\coordinate (31) at ($#3!\qoffs!90:#1$);
	\coordinate (32) at ($#3!\qoffs!-90:#2$);
	
	\def\nqhedge{
		(13) let \p1=($(13)-#1$), \p2=($(12)-#1$) in
			arc[start angle={atan2(\y1,\x1)}, delta angle={atan2(\y2,\x2)-atan2(\y1,\x1)-360*(atan2(\y2,\x2)-atan2(\y1,\x1)>0)}, x radius=\qoffs, y radius=\qoffs] --
		(21) let \p1=($(21)-#2$), \p2=($(23)-#2$) in
			arc[start angle={atan2(\y1,\x1)}, delta angle={atan2(\y2,\x2)-atan2(\y1,\x1)-360*(atan2(\y2,\x2)-atan2(\y1,\x1)>0)}, x radius=\qoffs, y radius=\qoffs] --
		(32) let \p1=($(32)-#3$), \p2=($(31)-#3$) in
			arc[start angle={atan2(\y1,\x1)}, delta angle={atan2(\y2,\x2)-atan2(\y1,\x1)-360*(atan2(\y2,\x2)-atan2(\y1,\x1)>0)}, x radius=\qoffs, y radius=\qoffs] --
		cycle}

		\ifx\relax#5\relax
		\def\qlwidth{1pt}
	\else
		\def\qlwidth{#5}
	\fi
	
		\ifx\relax#7\relax
		\fill \nqhedge;
	\else
		\fill[#7]\nqhedge;
	\fi

		\ifx\relax#6\relax
		\draw[line width=\qlwidth,rounded corners=\qoffs]\nqhedge;
	\else
		\draw[line width=\qlwidth,#6]\nqhedge;
	\fi
}

\let\epsilon=\varepsilon

\let\rho=\varrho
\let\theta=\vartheta

\def\FF{{\mathds F}}
\def\NN{{\mathds N}}

\def\ZZ{{\mathds Z}}
\def\HH{{\mathds H}}

\def\RR{{\mathds R}}
\def\CC{{\mathds C}}

\def\gth{\mathrm{girth}}
\def\GL{\mathrm{GL}}
\def\PGL{\mathrm{PGL}}
\def\PSL{\mathrm{PSL}}
\def\Cayley{\mathrm{Cayley}}
\def\Hom{\mathrm{Hom}}

\def\cofi{\mathrm{cf}}
\def\Sh{\mathrm{Sh}}

\def\ca{Apricot}    \def\cf{RoyalBlue!70}           \def\cc{Lavender!40} \def\cd{Yellow}
\def\ce{RoyalPurple!70}     \def\cb{Green!70}
\def\cg{Red}

\newcommand{\ccA}{\mathscr{A}}

\newcommand{\ccH}{\mathscr{H}}

\newcommand{\ccN}{\mathscr{N}}

\newcommand{\ccP}{\mathscr{P}}

\newcommand{\ccX}{\mathscr{X}}

\newcommand{\gb}{\mathfrak{b}}
\newcommand{\gv}{\mathfrak{v}}
\newcommand{\olS}{\overline{S}}
\newcommand{\olE}{\overline{E}}
\newcommand{\ole}{\overline{e}}

\theoremstyle{plain}
\newtheorem{thm}{Theorem}[section]

\newtheorem{prop}[thm]{Proposition}
\newtheorem{claim}[thm]{Claim}
\newtheorem{fact}[thm]{Fact}
\newtheorem{cor}[thm]{Corollary}
\newtheorem{lemma}[thm]{Lemma}

\theoremstyle{definition}

\newtheorem{dfn}[thm]{Definition}
\newtheorem{exmp}[thm]{Example}
\newtheorem{conj}[thm]{Conjecture}
\newtheorem{prob}[thm]{Problem}
\newtheorem{quest}[thm]{Question}

\usepackage{accents}

\let\lra=\longrightarrow
\let\phi=\varphi

\DeclareSymbolFont{stmry}{U}{stmry}{m}{n}
\DeclareMathSymbol\arrownot\mathrel{stmry}{"58}
\DeclareMathSymbol\Arrownot\mathrel{stmry}{"59}
\def\longarrownot{\mathrel{\mkern5.5mu\arrownot\mkern-5.5mu}}

\def\nlra{\longarrownot\longrightarrow}

\def\Gth{\mathrm{Girth}}
\def\NH{\mathrm{NH}}
\def\Ext{\mathrm{Ext}}
\def\GTH{\mathfrak{Girth}}

\begin{document}
\title[Graphs of large girth]{Graphs of large girth}
\author[Christian Reiher]{Christian Reiher}
\address{Fachbereich Mathematik, Universit\"at Hamburg, Hamburg, Germany}
\email{christian.reiher@uni-hamburg.de }

\dedicatory{Dedicated to founding editor Jaroslav Ne\v{s}et\v{r}il}

\begin{abstract}
	This survey on graphs of large girth consists of two parts. The first 
	deals with some aspects of algebraic and extremal graph theory loosely 
	related to the Moore bound. Our point of departure for the second, Ramsey 
	theoretic, part are some constructions of graphs with large chromatic number 
	and large girth; this will lead us to a discussion of the recent girth 
	Ramsey theorem. Both parts can be enjoyed independently of each other.
\end{abstract}
	
\maketitle

\section{Introduction}
Unless something else is explicitly said---which is occasionally going to happen---the 
word `graph' always means `finite, simple, undirected graph'. 
 Tutte~\cite{Tutte47} introduced the concept of girth at the same 
time Jarik was born: The {\it girth} of a graph $G$, denoted  
by $\gth(G)$, is the length of a shortest cycle in $G$. So graphs of 
large girth contain no short cycles and, accordingly, one sets $\gth(G)=\infty$
for acyclic graphs $G$ (also known as forests). 
People coming from various different 
directions have contributed to the study of this graph invariant during the last 
seven decades and an enormous corpus of interesting results has been accumulated. 
Sacrificing breadth for depth, we will only focus on two aspects of this vast 
topic in the sequel. 

First, there are obvious extremal problems motivated by the observation that the 
absence of short cycles makes graphs somewhat `sparse'. Locally, a graph of large 
girth looks like a tree. In fact, local considerations alone 
show that graphs of large minimum degree and large girth need to have quite a lot 
of vertices. Quantitatively this is made more precise by the Moore bound 
(Theorem~\ref{thm:21}). The innocent looking question to what extent this bound is 
sharp will lead us to a plethora of exciting algebraic, geometric, and number 
theoretic constructions (\S\ref{subsec:21} and \S\ref{subsec:cages}). 
As proved by Alon, Hoory, and Linial~\cite{AHL}, the Moore bound generalises to 
irregular graphs. Our discussion of their result draws attention to its connection 
with Sidorenko's conjecture for paths (\S\ref{subsec:23}). For directed graphs the 
problem to bound the girth in terms of minimum degree and the number of vertices has a 
quite different character. In comparison to the undirected setting not much is known in 
this area. However, there are many beautiful and tantalising conjectures, the most 
notable of which is due to Caccetta and H\"{a}ggkvist~\cite{CH78}. Some of these 
problems will be presented in~\S\ref{subsec:CH}.     

Our second topic gives plenty of opportunities to describe several of Jarik's 
results. We begin with Erd\H{o}s's classical theorem on graphs of arbitrarily 
large chromatic number and girth~(\S\ref{subsec:41}). It is well-known 
that Erd\H{o}s provided no examples of such graphs. Jarik's first 
publication~\cite{Ne66} deals with explicit constructions of graphs with large 
chromatic number whose girth is at least~$8$. Together with some other early 
constructions due to Zykov~\cite{Zykov} and Tutte~\cite{UD54} his work is described 
in~\S\ref{subsec:32}. Throughout his life, Jarik frequently returned to the area
of explicit Ramsey theoretic constructions. As he writes himself in the partially 
autobiographic article~\cite{Ne09}, 
\begin{quotation}
	``Mathematically (and otherwise) the most important thing I did in seventies 
	and eighties was Ramsey theory and my collaboration with Vojt\v{e}ch R\"odl.''
\end{quotation}  

In those days, the two young men authored more than forty joint articles. 
Their perhaps most important innovation was the discovery of the partite 
construction method~\cite{NR81}. Until today it remains the by far most 
powerful and flexible construction principle in structural Ramsey theory 
known to mankind. 

We only had the pleasure to collaborate with 
Jarik once~\cite{BNRR}, but this work led to an important insight on partite 
constructions, which later helped us in the proof of the 
girth Ramsey theorem~\cite{girth}. Here we introduce the partite construction 
method in a very simple context, that is far remote from its true potential: 
the existence of hypergraphs with large chromatic number and large 
girth~(\S\ref{subsec:33}). The remainder of Section~\ref{sec:3}
contains some related problems and results that we found interesting for various 
reasons. This includes a discussion of Erd\H{o}s' conjecture that graphs of huge 
chromatic number have subgraphs of large girth and chromatic 
number~(\S\ref{sec:EC}). In~\S\ref{subsec:35} we look at the following Ramsey 
theoretic generalisation of girth and chromatic number: What can be said about the 
local structure of graphs $H$ such that for every $r$-colouring of $V(H)$ there 
is a monochromatic induced copy of a given graph~$F$? 
Proceeding with an infinitary topic we shall then talk about 
finite substructures, which need to appear
in graphs and hypergraphs of uncountable chromatic 
number (\S\ref{subsec:36}, \S\ref{subsec:37}).

The next and last section is devoted to edge colourings. Mostly we attempt to 
provide some context to the following recent result from~\cite{girth}, the proof 
of which depends heavily on Jarik's work alluded to in the above quote. 

\begin{thm}\label{thm:grt}
	For every graph $F$ that is not a forest and every number of colours $r$ 
	there exists a graph $H$ of the same girth as $F$ such that for every 
	$r$-colouring of $H$ there is a monochromatic induced copy of $F$.
\end{thm} 

Without the girth requirement this statement, known as the {\it induced Ramsey 
theorem for graphs}, predates the collaboration of Jarik and R\"odl.  
Nowadays its most transparent and generalisable proofs are based  
on the partite construction method (\S\ref{subsec:61}). 

We shall then devote 
some pages to the implicit question whether proving Theorem~\ref{thm:grt} 
with the girth constraint is worth a lot of effort. Our point of view is that 
the real question is to determine the local structure of Ramsey graphs. For 
instance, given two graphs~$F$ and~$G$ we would like to know whether for every 
sufficiently large 
number of colours~$r$ every Ramsey graph~$H$ of~$F$ needs to contain a copy 
of~$G$ (cf.\ Theorem~\ref{thm:19}). E.g., if~$G=C_n$ for 
some $n\in [3, \gth(F)-1]$, then the girth Ramsey theorem provides a negative answer. 
At present nobody knows whether Theorem~\ref{thm:grt} can be proved without 
answering such more general questions along the way. Due to space limitations 
we cannot give a meaningful description of the proof strategy involved here. 
Nevertheless, we use the occasion for outlining some of Jarik's joint ideas 
with R\"{o}dl (\S\ref{subsec:63}).     
Finally, we conclude with some speculations on the possibility of a transfinite girth 
Ramsey theory (\S\ref{subsec:64}).

\subsection*{Notation and terminology}
For every graph $G$ we denote by $\delta(G)$, $\Delta(G)$, $d(G)$, and $e(G)$
its minimum degree, maximum degree, average degree, and the number of its edges. 
Given a set $X$ and a nonnegative integer $k$ we write $X^{(k)}$ for the set of 
all $k$-element subsets of $X$, i.e., $X^{(k)}=\{e\subseteq X\colon |e|=k\}$.
A $k$-uniform hypergraph is a pair $H=(V, E)$ consisting of a set~$V$ 
of {\it vertices} and a set $E\subseteq V^{(k)}$ of {\it edges}. Unless the context 
suggests something to the contrary, our hypergraphs will tacitly be assumed to be 
finite. Notice that graphs are the same as $2$-uniform hypergraphs. 

Mathematicians will be referred to by their surnames. An exception is made 
for {\sc Jaroslav Ne\v{s}et\v{r}il}, in honor of whom these pages are written: 
he will respectfully be called `Jarik'. 

Being a survey, this article contains no new results, but sometimes we give `proofs'
of old results, especially when they convey instructive ideas typical for the flavour 
of some subject. Often these `proofs' are in reality only `sketches of proofs' 
or `main ideas of proofs', but we made no attempt to draw a line between `full proofs' 
and `sketches'. In each case, a reference to the literature is provided.  
When a statement is immediately followed  by the end-of-proof symbol `$\Box$',
it means that the result is either trivial or so deep that we made no effort 
to describe its proof. 
 
\section{Girth, degrees, and the number of vertices}

\subsection{Moore graphs}\label{subsec:21}
In most texts covering extremal graph theory, the first result containing the word 
`girth' provides a lower bound on the number of vertices that a graph can 
have when its minimum degree and girth are given. This estimate, often called the 
Moore bound, involves the function $n_0(d, g)$ defined for every real $d\ge 1$ and every integer $g\ge 3$ by
\[
	n_0(d, g)=
	\begin{cases}
		1+d\sum_{i=0}^{h-1} (d-1)^i   & \text{ if $g=2h+1$ is odd} \cr
		2\sum_{i=0}^{h-1} (d-1)^i & \text{ if $g=2h$ is even}.
	\end{cases}
\]

\begin{thm}[Moore bound]\label{thm:21}
	Every graph $G$ with $\delta(G)\ge d\ge 1$ and $\gth(G)\ge g\ge 3$ has at least
	$n_0(d, g)$ vertices. 
\end{thm}

\begin{proof}
	Suppose first that $g=2h+1$ is odd. Fix an arbitrary vertex $x$ of $G$. For each 
	integer~$i\ge 0$ let $D_i$ be the set of all vertices of $G$ having the distance~$i$
	from $x$ (see Figure~\ref{fig21A}). 
	So $D_0=\{x\}$, $D_1$ is the neighbourhood of $x$, and so on. Clearly 
	$D_0, \dots, D_h$ are mutually disjoint sets, and the main point is that, 
	with the possible exception of $D_h$, all these sets are independent. 
	This is because otherwise 
	we could build an odd cycle whose length would be at most~$2h-1$. Using the 
	assumption $\delta(G)\ge d$ it is now straightforward 
	to show $|D_i|\ge d(d-1)^{i-1}$ for every positive $i\le h$, whence 
		\[
		|V(G)|
		\ge 
		\sum_{i=0}^h |D_i|\ge 1+d\sum_{i=1}^h (d-1)^{i-1} 
		=
		n_0(d, g)\,.
	\]
		
	The case that $g=2h$ is even can be treated similarly, starting with an arbitrary 
	edge $xy$ of~$G$ as opposed to a single vertex (see Figure~\ref{fig21B}).   
\end{proof}

	\begin{figure}[h!]
		\begin{subfigure}[b]{0.47\textwidth}
			\centering
			\begin{tikzpicture}
		\coordinate (x) at (0,0);
		\coordinate (a) at (-1, 1);
		\coordinate (b) at (0,1);
		\coordinate (c) at (1,1);
		\foreach \i in {1,...,6} \coordinate (\i) at (-2.45+.7*\i,2);
		\foreach \i in {1,2,3,4,5,6,x,a,b,c} \fill (\i) circle (1pt);
		\draw (x) circle (4pt);
		\draw (x)--(a)--(1)--(a)--(2);
		\draw (x)--(b)--(3)--(b)--(4);
		\draw (x)--(c)--(5)--(c)--(6);	
		\draw (1) edge [bend left = 50] (3);
		\draw (2) edge [bend left = 50] (5);
		\draw (4) edge [bend left = 50] (6);
		\draw (2) edge [bend left = 40] (4);
		\draw (3) edge [bend left = 40] (5);
		\draw (1) edge [bend left = 50] (6);
		\def \w{.15};  		\draw (-1.05,1-\w) edge [out=180, in=270] (-1.05-\w,1);
		\draw (-1.05-\w,1) edge [out=90, in=180] (-1.05, 1+\w);
		\draw (1.05,1-\w) edge [out=0, in=270] (1.05+\w,1);
		\draw (1.05+\w,1) edge [out=90, in=0] (1.05,1+\w);
		\draw (-1.05,1-\w)--(1.05,1-\w);
		\draw (-1.05,1+\w)--(1.05,1+\w);
		\draw (-1.8,2-\w) edge [out=180, in=270] (-1.8-\w,2);
		\draw (-1.8-\w,2) edge [out=90, in=180] (-1.8, 2+\w);
		\draw (1.8,2-\w) edge [out=0, in=270] (1.8+\w,2);
		\draw (1.8+\w,2) edge [out=90, in=0] (1.8,2+\w);
		\draw (-1.8,2-\w)--(1.8,2-\w);
		\draw (-1.8,2+\w)--(1.8,2+\w);
		\draw [thick, dashed] (-2.75,2)--(-1.95,2);
		\draw [thick, dashed] (-2.75,1)--(-1.22,1);
		\draw [thick, dashed] (-2.75,0)--(-.17,0);
		\node at (0,-.3) {$x$};
		\node at (-3, 0) {$D_0$};
		\node at (-3, 1) {$D_1$};
		\node at (-3, 2) {$D_2$};
	\end{tikzpicture}
\vskip -.1cm
\caption{$d=3$, $g=5$}\label{fig21A}
		\end{subfigure}
	\hfill
	\begin{subfigure}[b]{0.47\textwidth}
		\centering
		\begin{tikzpicture}
			\coordinate (x) at (-.8,0);
			\coordinate (y) at (.8,0);
			\coordinate (a) at (-1.2, 1);
			\coordinate (b) at (-.4,1);
			\coordinate (c) at (.4,1);
			\coordinate (d) at (1.2,1);
			\foreach \i in {1,...,8} \coordinate (\i) at (-2.25+.5*\i,2);
			\foreach \i in {1,2,3,4,5,6,7,8,x,y,a,b,c,d} \fill (\i) circle (1pt);
			\draw (1)--(a)--(x)--(b)--(3);
			\draw (4)--(5)--(c)--(y)--(d)--(8);
			\draw (x)--(y);
			\draw (a)--(2);
			\draw (b)--(4);
			\draw (c)--(6);
			\draw (d)--(7);	
			\draw (1) edge [bend left = 40] (5);
			\draw (1) edge [bend left = 50] (7);
			\draw (2) edge [bend left = 50] (6);
			\draw (2) edge [bend left = 50] (8);
			\draw (3) edge [bend left = 40] (6);
			\draw (3) edge [bend left = 50] (7);
			\draw (4) edge [bend left = 50] (8);
		
			\node at (.8,-.3) {$y$};
			\node at (-.8,-.3) {$x$};
		\end{tikzpicture}
		\vskip -.1cm
		\caption{$d=3$, $g=6$}\label{fig21B}
	\end{subfigure}
	\caption{Proof of the Moore bound} \label{fig21}
	\end{figure}
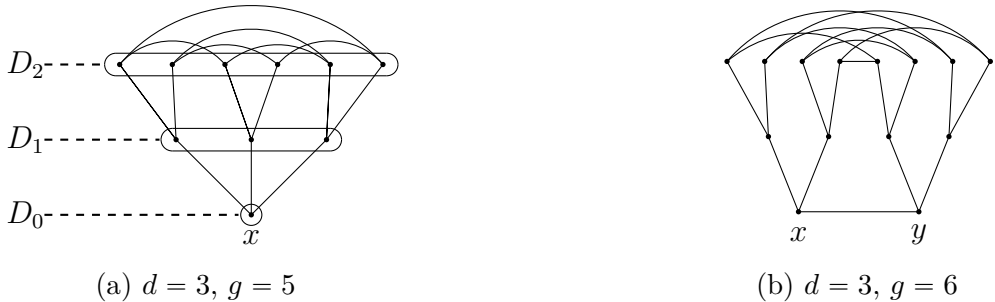
 
Despite the simplicity of its proof, the Moore bound is sharp for a surprisingly 
complex family of parameters, which is still not completely understood. Let us say that 
a graph $G$ is a {\it $(d, g)$-Moore graph} if $\delta(G)\ge d$, $\gth(G)\ge g$,
and $|V(G)|=n_0(d, g)$. It follows immediately from the above proof that any such 
graph must be $d$-regular and connected.  
 
Some small cases are quickly discussed. For instance, a $(d, 3)$-Moore graph is just 
a $d$-regular graph on $n_0(d, 3)=d+1$ vertices, so $G=K_{d+1}$ is the only example 
for $g=3$. Next, we have $n_0(d, 4)=2d$ and the only $d$-regular, triangle-free
graph on $2d$ vertices is the balanced, complete, bipartite graph~$K_{d, d}$ 
(e.g., by Mantel's theorem~\cite{Ma07}). Thus $K_{d, d}$ is the unique $(d, 4)$-Moore 
graph. 

The first nontrivial case is $g=5$. Note that $n_0(d, 5)=d^2+1$ and that, again 
by the proof of Theorem~\ref{thm:21}, two distinct vertices of a $(d, 5)$-Moore 
graph have a common neighbour if and only if 
they are non-adjacent. For $d=1, 2, 3$ the only such graphs can easily be seen 
to be the edge $K_2$, the pentagon $C_5$, and the so-called {\it Petersen graph} 
(see Figure~\ref{fig22}).

	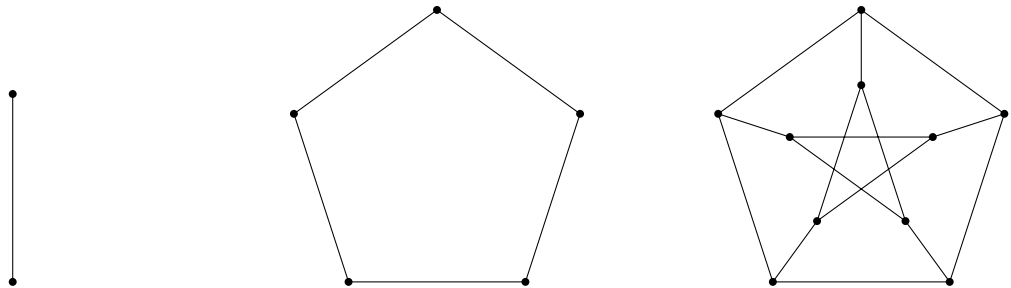
\begin{figure}[h!]
		\begin{subfigure}[b]{0.32\textwidth}
			\centering
			\begin{tikzpicture}
		\coordinate (a) at (0,0);
		\coordinate (b) at (0,2.5);  
		\draw (a)--(b);
		\foreach \i in {a,b}{\fill (\i) circle (1.5pt);}  	\end{tikzpicture}
		\end{subfigure}
	\hfill
	\begin{subfigure}[b]{0.32\textwidth}
		\centering
		\begin{tikzpicture}
			\foreach \i in {1,...,5}{
			\coordinate (\i) at (18+72*\i:2); 			\fill (\i) circle (1.5pt);
		}
		\draw (1)--(2)--(3)--(4)--(5)--cycle;
		\end{tikzpicture}
	\end{subfigure}
\hfill
\begin{subfigure}[b]{0.32\textwidth}
	\centering
	\begin{tikzpicture}
		\foreach \i in {1,...,5}{
			\coordinate (b\i) at (18+72*\i:1);
			\coordinate (a\i) at (18+72*\i:2);
			\draw (a\i)--(b\i);
			\fill (a\i) circle (1.5pt);
			\fill (b\i) circle (1.5pt);
		}
		\draw (a1)--(a2)--(a3)--(a4)--(a5)--cycle;
		\draw (b1)--(b3)--(b5)--(b2)--(b4)--cycle;
	\end{tikzpicture}
\end{subfigure}
	\caption{Edge, Pentagon, and Petersen graph}\label{fig22}
	\end{figure}
 
Hoffman and Singleton~\cite{HS60} constructed another such graph 
for $d=7$, and the same authors also established the following surprising result.     

\begin{thm}[Hoffman \& Singleton]
	If a $d$-regular graph $G$ on $d^2+1$ vertices satisfying $\gth(G)\ge 5$
	exists, then $d\in\{1, 2, 3, 7, 57\}$.
\end{thm}

\begin{proof}
		Set $n=d^2+1$ and consider any $d$-regular graph $G$ with vertex set $[n]$ 
	and $\gth(G)\ge 5$.
	Let $A\in \RR^{n\times n}$ be the adjacency matrix of $G$.
	Since the $(i, j)$-entry of $A^2$ is just the number of vertices $k$ such that 
	$ik, jk\in E(G)$, we have 
		\begin{equation}\label{eq:1616}
		A^2 + A = (d-1) \cdot I + J\,,
	\end{equation}
		where $I$ is the identity matrix of rank $n$ and~$J$ denotes 
	the $(n\times n)$-matrix all of whose entries are equal to $1$.
	Since $G$ is $d$-regular and connected,~$d$ is an eigenvalue of $A$ with 
	multiplicity~$1$, and 
	the corresponding eigenspace is spanned by the vector $\gb=(1, \dots, 1)^\top$. 
	Now let $\gv$ be an arbitrary further eigenvector 
	of $A$, say with eigenvalue~$\lambda$. Multiplying~\eqref{eq:1616}  
	with $\gv$ we obtain $\bigl(\lambda^2+\lambda-(d-1)\bigr)\gv=(\gb\gv)\gb$, 
	which entails $\lambda^2+\lambda-(d-1)=0$, because $\gb$ and $\gv$ are linearly 
	independent. Consequently, the eigenvalues 
	of $A$ other than $d$ are among $\lambda_{\pm}=(-1\pm\sqrt{4d-3})/2$.    
	Now let~$m_\pm$ denote the multiplicities of these eigenvalues. 
	Since $A$ has $n$ eigenvalues summing up to the trace of $A$, we obtain the 
	system of equations 
		\begin{align*}
		m_++m_-+1 &= d^2 + 1 \\ 
		\lambda_+m_++\lambda_-m_-+d &=0\,,
	\end{align*}
		which leads to $(m_+-m_-)\sqrt{4d - 3}=d^2-2d$. Unless $d=2$ this is 
	only possible if $4d-3$ is a perfect square, i.e., if there is an odd integer $s$
	such that $d=(s^2+3)/4$. In this case $s=\sqrt{4d-3}$ needs to divide 
	$16(d^2-2d)=s^4-2s^2-15$, whence $s\in\{1, 3, 5, 15\}$, i.e., 
	$d\in\{1, 3, 7, 57\}$.
\end{proof}

This result leaves the following major problem open.

\begin{quest}\label{qu:moore}
	Does there exist a $57$-regular graph $G$ on $3250$ vertices with $\gth(G)\ge 5$?
\end{quest}	

Such graphs are called `missing Moore graphs' in the literature. They have 
been studied intensively using a variety of combinatorial, spectral, 
and computational approaches. Moreover, starting with the work of 
Aschbach~\cite{Asch71}, group theoretic and representation theoretic methods 
have been employed as well. Special attention has been given to the possible 
automorphism groups of missing 
Moore graphs. Higman showed that such graphs cannot be vertex-transitive 
(see also~\cite{Cameron}); much more recently, Ma\v{c}aj
and \v{S}ir\'{a}\v{n}~\cite{MS10} improved this to $|\mathrm{Aut}(G)|\le 375$
for every missing Moore graph $G$. More information on this topic is contained 
in Dalf\'{o}'s survey~\cite{Dalfo}.

Why is Question~\ref{qu:moore} so difficult?
The most likely explanation might be that there are something like one billion 
non-isomorphic missing Moore graphs, all with very small automorphism groups. This 
would mean that there are so few of them that it is practically impossible to find 
any by a lucky guess or by an exhaustive search;
but, at the same time, there are so many of them, or the constraints of being 
$57$-regular and having girth $5$ are so `weak', that the search tree cannot be 
narrowed down substantially. With respect to some other very difficult combinatorial 
problems, a similar sentiment has recently been expressed more eloquently 
by Gowers~\cite{Gowers}. 
In the case of missing Moore graphs, it certainly does not help either that $3250$ 
vertices are, on the one hand, so few that contemporary methods of extremal 
and probabilistic graph theory become mute; but, on the other hand, more than three thousand vertices are so many that it is hard to deal with them in a concrete and 
explicit way. 

Before we proceed to larger girth, we quickly want to eliminate some small values 
of $d$. Due to $n_0(1, g)=2$ the edge $K_2$ can be viewed as a $(1, g)$-Moore graph for 
every $g\ge 3$. Only slightly more interestingly, we have $n_0(2, g)=g$ and thus the 
cycle $C_g$ is the only $(2, g)$-Moore graph. Henceforth we will always restrict 
our attention to the case $d\ge 3$.  

Even values of $g$ were studied in the PhD thesis of 
Singleton~\cites{Sing62, Sing66}, who made the astonishing discovery 
that here  $(d, g)$-Moore graphs can only exist if $g\in \{6, 8, 12\}$. 
At about the same time an equivalent algebraic result was obtained by Feit 
and Higman~\cite{FH}.
The odd case was solved independently by Damerell~\cite{Dam} and in joint work
of Bannai and Ito~\cite{BI73}. It turned out that for odd $g\ge 7$ there are no 
further Moore graphs, so that altogether the following result has been established. 
For a somewhat streamlined proof we refer to Biggs' textbook on algebraic
graph theory~\cite{Biggs}*{Theorem 23.6}.  

\begin{thm}
	Let $d\ge 3$ and $g\ge 5$. If there exists a $d$-regular graph $G$ on $n_0(d, g)$
	vertices with $\gth(G)\ge g$, then $g\in \{5, 6, 8, 12\}$. \qed
\end{thm}

In the study of Moore graphs with even girth the following observation is often 
useful. 

\begin{lemma}\label{lem:bip}
	If $g\ge 4$ is even and $d\ge 2$, then every $(d, g)$-Moore graph is bipartite. 
\end{lemma}

\begin{proof}
	Otherwise let $C=v_1\dots v_n$ be a shortest odd cycle in $G$. This cycle needs 
	to be geodetic, i.e., it predicts the distances of its vertices correctly. This 
	is because if two vertices $v_i$, $v_j$ could be connected by a path $P$ that is 
	shorter than both $v_i$-$v_j$-paths in $C$, then $P$ together with one of these 
	paths would create a closed walk of some odd length $n'<n$. But any such closed 
	walk would need to contain an odd cycle that contradicted the minimal choice 
	of~$n$. 
	
	Let us now run the proof of Theorem~\ref{thm:21} with the edge $v_1v_2$ in the 
	distinguished r\^{o}le. The vertex $v_{2+g/2}$ needs to appear somewhere in 
	Figure~\ref{fig21B} and thus its distance from at least one of  $v_1$ or $v_2$ is 
	beneath $g/2$. As $C$ is geodetic, this implies $n\le g$.
	But due to $\gth(G)\ge g$ and the fact that $g$, $n$ have different parities 
	this is absurd.  
\end{proof}

It is now natural to investigate the sets 
\begin{equation}\label{eq:1550}
	A_g
	=
	\{d\ge 2\colon \text{there exists a $(d+1, g)$-Moore graph}\}
\end{equation}
for $g=6, 8, 12$ (the reason why we wrote $d+1$ rather than $d$ will soon become 
apparent). Before summarising the known results on these sets, we briefly digress 
into projective geometry, referring to the two-volume treatise by Veblen and 
Young~\cites{VY1, VY2} for further background. 

Let us recall that a {\it projective plane} is given by a set of {\it points}, 
a set of {\it lines}, and an {\it incidence relation} between points and lines 
such that (i) any two distinct points determine a unique line, (ii) any two distinct 
lines intersect in a unique point, (iii) and there exist four points no three 
of which are collinear. The smallest projective plane is the {\it Fano plane} 
depicted in Figure~\ref{fig23A}.

	\begin{figure}[h!]
		\begin{subfigure}[b]{0.43\textwidth}
			\centering
			\begin{tikzpicture}
				\def\s{2.5}     				\def\h{1.732*\s}
				\def\m{.577*\s} 
		\coordinate (a) at (-\s,0);
		\coordinate (b) at (0,\h);
		\coordinate (c) at (\s,0); 
		\coordinate (d) at ($(a)!.5!(b)$) ; 
		\coordinate (e) at ($(b)!.5!(c)$) ; 
		\coordinate (f) at ($(a)!.5!(c)$) ; 
		\coordinate (s) at (0, \m);
		\draw (a)--(b)--(c)--cycle;
		\draw (a)--(e);
		\draw (b)--(f);
		\draw (c)--(d);
		\draw (s) circle (\m);
		\foreach \i in {a,b,c,d,e,f,s}{\fill (\i) circle (1.5pt);}  
	\end{tikzpicture}
\vskip .1cm  \caption{Fano plane}\label{fig23A}
		\end{subfigure}
	\hfill
	\begin{subfigure}[b]{0.43\textwidth}
		\centering
		\begin{tikzpicture}
			\def\s{1}
			\def\h{0.866*\s}
			
			\foreach \i in {0,2,4}
				\foreach \j in {-1,...,7}
				{	\coordinate (\i\j) at (\j*\s,\i*\h);}
		
				\foreach \i in {-1,1,3,5}
				\foreach \j in {-1,...,7}
					{	\coordinate (\i\j) at (\j*\s-.5*\s,\i*\h);}

			\fill [\ca] (57)--(46)--(37)--(47);
			\fill [\ca] (52)--(42)--(43)--(54)--cycle;
			\fill [\ca] (20)--(11)--(00)--(10)--cycle;
			\fill [\ca] (03)--(14)--(15)--(05)--cycle;
			\fill [\cb] (02)--(03)--(14)--(23)--(22)--(12)--(02);
			\fill [\cb] (55)--(45)--(46)--(57)--(55);
			\fill [\cc] (45)--(46)--(37)--(26)--(25)--(35)--(45);
			\fill [\cc] (31)--(32)--(42)--(52)--(31);
			\fill [\cd] (00)--(11)--(12)--(02)--(00);
			\fill [\cd] (54)--(55)--(45)--(35)--(34)--(43)--(54);
			\fill [\ce] (20)--(31)--(32)--(22)--(12)--(11)--(20);
			\fill [\ce] (05)--(15)--(25)--(26)--(05);
			\fill [\cf] (32)--(42)--(43)--(34)--(23)--(22)--(32);
			\fill [\cg] (34)--(35)--(25)--(15)--(14)--(23)--(34);
			
			\fill [white] ($(53)+(0,.1)$)--(53)--(47)--($(47)+(.1,0)$)--($(57)+(.1,.1)$)--cycle;
			\fill [white] (51)--($(10)+(-.1,0)$)--(10)--(53)--($(53)+(0,.1)$)--cycle;
			\fill [white] (10)--(04)--($(04)+(0,-.1)$)--($(00)+(-.1,-.1)$)--($(10)+(-.1,0)$)--cycle;
			\fill [white] (04) --(47)--($(47)+(.1,0)$)--($(06)+(0,-.1)$)--($(04)+(0,-.4)$)--cycle;
						
					\fill [\cb] (12)--(02)--(03)--(14)--(12);	
						\fill [\ce] (11)--(20)--(31)--(32)--(11);
						\fill [\cd] (54)--(55)--(45)--(35)--(34)--(43)--(54);
							\fill [\cc] (45)--(46)--(37)--(26)--(25)--(35)--(45);
							\fill [\ca] (20)--(2-1)--(1-1)--(0-1)--(00)--(11)--(20);
					
						\fill [white, opacity=0.8] (10)--(04)--($(04)+(0,-.1)$)--($(00)+(-.1,-.1)$)--($(10)+(-.1,0)$)--cycle;
							\fill [white, opacity=.7] ($(53)+(0,.1)$)--(53)--(47)--($(47)+(.1,0)$)--($(57)+(.1,.1)$)--cycle;
							\fill [white, opacity=0.8] (51)--($(10)+(-.1,0)$)--(10)--(53)--($(53)+(0,.1)$)--cycle;
								\fill [white, opacity=0.7] (04) --(47)--($(47)+(.1,0)$)--($(06)+(0,-.1)$)--($(04)+(0,-.4)$)--cycle;
								
								\fill [white, opacity=0.8] ($(10)+(-.1,0)$)--($(40)+(.23,0)$)--(3-1)--(-1-1)--($(00)+(-.1,-.1)$)--($(10)+(-.1,0)$);
								
								\draw [dashed] (10) -- ($(1-1)!.2!(2-1)$);
								\draw [dashed] (10) -- ($(1-1)!.7!(0-1)$);
							 
					\draw [dashed] ($(11)!.24!(00)$)--(00)--(0-1)--(1-1)--(2-1)--(20);

			\draw [shorten <=23pt, shorten >=23pt] (00)--(11)--(12)--(22)--(23)--(14)--(15)--(05);
			\draw (22)--(32)--(42)--(43)--(34)--(23);
			\draw (15)--(25)--(35)--(34);
			\draw [shorten >= 6](35)--(45)--(46)--(37);
			\draw [dashed] ($(11)!.85!(20)$)--(20)--(31)--($(31)!.5!(32)$);
			\draw [dashed] ($(12)!.57!(02)$)--(02)--(03)--($(03)!.2!(14)$);
			\draw [ dashed] ($(25)!.57!(26)$)--(26)--(37)--($(37)!.2!(46)$);
			\draw [dashed] ($(45)!.57!(55)$)--(55)--(54)--($(54)!.2!(43)$);
			\draw [shorten <=23pt] (52)--(42);
			\draw [shorten >= 6pt] (43)--(54);
			\draw [shorten <= 15pt]  (55)--(45);
			\draw  (46)--($(46)!.2!(57)$);
			\draw [shorten <= 14pt] (26)--(25);
			\draw [shorten <= 6pt] (03)--(14);
			\draw [shorten >= 15pt] (12)--(02);
			\draw [shorten <=6pt] (20)--(11);
			\draw [shorten <=15] (31)--(32);
	
		\draw [ultra thick]  (10)--(04)--(47)--(53)--cycle;		
		
		\foreach \i in {11,12,14,15,22,23,25,32,34,35,42,43,45,46} \fill (\i) circle (1.5pt);

		\end{tikzpicture}
	\vskip -.3cm  	\caption{Heawood graph}\label{fig23B}
\end{subfigure}
	\caption{The smallest projective plane and a tiling of the torus (black rhombus whose opposite sides 
	are identified) with seven hexagons}
	\label{fig:23}
	\end{figure}
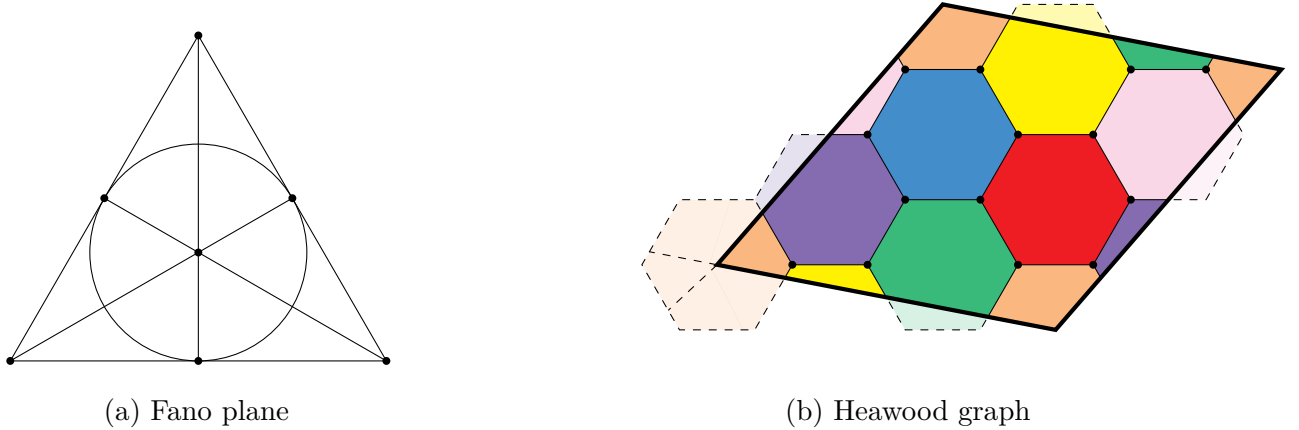

It is well known that for each finite projective plane there 
exists an integer $n$, called its {\it order}, such that every line contains $n+1$ 
points, through every point there pass $n+1$ lines, and the total numbers of points and 
lines are $n^2+n+1$ each. For every finite field~$F$ we can construct a projective 
plane of order $|F|$ whose points and lines are the one- and two-dimensional linear 
subspaces of $F^3$, respectively; the incidence relation of this 
plane is inclusion. Thereby one obtains for every prime power~$n$ a projective 
plane of order~$n$. Some finite projective planes that do not arise from this  
construction have been discovered, but the orders of all of them are still 
prime powers. In fact, the following problem is wide open.

\begin{conj}[Strong prime power conjecture]
	If a projective plane of order $n$ exists, then $n$ is a prime power.
\end{conj} 

Currently it is not even known whether a projective plane of order~$12$ exists
and it would not contradict known results if one counter-conjectured that 
projective planes of order~$n$ exist whenever $n$ is a sufficiently large multiple 
of $4$. 

There is also another construction of projective planes that on first sight might 
seem preferable, as it only requires an additive structure rather than a field 
structure. A {\it perfect difference set of order $n$} is a subset $K$ of the cyclic 
group $\ZZ/(n^2+n+1)\ZZ$ such that $|K|=n+1$ and every nonzero residue class 
modulo $n^2+n+1$ can be expressed
(uniquely) as a difference of two members of $K$. For instance, $\{0, 1, 4, 6\}$
is a perfect difference set of order $3$. From any perfect difference set $K$ of
order $n$ we can construct a projective plane of order $n$ whose points are the 
residue classes modulo $n^2+n+1$ and whose lines are the translates of $K$. 
It has been shown by Singer~\cite{Singer} that for every prime power $n$ there exists 
a perfect difference set of order $n$. 

\begin{conj}[Weak prime power conjecture]
	If a perfect difference set of order $n$ exists, then $n$ is a prime power.
\end{conj} 

In light of the above construction, the strong conjecture implies the weak one. 
However, there is much more computational evidence for the weak conjecture (reaching 
up to $2\cdot 10^{9}$, see~\cite{BG}). 
Peluse~\cite{Peluse} has recently obtained spectacular progress on the weak 
conjecture by proving that for every $N$ the number of all $n\le N$ such that 
a perfect difference set of order~$n$ exists is indeed $(1+o(1))N/\log N$. 
Her profound work combines biquadratic reciprocity, various sieve methods,
and difficult counting techniques for lattice points on hyperboloids.  

The relevance of projective planes to Moore graphs of even girth was apparently 
first understood by K\`arteszi~\cite{Kart}, who obtained one direction of the 
following result that we find in the PhD thesis of Singleton~\cites{Sing62, Sing66} 
(see also Longyear~\cite{Longyear}).

\begin{thm}[Singleton]\label{thm:2014}
	For every $d\ge 2$ there is a bijective correspondence between projective 
	planes of order $d$ and $(6, d+1)$-Moore graphs.  
\end{thm}

\begin{proof}
	Given a projective plane $\varUpsilon$ of order $d$ we construct a 
	bipartite $(d+1)$-regular graph~$B_\varUpsilon$ of girth at least $6$ with 
	$n_0(d+1, 6)=2(d^2+d+1)$ vertices as follows: The two vertex classes 
	of~$B_\varUpsilon$ are the sets of points and lines of $\varUpsilon$;
	edges are determined by incidence, i.e., a point~$p$ is joined to a line $\ell$
	by an edge of $B_\varUpsilon$ if and only if $\ell$ passes through $p$. 
	Notice that the absence of four-cycles in $B_\Upsilon$ follows from the fact 
	that two distinct lines cannot intersect in more than one point. 
		
	Now suppose, conversely, that a $(6, d+1)$-Moore graph is given. 
	Lemma~\ref{lem:bip} tells us that $G$ is bipartite and thus we can obtain 
	an incidence structure with points and lines by reversing the above construction. 
	The first two axioms of a projective plane follow from the fact that~$G$ contains 
	no four-cycles, and the non-degeneracy axiom can be derived from $d\ge 2$.
\end{proof}

As a little fun fact we point out that the $(3, 6)$-Moore graph derived in this 
way from the Fano plane, called the {\it Heawood graph}, corresponds to the 
well-known tiling of a torus with seven mutually touching hexagons (see 
Figure~\ref{fig23B}).
An alternative drawing of this graph is shown in Figure~\ref{fig21B}.
Concerning the set $A_6$ introduced in~\eqref{eq:1550} Theorem~\ref{thm:2014} yields
\[
	A_6=\{\text{orders of finite projective planes}\}\,,
\]
which illustrates the relevance of the strong prime power conjecture to 
algebraic and extremal graph theory.

Continuing with projective geometry, we recall that Veblen and Young~\cites{VY1, VY2}
define a {\it projective space} to be an incidence structure with points and lines 
satisfying the following four axioms: (i) any two distinct points determine a unique 
line; (ii) if $p$, $q$, $r$, $s$ are four distinct points such that the lines 
$pq$, $rs$ are distinct and intersect, then the lines $pr$ and $qs$ intersect as well; 
(iii) every line passes through at least three points; (iv) and there exist two 
non-intersecting lines. Generalising a construction mentioned earlier one can 
define for every field $F$ and every dimension $n\ge 3$ a projective space $P_n(F)$ 
whose points and lines are the one- and two-dimensional linear subspaces of $F^{n+1}$. 
In sharp contrast with the planar case, however, all finite projective spaces can be 
shown to be of this form. 
Roughly speaking this is because the availability of a third dimension allows us
to prove Desargues's theorem (see Figure~\ref{fig24}), which in turn means that 
coordinates from a skew field can be introduced. To conclude the argument one finally 
appeals to a theorem of Wedderburn~\cite{Wedderburn} (see also~\cites{Witt, Ted}), 
which asserts that all finite skew fields are commutative. 

	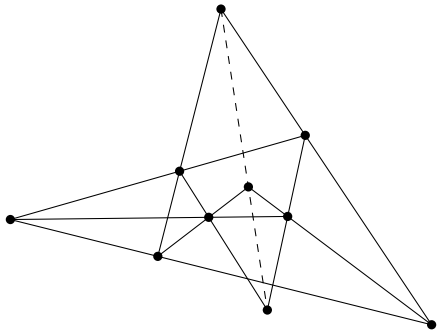
\begin{figure}[h!]
	\begin{tikzpicture}[scale=.7]
		
		\coordinate (a) at (0,0);
		\coordinate (b) at (8,-2);
		\coordinate (c) at (4,4);
		\coordinate (d) at ($(a)!.35!(b)$);
		\coordinate (e) at ($(b)!.6!(c)$);
		\coordinate (f) at ($(a)!.6!(b)$);
		\coordinate (g) at ($(c)!.65!(f)$);
		\coordinate (h) at ($(c)!1.1!(f)$);
		
		 \path[name path=line 1] (a) --(e);
		\path[name path=line 2] (c) -- (d);
		\path [name intersections={of=line 1 and line 2,by=K}];
		
		 \path[name path=line 3] (K) --(h);
		\path[name path=line 4] (g) -- (d);
		\path [name intersections={of=line 3 and line 4,by=L}];
		
		 \path[name path=line 5] (b) --(g);
		\path[name path=line 6] (e) -- (h);
		\path [name intersections={of=line 5 and line 6,by=M}];
		
		\foreach \i in {a,b,c,d,e,g,h,K,L,M} \fill (\i) circle (2.5pt);
					
		\draw (a)--(b)--(c)--(d);
		\draw (a)--(e);
		\draw [dashed] (c)--(h);
		\draw (d)--(g)--(b);
		\draw (e)--(h)--(K);
		\draw (a)--(M);

	\end{tikzpicture}
	\caption{Desargues's theorem states that the three points on the dashed line are collinear, provided that the nine triples on the solid lines are.}
	\label{fig24}
	\end{figure}
 
With respect to three-dimensional projective spaces $P_3(F)$ we need a few more 
concepts. For every nonzero vector $p=(p_1, p_2, p_3, p_4)\in F^4$ we denote the 
subspace of $F^4$ generated by $p$, which is a point of $P_3(F)$, 
by $[p_1, p_2, p_3, p_4]$.
Three-dimensional linear subspaces of $F^4$ are called the {\it planes} of $P_3(F)$.
With the standard scalar product in mind, we can represent planes in the form 
\[
	[p_1, p_2, p_3, p_4]^\perp
	=
	\Bigl\{[x_1, x_2, x_3, x_4]\in P_3(F)\colon \sum_{i=1}^4 p_ix_i=0\Bigr\}\,.
\]

A {\it polarity} of $P_3(F)$ is a bijective map $\pi$ from the points to the planes 
that reverses the incidence relation. That is, for any two points $p$, $q$ it is 
demanded that $p\subseteq \pi(q)$ holds if and only if $q\subseteq \pi(p)$. 
For instance, the map $p\longmapsto p^\perp$ is a polarity. A {\it null polarity}
is a polarity~$\pi$ with the additional property that $p\subseteq \pi(p)$ holds for 
every point $p$. This happens, for example, for the `symplectic' polarity 
$[p_1, p_2, p_3, p_4]\longmapsto [p_2, -p_1, p_4, -p_3]^\perp$.
We proceed with a result that is, again, from Singleton's PhD 
thesis~\cites{Sing62, Sing66}. The statement becomes more transparent when
we present $(d+1, 8)$-Moore graphs as bipartite graphs $(P, L, E)$ with 
vertex partition $P\dcup L$ and $E\subseteq P\times L$. By Lemma~\ref{lem:bip}
this causes no loss of generality. 

\begin{thm}[Singleton]\label{thm:2015}
	For every $d\ge 2$ there is a bijective correspondence between 
	$(d+1, 8)$-Moore graphs $(P, L, E)$ and pairs $(\Sigma, \nu)$ 
	consisting of a $3$-dimensional projective space of order~$d$ 
	and a null polarity $\nu$ of $\Sigma$.  
\end{thm}
     
\begin{proof}
	Suppose first that some $(d+1, 8)$-Moore graph $G=(P, L, E)$ is given. 
	For every point $p\in P$ we call $\nu(p)=\{p'\in P\colon d(p, p')\le 2\}$
	its {\it polar plane}. By a {\it line} we mean an intersection of two planes.
	It can be shown that the points and lines form a 
	$3$-dimensional projective space $\Sigma$ of order $d+1$ and that the 
	map $p\longmapsto \nu(p)$ is a null polarity of $\Sigma$. 
		
	In the converse direction, let $\nu$ be a null polarity of a $3$-dimensional 
	space $\Sigma$ of order $d+1$. Call a line $\ell$ {\it special} if for every 
	point $p$ on $\ell$ the plane $\nu(p)$ contains $\ell$. The incidence graph 
	between the points of $\Sigma$ and the special lines is 
	the desired $(d+1, 8)$-Moore graph. 
\end{proof}

Thus we are in the curious situation that while nobody can decide whether $12\in A_6$
is true or not, the set $A_8$ has been described explicitly as 
\[ 
	A_8=\{\text{prime powers}\}\,.
\]

The available results on $A_{12}$ are by far less complete. Benson~\cite{Benson} proved 
that $A_{12}$ contains all prime powers, but so far no analogue of Theorem~\ref{thm:2014} and Theorem~\ref{thm:2015} is known.

\begin{prob}
	Classify $(d+1, 12)$-Moore graphs in terms of projective geometry.
\end{prob}

It is also unknown whether Benson's result $A_{12}\supseteq\{\text{prime powers}\}$
holds with equality. 

\begin{quest}
	Does there exist a $(d+1, 12)$-Moore graph such that $d$ is not a prime power?
\end{quest}

Inspired by Peluse's asymptotic prime power theorem one can
also ask whether the number of all $d\le N$ such that some $(d+1, 12)$-Moore 
graph exists is $(1+o(1))N/\log N$.

\subsection{Cages and upper bounds}\label{subsec:cages}
Having thus seen that there are many pairs $(d, g)$ for which the Moore bound 
fails to be sharp, one may wish to study the following objects.

\begin{dfn}
	Given two integers $d\ge 2$ and $g\ge 3$ a {\it $(d, g)$-cage} is a $d$-regular 
	graph~$G$ with $\gth(G)\ge g$ which has as few vertices as possible. We shall 
	write $f(d, g)$ for this minimal number of vertices.   
\end{dfn}

A dynamic survey on cages is maintained by Exoo and Jajcay~\cite{EJS}.
The existence of cages, that is the fact that for every $d\ge 2$ there are
$d$-regular graphs of arbitrarily large girth, was first established by 
Sachs~\cite{Sachs}, who then informed Erd\H{o}s that the upper bound 
on~$f(d, g)$ his argument would yield seemed very weak to him\footnote[1]{It should
be pointed out, however, that the focus of~\cite{Sachs} is not so much on bounding 
the function $f(d, g)$ efficiently, but rather on constructing $d$-regular graphs
of large girth with additional structural properties, such as Hamiltonicity and 
the existence of certain kinds of factorisations.}. 
In subsequent joint work of Erd\H{o}s and Sachs~\cite{ES63} the following bound 
was produced. 

\begin{thm}[Erd\H{o}s \& Sachs]
	If $d\ge 2$ and $g\ge 3$, then $f(d, g)\le 4\sum_{i=0}^{g-2}(d-1)^i$.
\end{thm}

\begin{proof}
	Fix $g$ and put $h(d)=\sum_{i=0}^{g-2}(d-1)^i$ for every $d\ge 2$. 
	We want to show the following statement by induction on $d$.
	\begin{quotation}
		For every even $n\ge 4h(d)$ there is a $d$-regular graph on $n$ vertices 
		whose girth is at least $g$.
	\end{quotation}
	
	In the base case, $d=2$, this is exemplified by the even cycle $C_n$, 
	because $4h(2)=4(g-1)\ge g$. Now suppose $d\ge 3$, that $n\ge 4h(d)$ is even, 
	and that the above statement holds for $d-1$ in place of $d$. Consider the 
	class $\ccA$ of all $n$-vertex graphs $G$ such that 
		\begin{enumerate}[label=\nlabel]
		\item\label{it:231} all vertices of $G$ have degree $d-1$ or $d$;
		\item\label{it:232} and $\gth(G)\ge g$.
	\end{enumerate}
		
	The induction hypothesis implies $\ccA\ne\varnothing$. Thus we can pick a 
	graph $G\in \ccA$ with the maximal number of edges. If $G$ is $d$-regular we are 
	done, so assume from now on that this is not the case. For parity reasons, this 
	implies that $G$ has two distinct vertices $x$, $y$ of degree~$d-1$. As in the 
	proof of the Moore bound at most $h(d)$ vertices  
	have distance at most~$g-2$ from $x$, and the same holds for $y$, too. Thus the 
	set $Z$ of all vertices that have distance at least $g-1$ from both $x$, $y$
	satisfies $|Z|\ge n-2h(d)\ge n/2$ (see Figure~\ref{fig25a}). 
	Each vertex~${z\in Z}$ has degree~$d$, since 
	otherwise we could simply add the edge~$xz$ without creating a cycle 
	violating~\ref{it:232}, contrary to the maximality of~$e(G)$. 
	
	By counting the edges between $Z$ and the rest of $G$ we see that $Z$ cannot be 
	independent. Let $x'y'$ be an arbitrary edge connecting two vertices in $Z$. 
	The graph $G'$ obtained from $G$ by adding the edges $xx'$, $yy'$ and 
	deleting~$x'y'$ can be shown 
	to contradict the maximality of~$e(G)$ (see Figure~\ref{fig25b}). 
\end{proof}

	\begin{figure}[h!]
		\begin{subfigure}[b]{.45\textwidth}
			\centering
			\begin{tikzpicture}
				\def\ha{.45}
				\def\hb{.4}
				\def\hc{.3}
				\def\da{.7}
				\def\db{.35}
				\def\dc{.19}
				
				\draw (-2.5,-.7) rectangle (2.5,.3);
				\coordinate (x1) at (-1.05,-.2);
				\coordinate (y1) at (1.05,-.2);
				\coordinate (x) at (-1.05,-2.2);
				\coordinate (y) at (1.05,-2.2);
				\coordinate (a1) at ($(x)+(-\da,\ha)$);
					\coordinate (a2) at ($(x)+(\da,\ha)$);
						\coordinate (a3) at ($(y)+(-\da,\ha)$);
							\coordinate (a4) at ($(y)+(\da,\ha)$);
				\coordinate (b1) at ($(a1)+(-\db,\hb)$);
					\coordinate (b2) at ($(a1)+(\db,\hb)$);
						\coordinate (b3) at ($(a2)+(-\db,\hb)$);
							\coordinate (b4) at ($(a2)+(\db,\hb)$);
							\coordinate (b5) at ($(a3)+(-\db,\hb)$);
							\coordinate (b6) at ($(a3)+(\db,\hb)$);
							\coordinate (b7) at ($(a4)+(-\db,\hb)$);
							\coordinate (b8) at ($(a4)+(\db,\hb)$);
							
							\coordinate (c1) at ($(b1)+(-\dc,\hc)$);
							\coordinate (c2) at ($(b1)+(\dc,\hc)$);
							\coordinate (c3) at ($(b2)+(-\dc,\hc)$);
							\coordinate (c4) at ($(b2)+(\dc,\hc)$);
							\coordinate (c5) at ($(b3)+(-\dc,\hc)$);
							\coordinate (c6) at ($(b3)+(\dc,\hc)$);
							\coordinate (c7) at ($(b4)+(-\dc,\hc)$);
							\coordinate (c8) at ($(b4)+(\dc,\hc)$);
							\coordinate (c9) at ($(b5)+(-\dc,\hc)$);
							\coordinate (c10) at ($(b5)+(\dc,\hc)$);
							\coordinate (c11) at ($(b6)+(-\dc,\hc)$);
							\coordinate (c12) at ($(b6)+(\dc,\hc)$);
							\coordinate (c13) at ($(b7)+(-\dc,\hc)$);
							\coordinate (c14) at ($(b7)+(\dc,\hc)$);
							\coordinate (c15) at ($(b8)+(-\dc,\hc)$);
							\coordinate (c16) at ($(b8)+(\dc,\hc)$);
				
				\draw (x1)-- (y1);
				
				\draw (x)--(a1)--(b1)--(c1)--(b1)--(c2);
				\draw (a1) --(b2)--(c3)--(b2)--(c4);
				\draw (x)--(a2)--(b3)--(c5)--(b3)--(c6);
				\draw (a2) --(b4)--(c7)--(b4)--(c8);
				
				\draw (y)--(a3)--(b5)--(c9)--(b5)--(c10);
				\draw (a3) --(b6)--(c11)--(b6)--(c12);
				\draw (y)--(a4)--(b7)--(c13)--(b7)--(c14);
				\draw (a4) --(b8)--(c15)--(b8)--(c16);

		\draw [rounded corners=10] (-1.05,-2.4)--(-1.7,-2)--(-2.55,-.95)--(.45,-.95)--(-.4,-2)--cycle;
		\draw [rounded corners=10] (1.05,-2.4)--(1.7,-2)--(2.55,-.97)--(-.45,-.97)--(.4,-2)--cycle;

				\foreach \i in {x,y,x1,y1} \fill (\i) circle (1.5pt);
				\node [left] at (x1) {$x'$};
				\node [right] at (y1) {$y'$};
				\node [below] at (x) {$x$};
				\node [below] at (y) {$y$};
				\node at (-2.9,-.2) {\large $Z$};
				
				\foreach \i in {1,...,16} \fill (c\i) circle (1pt);
					\foreach \i in {1,...,8} \fill (b\i) circle (1pt);
						\foreach \i in {1,...,4} \fill (a\i) circle (1pt);
				
			\end{tikzpicture}
		\caption{The set $Z$}
		\label{fig25a}
		\end{subfigure}
	\hfill
	\begin{subfigure}[b]{.45\textwidth}
		\centering
		\begin{tikzpicture}
				\draw (-2.5,-.7) rectangle (2.5,.3);
			\draw (x1)-- (x);
			\draw (y1)-- (y);
			\foreach \i in {x,y,x1,y1} \fill (\i) circle (1.5pt);
			
							\draw [rounded corners=10] (1.05,-2.4)--(1.7,-2)--(2.55,-.98)--(-.45,-.98)--(.4,-2)--cycle;
					\draw [rounded corners=10] (-1.05,-2.4)--(-1.7,-2)--(-2.55,-.95)--(.45,-.95)--(-.4,-2)--cycle;
			
			\node [left] at (x1) {$x'$};
			\node [right] at (y1) {$y'$};
			\node [below] at (x) {$x$};
			\node [below] at (y) {$y$};

		\end{tikzpicture}
	\caption{The graph $G'$}
	\label{fig25b}
	\end{subfigure}
	\caption{Proof of the Erd\H os-Sachs theorem.}
	\label{fig:25}
	\end{figure}
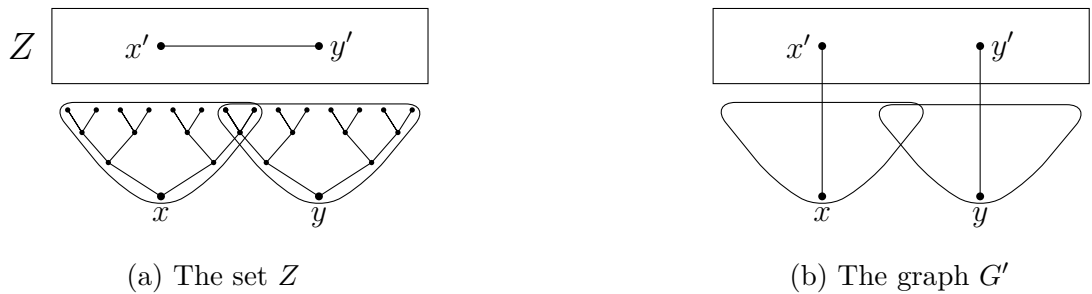
 
In the regime $d, g\to\infty$ the Moore bound and the Erd\H{o}s-Sachs theorem
yield the asymptotic relations
\[
	d^{(1/2+o(1))g} 
	\le
	f(d, g)
	\le 
	d^{(1+o(1))g}\,.
\]

The upper bound cannot be improved by a straightforward probabilistic attempt.
However, Lubotzky, Phillips, and Sarnak~\cite{LPS88} discovered an explicit number 
theoretic construction leading to the 
superior bound 
\begin{equation}\label{eq:34}
	f(d, g)
	\le 
	d^{(3/4+o(1))g}\,.
\end{equation}

Before describing their graphs we agree on some notation and terminology. 
We say that a subset $S$ of a (finite or infinite) group $\Gamma$ is {\it symmetric}
if $S^{-1}=S$, i.e., if $S$ is closed under taking inverses. When we have this 
situation and $1\not\in S$, then the {\it Cayley graph} $\Cayley(\Gamma, S)$ 
is defined to be the graph on $\Gamma$ with all edges of the form $\{g, gs\}$, 
where $g\in G$ and $s\in S$. 
Roughly speaking, the girth of this graph is large if the members of $S$ satisfy 
no `short' nontrivial relation. For instance, if $S$ contains two distinct elements 
$a$ and $b$ which commute but are not inverse to each other, then through every 
vertex $x$ there passes a four-cycle $x-xa-xab-xb$. At the other extreme, if $\Gamma$
is freely generated by a set $T$, then $\Cayley(\Gamma, T\cup T^{-1})$ is a tree 
all of whose vertices have degree $2|T|$. The graphs of Lubotzky, Phillips, and 
Sarnak can be viewed as `finite quotients' of this example. 

Concerning their underlying groups, we recall that for every field $F$ 
the {\it general linear group}~$\GL(2, F)$ consists of all invertible 
$(2\times 2)$-matrices with entries from $F$. Its centre is the group of
non-zero scalar multiples of the identity matrix; the quotient  
of $\GL(2, F)$ modulo its centre is called the {\it projective linear group} 
$\PGL(2, F)$. On this group determinants are only well-defined up to 
multiplication by squares in $F^\times$. Thus if $|F|$ is an odd integer, 
which we shall assume from now on, then $\PGL(2, F)$ has a subgroup of 
index $2$ consisting of all cosets containing a representative whose determinant 
is $1$. It is called the {\it projective special linear group} and denoted 
by $\PSL(2, F)$. One confirms easily that $|\PSL(2, F)|=\frac12 |F|(|F|^2-1)$.

We proceed with some considerations that will eventually lead us to the 
generating set $S$ of the Cayley graph we wish to define. Fix a prime 
number~$p$ such that $p\equiv 1\pmod{4}$. A result due to Jacobi~\cite{Jacobi}*{\S66}
(see also~\cite{HW}*{Theorem~386}) informs us that there are $8(p+1)$ quadruples 
of integers whose 
squares sum up to $p$. Hence there are $p+1$ quadruples $(a, b, c, d)$ such that 
$a$ is a positive odd integer, $b$, $c$, $d$ are even integers, 
and $p=a^2+b^2+c^2+d^2$. Let us now write
\[
	\HH_\ZZ=\{a+bi+cj+dk\colon a, b, c, d\in \ZZ\}
\]
for the ring of {\it integer quaternions}. Our $p+1$ integer quadruples correspond
to a set $W_p\subseteq \HH_\ZZ$ of $p+1$ quaternions with norm $p$; they come 
in $(p+1)/2$ conjugate pairs. We shall require the following easy fact 
from quaternion arithmetic a proof of which is sketched in~\cite{LPS88}*{Lemma~3.1}.

\begin{fact}\label{f:H}
	If $\alpha_1, \dots, \alpha_t\in W_p$ and the product $\alpha_1\cdots\alpha_t$
	is divisible\footnote[1]{Since $p$ is in the centre of $\HH_\ZZ$, there is no 
	need to distinguish left- and right divisibility here.} by $p$, then there is 
	some $i\in [t-1]$ such that $\alpha_i$ and $\alpha_{i+1}$ are conjugates. \qed
\end{fact}   

Intuitively speaking, this means that $W_p$ behaves like $T\cup T^{-1}$, where $T$
freely generates a group, and conjugation corresponds to taking inverses. In order 
to build a Cayley graph from this situation, we recall that the quaternion algebra 
has a two-dimensional complex representation. In particular, non-zero 
quaternions $a+bi+cj+ck$ multiply in the same way as matrices 
\begin{equation}\label{eq:0041}
	\begin{pmatrix}
		\phantom{-}a+bi & c+di \\
		-c+di & a-bi
	\end{pmatrix}
	\in \GL(2, \CC)\,.
\end{equation}

As we are aiming for a finite structure, we shall take another prime number $q\ne p$ 
and work with the finite field $\FF_q=\ZZ/q\ZZ$ as opposed to~$\CC$. Moreover, we 
demand $q\equiv 1\pmod{4}$, because then there exists an integer $f$ such that 
$f^2+1$ is divisible by $q$. Thus $f$ can play the r\^{o}le of $i$ in~\eqref{eq:0041}.
Let us write $S_{p,q}\subseteq\GL(2, \FF_q)$ for the image of $W_p$ under the map 
\begin{equation}\label{eq:0149}
	a+bi+cj+dk\longmapsto
		\begin{pmatrix}
		\phantom{-}a+bf & c+df \\
		-c+df & a-bf
	\end{pmatrix}
\end{equation}
and $\olS_{p, q}$ for the corresponding set in $\PGL(2, \FF_q)$. 
Clearly the matrices in $S_{p,q}$ have determinant~$p$. Moreover, 
conjugate quaternions $\alpha, \overline{\alpha}\in W_p$ represent 
inverse cosets in $\PGL(2, \FF_q)$. So $\olS_{p, q}$ is a 
symmetric subset of $\PGL(2, \FF_q)$ and one checks easily that
\[
	|\olS_{p, q}|=|S_{p, q}|=|W_p|=p+1\,.
\]

\begin{thm}[Lubotzky, Phillips \& Sarnak]\label{thm:lps}
	Let $p$ and $q$ be distinct primes such that $p$
	is a quadratic nonresidue modulo $q$ and $p, q\equiv 1\pmod{4}$. 
	If $t\ge 2$ denotes a further integer 
	such that $q^4>4p^t$, then the girth of the Cayley graph 
		\[
		G_{p, q}=\Cayley(\PGL(2, \FF_q), \olS_{p, q})
	\]
		exceeds $t$. Moreover, $G_{p, q}$ is bipartite. 
\end{thm}

\begin{proof}
	As the determinants of the matrices in $S_{p, q}$ fail to be squares in $\FF_q$,
	every edge of $G_{p, q}$ has exactly one endvertex in $\PSL(2, \FF_q)$ and, 
	therefore, $G_{p, q}$ is indeed bipartite.
	
	Now consider a cycle $x_1-x_2-\dots-x_r$ of length $r=\gth(G_{p, q})$
	in $G_{p, q}$. The cosets 
	$\overline{\beta}_\rho\in\olS_{p, q}$ defined by 
	$x_{\rho+1}=x_\rho\overline{\beta}_\rho$
	for every index $\rho\in \ZZ/r\ZZ$ have the property 
	that $\overline{\beta}_1\cdots\overline{\beta}_r$ is the neutral element 
	of $\PGL(2, \FF_q)$.
	Therefore there is some $w\in \FF_q^\times$ such that 
		\[
		\beta_1\cdots\beta_r
		=
		\begin{pmatrix}
		 w & 0 \\
		 0 & w
		\end{pmatrix}
	\]
		holds for the corresponding matrices $\beta_1, \dots, \beta_r\in S_{p, q}$.
	Back to quaternions this means that there are integers $W$, $X$, $Y$, $Z$
	such that 
		\begin{equation}\label{eq:0204}
		\alpha_1\cdots\alpha_r=W+q(Xi+Yj+Zk)\,,
	\end{equation}
		where $\alpha_\rho\in W_p$ denotes the preimage of $\beta_\rho$ with respect to 
	the map~\eqref{eq:0149}. Taking the norms of both sides we deduce 
		\[
		p^r=W^2+q^2(X^2+Y^2+Z^2)\,.
	\]
		Since $G_{p, q}$ is bipartite, we also know that $r$ is even.
	So $(p^{r/2}+W)(p^{r/2}-W)$ is divisible by $q^2$ and due to 
	$q\not\in\{2, p\}$ this is only possible if $q^2$ divides one factor of 
	this product. 
	
	Let us now assume for the sake of contradiction that $r\le t$.
	By our assumption $4p^t<q^4$ this yields $p^{r/2}<q^2/2$ and in combination 
	with $|W|\le p^{r/2}$ we learn $|p^{r/2}\pm W|<q^2$. Altogether 
	we must have $W=\pm p^{r/2}$ and $X=Y=Z=0$. So~\eqref{eq:0204} tells us,
	in particular, that $\alpha_1\cdots\alpha_r$ is divisible by $p$. 
	Owing to Fact~\ref{f:H} this means that for some $\rho\in [r-1]$ the 
	quaternions $\alpha_\rho$, $\alpha_{\rho+1}$ are conjugates. 
	Consequently $\overline{\beta}_\rho$, $\overline{\beta}_{\rho+1}$ are 
	inverse to each other, which in turn implies $x_\rho=x_{\rho+2}$.
	This contradiction to our assumption that $x_1-\dots-x_r$ be a cycle 
	proves $\gth(G_{p, q})=r>t$. 
\end{proof}

Let us now connect this result to the problem of bounding $f(d, g)$. 
It is not difficult to see that for every $d\le p+1$ the graph~$G_{p, q}$
has a $d$-regular subgraph. Indeed, if $d$ is even we just need to 
replace~$\olS_{p, q}$ by a subset of size $d/2$, and to cover the odd case 
as well one can exploit that Cayley graphs have cycle factors corresponding 
to the left cosets of a cyclic subgroup. Thus given $d$ and $g$ we first 
determine the least prime $p\ge d-1$ with $p\equiv 1\pmod{4}$; next
we choose the least prime $q>(4p^g)^{1/4}$ distinct from $p$ such that $p$ is a quadratic non-residue modulo $q$ and $q\equiv 1\pmod{4}$. We then have 
$f(d, q)\le |\PGL(2, \FF_q)|<q^3$. By standard results on 
primes in arithmetic progressions and quadratic reciprocity
we have $p=(1+o(1))d$ and $q=(\sqrt{2}+o(1))p^{g/4}$
(as $d, g\lra\infty$), which proves~\eqref{eq:34}. For the background 
in multiplicative number theory required here we refer to Davenports's 
textbook~\cite{Davenport}. 

It is open whether the constant~$3/4$ appearing in~\eqref{eq:34} can be 
replaced by any smaller number, but 
there have been some other minor improvements during the last decades. 
For the sake of completeness, we quote the current world record~\cite{LUW}.

\begin{thm}[Lazebnik, Ustimenko \& Woldar]
	Let $d\ge 3$ and $g\ge 5$ be given. If $q$ denotes the least odd prime power 
	with $q\ge d$, then 
		\[
		f(d, g)\le 2d q^{3g/4-a}\,,
	\]
		where $a=4, 11/4, 7/2, 13/4$ for $g\equiv 0, 1, 2, 3\pmod{4}$. \qed
\end{thm}

Let us conclude this subsection with a historical remark. 
Both the Moore bound and the concept of cages are often attributed to 
Tutte's article~\cite{Tutte47}. But, while this work is certainly related to our topic, 
it studies a somewhat different problem. Tutte begins by defining an {\it $s$-arc}
in a graph to be a walk of length $s$ with the property that any two consecutive 
edges are distinct (but there may be other repetitions of vertices and edges). 
For expository purposes let us call a connected, cubic graph {\it $s$-strong} 
if its automorphism group acts transitively on its $s$-arcs.\footnote[1]{Actually 
Tutte himself uses the term ``$s$-regular'' instead of $s$-strong, which could 
for obvious reasons seem confusing to the contemporary reader.} 
Tutte proves that every $s$-strong graph $G$ satisfies $\gth(G)\ge 2s-2$. 
By a {\it cage of order $m$} he understands a connected cubic graph $G$ of girth $m$ 
which is ``as strong as possible'', i.e., $(\lfloor m/2\rfloor+1)$-strong.
His main result, proved by group theoretic means, asserts that 
there exist only six cages, notably the graphs $K_2$, $K_4$, $K_{3, 3}$, 
the Petersen graph, the Heawood graph, and 
a graph known today as the unique $(3, 8)$-Moore graph.  
 
\subsection{Average degree}\label{subsec:23}
In his book on extremal graph theory~\cite{Boll-Ex} Bollob\'as poses the question 
whether the Moore bound remains valid when the minimum degree condition gets 
weakened to an average degree condition. This problem remained open for quite 
a long time until it was finally settled in~\cite{AHL}.

\begin{thm}[Alon, Hoory \& Linial]\label{thm:AHL1}
	Every graph $G$ with 
		\[
		d(G)\ge d\ge 2
		\quad \text{ and } \quad 
		\gth(G)\ge g\ge 3
	\]
		has at least $n_0(d, g)$ vertices. 
\end{thm}

Notice that in the situation considered here, if $G$ has a vertex of degree $0$ 
or $1$, then we can remove it without decreasing the average degree, and apply 
induction. Thus it suffices to prove Theorem~\ref{thm:AHL1} for graphs $G$ with
$\delta(G)\ge 2$. Under this assumption Alon et al.\ obtained a slightly stronger 
result involving a parameter they denote by $\Lambda(G)$. If~$G$ has~$n$ vertices and 
degree sequence $(d_1, \dots, d_n)$, the definition of this graph invariant reads
\[
	\Lambda(G)=\prod_{i=1}^n (d_i-1)^{d_i/nd}\,,
\]
where $d=d(G)$ is the average 
degree of $G$. As the function $x\longmapsto (x+1)\log x$ is convex on~$\RR_{\ge 1}$,
we have 
\begin{equation}\label{eq:1144}
	\Lambda(G)\ge d(G)-1\,.
\end{equation}
So altogether the following estimate strengthens Theorem~\ref{thm:AHL1}.

\begin{thm}[Alon, Hoory \& Linial]\label{thm:AHL2}
	Let $G$ be a graph with $\delta(G)\ge 2$. If $\gth(G)\ge g\ge 3$,
	then
		\[
		|V(G)|\ge n_0(\Lambda(G)+1, g)\,.
	\]
	\end{thm}
 
We would like to emphasise a similarity between the proof of this result and 
the entropy based proof of Sidorenko's conjecture for paths. Thus it is our next task 
to provide a brief introduction to the latter topic. Given two graphs $F$ and $G$  
we write $\Hom(F, G)$ for the set of homomorphisms from $F$ to $G$.
The probability $t(F, G)=|\Hom(F, G)|/|V(G)|^{V(F)|}$
that a random map from $V(F)$ to $V(G)$ is in $\Hom(F, G)$ is called the 
{\it homomorphism density} from~$F$ to~$G$. The following conjecture of 
Sidorenko~\cite{Sid} (see also Simonovits~\cite{Sim}) is arguably the most 
important problem on graph homomorphism densities. 

\begin{conj}[Sidorenko]
	For every bipartite graph $F$ and every graph $G$ we have 
		\[
		t(F, G)\ge t(K_2, G)^{e(F)}\,.
	\]
	\end{conj}
  
The restriction that $F$ needs to be bipartite is certainly necessary, because 
for non-bipartite graphs~$F$ every bipartite graph $G$ of positive density is a 
counterexample. The long standing `smallest unsolved case' is the following. 

\begin{prob}
	Let $M$ be the bipartite graph obtained from $K_{5, 5}$ by removing a Hamiltonian 
	cycle (see Figure~\ref{fig29}). Prove or disprove that Sidorenko's conjecture holds 
	for $F=M$. 
\end{prob}

\usetikzlibrary{calc}
	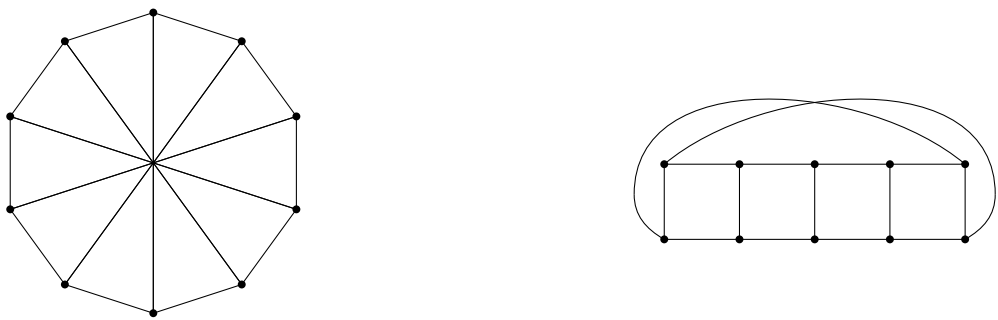
\begin{figure}[h!]
		\begin{subfigure}[b]{0.47\textwidth}
			\centering
			\begin{tikzpicture}
	\foreach \i in {0,...,14}\coordinate (\i) at (90+36*\i:2);
	\foreach \i [evaluate=\i as \j using \i+1, evaluate=\i as \k using \i+5] in {0,...,9}{
\draw (\i)--(\j);
\draw (\i)--(\k);
\fill (\i) circle (1.5pt);	
}
	\end{tikzpicture}
		\end{subfigure}
	\hfill
	\begin{subfigure}[b]{0.47\textwidth}
		\centering
		\begin{tikzpicture}
			\def\d{1}
			\foreach \i in {0,...,4}{
		\coordinate (a\i) at (\i*\d, 0);
		\coordinate (b\i) at (\i*\d,\d)	;
			\draw (a\i)--(b\i);
		\fill (a\i) circle (1.5pt);
		\fill (b\i) circle (1.5pt);
	}
\draw (a0)--(a4);
\draw (b0)--(b4);
\draw (b0) [out=40, in=90] to (4.4*\d,.6*\d );
\draw (4.4*\d,.6*\d ) [out=270, in =30] to (a4);
\draw (b4) [out=140, in=90] to (-.4*\d,.6*\d );
\draw (-.4*\d,.6*\d ) [out=270, in =150] to (a0);

\phantom{\fill (0,-1) circle (1pt);}
		\end{tikzpicture}
	\end{subfigure}
	\caption{Two drawings of the graph $M$}
	\label{fig29}
	\end{figure}
 
It should be pointed out that Lee and Sch\"ulke~\cite{LS21} refuted a natural 
strengthening of Sidorenko's conjecture for this graph $M$. Nevertheless, the 
conjecture itself is still open and we refer 
to~\cites{CFS10, CKLL18, CL17, CL21, Lov11} 
for some of the most recent contributions to this problem. 

Returning to our main story we observe that a homomorphic image of the path $P_s$
with~$s$ edges in a graph $G$ is the same as a walk of length $s$ in $G$. Thus the 
next statement agrees with the special case $F=P_s$ of Sidorenko's conjecture. 

\begin{thm}[Blakley \& Roy]\label{thm:BR}
   For every $n$-vertex graph $G$ with average degree $d$ and every positive
   integer $s$ there are at least $d^sn$ walks of length $s$ in $G$.
\end{thm}

The original proof of Blakley and Roy~\cite{BR} used linear algebra and spectral 
properties of the adjacency matrix of $G$. Later Alon and Ruzsa~\cite{AR}*{Lemma~3.8} 
developed a different approach using vertex deletions followed by the tensor power 
trick, which has the advantage that it generalises more readily to hypergraphs (see 
e.g.,~\cite{tyh}*{Lemma~2.8}). A third proof motivated by the entropy method was 
worked out by Fitch~\cite{Fitch}*{Lemma 7} and by 
Lee~\cite{Lee}*{Theorems~2.6 and~2.7} (see also~\cite{LS11}). 
Below we tell this argument with the connection 
to the theorem of Alon, Hoory, and Linial in mind. In fact, both proofs 
rely on iterated applications of the weighted inequality between the arithmetic
and the geometric mean, which states that all nonnegative reals $a_1, \dots, a_n$
and $\lambda_1, \dots, \lambda_n$ with $\lambda_1+\dots+\lambda_n=1$ satisfy 
\begin{equation}\label{eq:WAG1}
	a_1^{\lambda_1}\cdots a_n^{\lambda_n}
	\le
	\lambda_1a_1+\dots+\lambda_na_n
\end{equation}
or, equivalently, 
\begin{equation}\label{eq:WAG2}
	\lambda_1\log a_1+\dots+\lambda_n\log a_n
	\le
	\log(\lambda_1a_1+\dots+\lambda_na_n)\,.
\end{equation}

\begin{proof}[Proof of Theorem~\ref{thm:BR}]
	For standard reasons we can assume that $G=(V, E)$ has no isolated vertices, so that 
	all vertex degrees are positive. 
	We begin by observing that 
		\[
		\Psi=\prod_{x\in V}d(x)^{d(x)/dn}
	\]
		is at least $d$, because 
		\[
		\frac1\Psi
		=
		\prod_{x\in V}\Bigl(\frac1{d(x)}\Bigr)^{d(x)/dn}
		\overset{\eqref{eq:WAG1}}{\le}
		\sum_{x\in V}\frac{d(x)}{dn}\cdot\frac1{d(x)}
		=
		\frac1d\,.
	\]
		
	Now for every vertex $x$ and every positive integer $t$ we denote the
	number of $t$-walks in $G$ starting at $x$ by $W^{(t)}_x$. Due to~\eqref{eq:WAG2}
	we have 
		\begin{align*}
		\sum_{x\in V}d(x)\log\frac{W^{(t+1)}_x}{d(x)}
		&\ge
		\sum_{x\in V}\sum_{y\in N(x)}\log W^{(t)}_y
		=
		\sum_{y\in V}d(y)\log W^{(t)}_y \\
		&=
		dn\log\Psi + \sum_{x\in V}d(x)\log\frac{W^{(t)}_x}{d(x)}\,.
	\end{align*}
		In view of $\sum_{x\in V}d(x)\log\bigl(W^{(1)}_x/d(x)\bigr)=0$ this yields 
	inductively
		\[
		\sum_{x\in V}d(x)\log\frac{W^{(s)}_x}{d(x)}
		\ge
		(s-1)dn\log\Psi
		\ge 
	 	(s-1)dn\log d\,.
	\]
		For the total number $W^{(s)}$ of $s$-walks in $G$ we thus obtain
		\[
		\log\frac{W^{(s)}}{dn}
		=
		\log\sum_{x\in V}\frac{d(x)}{dn}\cdot\frac{W^{(s)}_x}{d(x)}
		\overset{\eqref{eq:WAG2}}{\ge}
		\sum_{x\in V} \frac{d(x)}{dn}\log\frac{W^{(s)}_x}{d(x)}
		\ge
		(s-1)\log d\,,
	\]
		whence $W^{(s)}\ge d^sn$.
\end{proof}
 
Now it turns out that the same method can be used not only for bounding the 
number of $s$-walks, but also for the number of $s$-arcs in Tutte's sense we
mentioned at the end of the previous subsection. Roughly speaking, this has 
the advantage that in graphs of large girth distinct $s$-arcs starting with 
the same edge need to end in different vertices, which is exactly what we need
for proving Theorem~\ref{thm:AHL2}. 

Let us fix some notation for the ensuing details. Given a graph $G=(V, E)$ 
we write $\olE$ for the set of ordered pairs $(x, y)\in V^2$ with $\{x, y\}\in E$,
so that every edge contributes two pairs to $\olE$. By an {\it $s$-arc} in $G$ 
we shall mean, from now on, a sequence $(\ole_1, \dots, \ole_s)\in\olE^s$ such that 
for every $i\in [s-1]$ the second vertex of $\ole_i$ agrees with the first vertex
of $\ole_{i+1}$, and the underlying edges of $\ole_i$, $\ole_{i+1}$ are distinct.
Given a pair $(x, y)\in\olE$ and a positive integer $s$ we write $A^{(s)}_{xy}$
for the number of $s$-arcs in $G$ starting with $(x, y)$. Finally, $A^{(s)}$ 
denotes the total number of $s$-arcs in $G$. 

\begin{lemma}\label{lem:2140}
	For every $n$-vertex graph $G=(V, E)$ with $\delta(G)\ge 2$ and every 
	positive integer~$s$ we 
	have $A^{(s)}\ge dn\Lambda^{s-1}$, where $d=d(G)$ and $\Lambda=\Lambda(G)$.
\end{lemma}   

\begin{proof}
	Given any integer $t\ge 2$ the inequality~\eqref{eq:WAG2} yields
		\begin{align*}
		\sum_{(x, y)\in \olE} \log A^{(t)}_{xy}
		&=
		\sum_{(x, y)\in \olE} \log\bigl(d(y)-1\bigr)
		+
		\sum_{(x, y)\in \olE} \log\frac{\sum_{z\in N(y)\sm\{x\}} A^{(t-1)}_{yz}}
			{d(y)-1} \\
		&\ge
		\sum_{y\in V} d(y)\log\bigl(d(y)-1\bigr)
		+
		\sum_{xyz}\frac{\log A^{(t-1)}_{yz}}{d(y)-1}\,,
	\end{align*}
		where the last sum is extended over all triples $(x, y, z)\in V^3$ such that 
	$xy, yz\in E$ and $x\ne z$. This implies
		\[
		\sum_{(x, y)\in \olE} \log A^{(t)}_{xy}
		\ge
		dn\log\Lambda+\sum_{(x, y)\in \olE} \log A^{(t-1)}_{xy}\,,
	\]
		and in view of $A^{(1)}_{xy}=1$ for all $(x, y)\in\olE$ we obtain
		\[
		\sum_{(x, y)\in \olE} \log A^{(s)}_{xy}
		\ge
		(s-1)dn\log\Lambda
	\]
		by induction. Now a final application of~\eqref{eq:WAG2} discloses
		\[
		\log\frac{A^{(s)}}{dn}
		\ge 
		(dn)^{-1}\sum_{(x, y)\in \olE} \log A^{(s)}_{xy}
		\ge
		(s-1)\log\Lambda\,,
	\]
		from which the result follows.
\end{proof}

\begin{proof}[Proof of Theorem~\ref{thm:AHL2}]
	We begin with the easier case that $g=2h$ is even. Due to Lemma~\ref{lem:2140}
	we have 
		\[
		\sum_{xy\in E}\,\sum_{i=1}^h \bigl(A_{xy}^{(i)}+A_{yx}^{(i)}\bigr)
		\ge
		dn\sum_{i=1}^h \Lambda^{i-1}
		=
		|E|n_0(\Lambda+1, g)\,.
	\]
		Thus there exists an edge $xy\in E$ with 
		\begin{equation}\label{eq:2158}
		\sum_{i=1}^h \bigl(A_{xy}^{(i)}+A_{yx}^{(i)}\bigr)
		\ge
		n_0(\Lambda+1, g)\,.
	\end{equation}
		Starting from this edge we build the same tree as in the proof of 
	Theorem~\ref{thm:21} (see Figure~\ref{fig21B}). Because of $\gth(G)\ge g$
	the number of vertices belonging to this tree is exactly the left side 
	of~\eqref{eq:2158} and, therefore, we have indeed $|V(G)|\ge n_0(\Lambda+1, g)$.
	
	It remains to deal with the case that $g=2h+1$ is odd. For every vertex~$y$ 
	and every positive integer $t$ we denote the number of $t$-arcs starting at~$y$
	by $A^{(t)}_y$. A simple counting argument reveals 
	$\bigl(d(y)-1\bigr)A^{(t)}_y=\sum_{x\in N(y)} A^{(t+1)}_{xy}$. Together with 
	Lemma~\ref{lem:2140} this leads to
		\[
		\sum_{y\in V}\bigl(d(y)-1\bigr)\bigl(A^{(1)}_y+\dots+A^{(h)}_y\bigr)
		=
		\sum_{(x, y)\in\olE}\bigl(A^{(2)}_{xy}+\dots+A^{(h+1)}_{xy}\bigr)
		\ge
		dn(\Lambda+\dots+\Lambda^h)\,.
	\]
		
	Since~\eqref{eq:1144} implies $d\Lambda\ge (d-1)(\Lambda+1)$, the right side 
	is at least
		\[
		\sum_{y\in V}\bigl(d(y)-1\bigr)(\Lambda+1)(1+\dots+\Lambda^{h-1})\,.
	\]
		Consequently there exists a vertex $y$ such that 
		\[
		1+A^{(1)}_y+\dots+A^{(h)}_y
		\ge
		1+(\Lambda+1)(1+\dots+\Lambda^{h-1})
		=
		n_0(\Lambda+1, g)\,,
	\]
		and the proof can be completed by drawing the tree in Figure~\ref{fig21A}
	rooted at~$y$. 
\end{proof}

We would finally like to mention that Hoory~\cite{Hoory} suggested very 
recently to study generalised Moore bounds for irregular graphs in terms 
of universal coverings. This gives rise to some interesting open problems 
stated at the end of his manuscript. 

\subsection{Directed graphs}\label{subsec:CH}
Problems of a completely different flavour arise when instead of ordinary graphs we 
consider directed graphs. For definiteness we agree that our directed graphs, or 
{\it digraphs} for short, have no loops or parallel arcs, but we allow cycles of 
length~$2$. For every vertex $x$ of a 
directed graph $G$ we denote its {\it out-degree}, i.e., the number of arcs 
leaving~$x$, by~$d^+(x)$, and we write $\delta^+(G)=\min\{d^+(x)\colon x\in V(G)\}$ 
for the {\it minimum out-degree} of $G$. The {\it girth} of a directed graph $G$, 
denoted again by $\gth(G)$, is the length of a shortest directed cycle in $G$, if 
there exists any. If $G$ contains no directed cycle, or equivalently if $G$ is a 
subdigraph of a transitive tournament, we set $\gth(G)=\infty$. In analogy with the 
Moore bound for undirected graphs, it is natural to ask for a strong lower bound on 
$|V(G)|$ in terms of $\delta^+(G)$ and $\gth(G)$. Here is a construction due to 
Behzad, Chartrand, and Wall~\cite{BCW}.
    
\begin{exmp}\label{ex:1}
	Let integers $d\ge 1$ and $g\ge 2$ be given, and set $n=d(g-1)+1$. Let
	$G$ be the directed graph on $\ZZ/n\ZZ$ whose arcs are all pairs of the form 
	$(x, x+i)$, where $x\in V(G)$ and $i\in [d]$. Clearly we have $\delta^+(G)=d$
	and it is not difficult to verify $\gth(G)=g$.
\end{exmp}

A famous conjecture of Caccetta and H\"{a}ggkvist~\cite{CH78} asserts that this 
construction is optimal.

\begin{conj}[Caccetta \& H\"{a}ggkvist]\label{conj:CH}
	If $g, n\ge 2$, then every directed graph~$G$ on~$n$ vertices with 
	$\delta^+(G)\ge n/g$ satisfies $\gth(G)\le g$.
\end{conj} 

For $g=2$ an easy application of the box principle (Schubfachprinzip) 
shows that this is indeed true. So far most of the effort devoted to 
the Caccetta-H\"{a}ggkvist conjecture has revolved around the case $g=3$,
which seems to be both the most approachable and the most plausible one. 
Let us restate this case as follows. 

\begin{conj}[Caccetta \& H\"{a}ggkvist, $g=3$]\label{conj:CH3}
	Every directed graph $G$ on $n$ vertices without $2$-cycles 
	which satisfies $\delta^+(G)\ge n/3$ contains a directed $3$-cycle.
\end{conj} 

An often cited reason for the enormous difficulty of this problem is that,
apart from the construction described in Example~\ref{ex:1}, it has a large 
number of further extremal configurations. This can already be seen for $n=16$,
where a second construction is obtained by starting with four blocks containing
four vertices each. Into every block we insert a directed four-cycle and then 
the blocks themselves are joined cyclically to each other (see Figure~\ref{fig26}).

	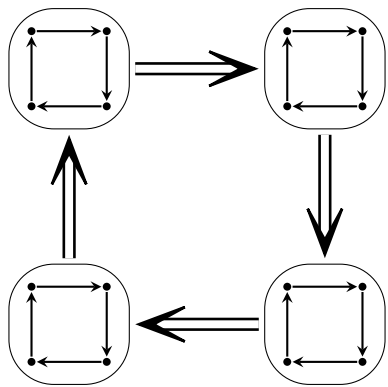
\begin{figure}[h!]
	\begin{tikzpicture}
		
		\def\jj{1.7}  		\def\ii{.5}  		\coordinate (a) at (-\jj,-\jj);
		\coordinate (b) at (-\jj, \jj);
		\coordinate (c) at (\jj, \jj);
		\coordinate (d) at (\jj,-\jj);
		\foreach \k in {a,b,c,d}{
			\coordinate (\k1) at ($(\k) +(-\ii,-\ii)$);
			\coordinate (\k2) at ($(\k) +(-\ii,\ii)$);
			\coordinate (\k3) at ($(\k) +(\ii,\ii)$);
			\coordinate (\k4) at ($(\k) +(\ii,-\ii)$);
			\coordinate (\k5) at ($(\k) +(-1.6*\ii,-1.6*\ii)$);
			\coordinate (\k6) at ($(\k) +(-1.6*\ii,1.6*\ii)$);
			\coordinate (\k7) at ($(\k) +(1.6*\ii,1.6*\ii)$);
			\coordinate (\k8) at ($(\k) +(1.6*\ii,-1.6*\ii)$);
			\foreach \s in {1,...,4}{\fill (\k\s) circle (1.5pt);}
			
			\draw [thick, shorten <= 2pt, shorten >=2pt, -stealth] (\k1)--(\k2);
			\draw [thick, shorten <= 2pt, shorten >=2pt, -stealth] (\k2)--(\k3);
			\draw [thick, shorten <= 2pt, shorten >=2pt, -stealth] (\k3)--(\k4);
			\draw [thick, shorten <= 2pt, shorten >=2pt, -stealth] (\k4)--(\k1);
			
			\draw [rounded corners=17] (\k5)--(\k6)--(\k7)--(\k8)--cycle;
}
	\foreach \i/\j in {a/b, b/c, c/d, d/a}
	\draw [line width = 1pt, double distance=3pt, {-Stealth[length=25pt, inset=17pt]}, shorten <= 25pt, shorten >=25pt] (\i)--(\j);

	\end{tikzpicture}
	\caption{A digraph $G$ with $|V(G)|=16$, $\delta^+(G)=5$, and $\gth(G)>3$. Each of the four 
	double-arrows represents $4\cdot 4=16$ arcs.}
	\label{fig26}
	\end{figure}
 
More generally, we can recursively do the following: Our building blocks are  
the digraphs provided by the case $g=3$ of Example~\ref{ex:1}; for every
integer $n\ge 4$ with $n\equiv 1\pmod{3}$ there is one of them on $n$ vertices 
with $\delta^+(G)\ge (n-1)/3$ and $\gth(G)>3$.
Now suppose that two integers $m, n\ge 4$ with $m, n\equiv 1\pmod{3}$ are given.
Take $m$ disjoint blocks consisting of $n$
vertices. Put into every block a digraph~$G$ with $\delta^+(G)\ge (n-1)/3$ and 
$\gth(G)>3$ (there is no need to take isomorphic digraphs for different 
blocks). Then join the blocks to each other according to a digraph~$H$ 
on $m$ vertices with $\delta^+(H)\ge (m-1)/3$ and $\gth(H)>3$. More explicitly,
this means that we replace the vertices of $H$ by the blocks and every arc $(x, y)$
of~$H$ by the $n^2$ arcs from the vertices in the block replacing~$x$ to the block 
replacing~$y$. Clearly the resulting digraph~$K$ has~$mn$ vertices, its 
minimum out-degree is at least $(m-1)n/3+(n-1)/3=(mn-1)/3$, and by inspection we 
see $\gth(K)>3$. At this level of generality the construction is due to 
Razborov~\cite{Raz13}, but the special case where in each step one inserts mutually 
isomorphic digraphs $G$ into the blocks can already be found in the work of 
Bondy~\cite{Bondy} (who framed it as taking the lexicographic product of $G$ and $H$).  

Partial results towards Conjecture~\ref{conj:CH3} are mostly of one of two kinds. 
First, many authors have proved the conjecture under the more restrictive minimum 
degree condition ${\delta^+(G)\ge \gamma n}$ for smaller and smaller values of 
$\gamma>\frac13$. This line of research was initiated by Caccetta and 
H\"{a}ggkvist~\cite{CH78} themselves, who obtained such a result for
$\gamma=(3-\sqrt{5})/2\approx 0.3820$.    
A numer\-ically negligible improvement to~$\gamma=(2\sqrt{6}-3)/5\approx 0.3798$ 
was reached by Bondy~\cite{Bondy}. Nevertheless the subgraph counting strategy
Bondy introduced turned out to have far-reaching consequences. In fact, it can 
be viewed as an important precursor of Razborov's influential flag algebra 
method~\cite{Raz07}. Most of the subsequent progress depends heavily on 
Razborov's ideas and on massive electronic computations. The current world record is 
an unpublished result of de Joannis de Verlos, Sereni, and Volec, who showed 
that $\gamma=0.3388$ is admissible (as reported in~\cite{GV}).   

The second group of partial results towards Conjecture~\ref{conj:CH3} addresses 
special classes of digraphs. Perhaps the most promising among them is due to 
Razborov~\cite{Raz13}. To provide some context, we remark that the extremal digraphs
described above contain no induced copies of the three digraphs drawn in 
Figure~\ref{fig27}.

\begin{thm}[Razborov]\label{thm:raz13}
	Let $G$ be a digraph on $n$ vertices satisfying $\delta^+(G)\ge n/3$.
	If $G$ contains no induced copies of the three digraphs in Figure~\ref{fig27}, 
	then $\gth(G)\le 3$. \qed
\end{thm} 

\usetikzlibrary {arrows.meta} 

	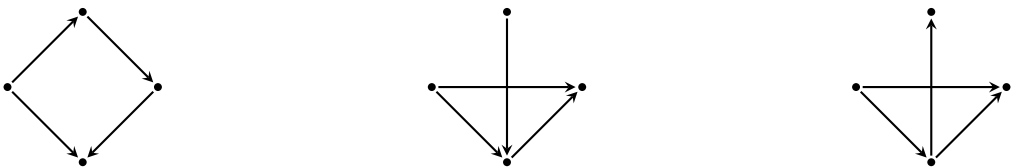
\begin{figure}[h!]
		\begin{subfigure}[b]{0.32\textwidth}
			\centering
			\begin{tikzpicture}
				\def\ii{1}
		\coordinate (a) at (-\ii,0);
		\coordinate (b) at (0, \ii);
		\coordinate (c) at (\ii,0);
		\coordinate (d) at (0,-\ii);
		\foreach \i in {a,b,c,d} \fill (\i) circle (1.5pt);
		
		\draw [thick, shorten <=2.5pt, shorten >=2.5pt, -stealth ](a) -- (b);
		\draw [thick, shorten <=2.5pt, shorten >=2.5pt, -stealth ](b) -- (c);
		\draw [thick, shorten <=2.5pt, shorten >=2.5pt, -stealth ](c) -- (d);
		\draw [thick, shorten <=2.5pt, shorten >=2.5pt, -stealth ](a) -- (d);

	\end{tikzpicture}
		\end{subfigure}
	\hfill
	\begin{subfigure}[b]{0.32\textwidth}
		\centering
		\begin{tikzpicture}
			\foreach \i in {a,b,c,d} \fill (\i) circle (1.5pt);
			\draw [thick, shorten <=2.5pt, shorten >=2.5pt, -stealth ](a) -- (c);
			\draw [thick, shorten <=2.5pt, shorten >=2.5pt, -stealth ](a) -- (d);
			\draw [thick, shorten <=2.5pt, shorten >=2.5pt, -stealth ](d) -- (c);
			\draw [thick, shorten <=2.5pt, shorten >=2.5pt, -stealth ](b) -- (d);
		
		\end{tikzpicture}
			\end{subfigure}
\hfill
\begin{subfigure}[b]{0.32\textwidth}
	\centering
	\begin{tikzpicture}
		\foreach \i in {a,b,c,d} \fill (\i) circle (1.5pt);
		\draw [thick, shorten <=2.5pt, shorten >=2.5pt, -stealth ](a) -- (c);
		\draw [thick, shorten <=2.5pt, shorten >=2.5pt, -stealth ](a) -- (d);
		\draw [thick, shorten <=2.5pt, shorten >=2.5pt, -stealth ](d) -- (c);
		\draw [thick, shorten <=2.5pt, shorten >=2.5pt, -stealth ](d) -- (b);
		
	\end{tikzpicture}
		\end{subfigure}
	\caption{Razborov's forbidden subdigraphs}
	\label{fig27}
	\end{figure}
 
Next we come to some selected partial results towards the general version 
of the problem, Conjecture~\ref{conj:CH}. Chv\'atal and Szemer\'edi~\cite{CS83}
showed $\gth(G)\le |V(G)|/\delta^+(G)+2500$ for every digraph $G$. The explicit 
constant $2500$ was later lowered to $73$ by Shen~\cite{Shen02}. Earlier, 
Shen had already resolved the case 
$|V(G)|\ge (\delta^+(G)-1)(2\delta^+(G)-1)$ in~\cite{Shen00},
but it should be mentioned that in this regime the conjectured nested nature 
of the extremal configurations is irrelevant.
In a completely different direction Hamidoune proved the Caccetta-H\"{a}ggkvist
conjecture for vertex transitive digraphs~\cite{Ham81a}.

There are also quite a few problems on digraphs motivated by or related to the
Caccetta-H\"{a}ggkvist conjecture. Here we would like to offer two of them, chosen 
for aesthetic reasons alone. The first is from~\cite{CSS}.

\begin{conj}[Chudnovsky, Seymour \& Sullivan]
	Every digraph $G$ with $\gth(G)>3$ satisfies $\beta(G)\le \gamma(G)/2$,
	where $\beta(G)$ denotes the least number of arcs of $G$ whose deletion 
	yields an acyclic digraph, and $\gamma(G)$ is the number of non-adjacent 
	pairs of vertices of $G$.
\end{conj}

Equality holds for digraphs obtained from balanced blow-ups of the directed
four-cycle by inserting transitive tournaments into the four vertex classes. 
Chudnovsky, Seymour, and Sullivan themselves proved their conjecture
for a natural class of digraphs containing these examples, called {\it circular 
interval digraphs}. These are the digraphs whose vertex sets can be enumerated 
in such a way as $\{v_i\colon i\in \ZZ/n\ZZ\}$ that every vertex $v_i$ has an 
out-neighbourhood of the form $\{v_{i+1}, \dots, v_{i+j(i)}\}$ and an in-neighbourhood 
of the form $\{v_{i-1}, \dots, v_{i-k(i)}\}$. Furthermore they proved the linear
bound $\beta(G)\le \gamma(G)$ for all digraphs $G$ with $\gth(G)>3$, which 
was strengthened to $\beta(G)\le 0.88 \gamma(G)$ by Dunkum, Hamburger, 
and P\'or~\cite{DHP}.

The next problem is due to Seymour and Spirkl~\cite{SS20}. They call a digraph
{\it bipartite} if its underlying graph is bipartite; similarly, by a {\it bipartition}
of a bipartite digraph they mean a bipartition of its underlying graph. 

\begin{conj}[Seymour \& Spirkl]
	Let $k$ be a positive integer, and let $\alpha$, $\beta$ be positive 
	reals such that $k\alpha+\beta\ge 1$. Further, let $(A, B)$ be a bipartition 
	of a bipartite digraph $G$. If every vertex in $A$ has out-degree at least 
	$\beta|B|$ and every vertex in $B$ has out-degree at least~$\alpha |A|$,
	then $\gth(G)\le 2k$.
\end{conj} 

As observed in~\cite{SS20}, this would imply Conjecture~\ref{conj:CH}. 
Seymour and Spirkl proved their conjecture for $k=2$. 
We would finally like to mention that Grzesik and Volec~\cite{GV}
have strong results on the problem where one wants to use
a minimum out-degree condition to enforce a directed cycle of given length
(rather than bounded length).

\section{The chromatic number}\label{sec:3}

\subsection{A theorem of Erd\H{o}s}\label{subsec:41}
A colouring of the vertices of a graph is said to be {\it proper} if any two 
adjacent vertices receive distinct colours.
The {\it chromatic number} of a graph $G$, denoted by $\chi(G)$, is the least natural 
number~$r$ such that there exists a proper $r$-colouring of $G$.
For reasons that will become apparent in \S\ref{subsec:35} and 
Section~\ref{sec:grt} this is a Ramsey theoretic invariant of~$G$. 
The question motivating us here is which graphs $F$ appear 
in all graphs whose chromatic number is sufficiently large.

\begin{fact}\label{f:F}
	For every forest $F$ there is a natural number $r$ such that every graph $G$
	with $\chi(G)>r$ has a subgraph isomorphic to $F$. 
\end{fact}

\begin{proof}
	Set $r=|V(F)|$. Choose a minimal subgraph $G'$ of $G$ such that $\chi(G')>r$.
	For every vertex $x$ of $G'$ there is a proper $r$-coloring of $G'-x$; if~$x$
	had fewer than $r$ neighbours in~$G'$, then we had a free colour for $x$, 
	thus getting a proper $r$-colouring of $G'$. This proves $\delta(G')\ge r$ 
	and, consequently, we can embed $F$ greedily into $G'$. 
\end{proof}  

A famous result of Erd\H{o}s~\cite{Erd59} endows this observation with an aura
of optimality: large chromatic number is compatible with the absence of short cycles. 

\begin{thm}[Erd\H{o}s]\label{thm:41}
	For all natural numbers $g$ and $r$ there exists a graph $G$ such that $\gth(G)>g$ 
	and $\chi(G)>r$.
\end{thm}

Erd\H{o}s' own proof was probabilistic and has been repeated in 
many textbooks (see e.g., Bollob\'as~\cite{Boll-Mod}*{Theorem VII.4}),
so we can be very brief about it: for a large number of vertices~$n$ and 
probability $p=(\log n)/n$ (say) one considers the random graph~$G(n, p)$. 
With positive probability (in fact almost surely), it contains $o(n)$ short 
cycles and has no independent set of size $\Omega(n)$. So by deleting all vertices 
in cycles of length at most $g$ one obtains a graph on more than $n/2$
vertices whose chromatic number exceeds $r$. 

There is a less well-known variant of this argument, due to R\"odl~\cite{Rodl90}, 
which we would like to describe in more detail, because it is sometimes quite useful
in other contexts (see e.g.,~\cites{pisier, Rodl90}). 
The basic idea is that we start with a large set of vertices,
which does not have any edges yet, and keep adding edges one by one. 
In each step we want to decrease the number of proper $r$-colourings still available
by a constant proportion, so that after not too many steps all potential colourings 
have been `killed'. 
The only thing we need to avoid is that at some moment we cannot continue 
because too many candidate edges would close a short cycle. To exclude this 
outcome, we shall maintain a maximum degree condition, which will ensure that the 
number of unavailable edges stays under control.

\begin{proof}[Proof of Theorem~\ref{thm:41}]
	Given $g$ and $r$ we choose auxiliary constants $\alpha>0$ and $C, n\in\NN$ 
	according to the hierarchy 
		\[
		n\gg C\gg\alpha^{-1}\gg g, r\,.
	\]
		For instance, all of our estimates go through for  
		\[
		\alpha=(3r)^{-1}\,,
		\quad
		C=\lceil 48r^2\log r\rceil\,,
		\quad \text{ and } \quad 
		n=24rC^{g-1}\,.
	\]
	
	Given a graph $G$ we denote the set of its proper $r$-colourings 
	$\phi\colon V(G)\lra [r]$ by $B(G)$. Fix a set $V$ of $n$ vertices.  
	We call a graph $G$ on $V$ {\it good}, if 
		\begin{enumerate}[label=\rmlabel]
		\item\label{it:gi} $\gth(G)> g$;
		\item\label{it:gii} $\Delta(G)\le C$;
		\item\label{it:giii} and $|B(G)|\le (1-\alpha)^{e(G)}r^n$.
	\end{enumerate}
 		E.g., the edgeless graph on $V$ is good. Pick a good graph $G$ such that $e(G)$
	is maximal. If $G$ has more than 
		\[
		q=\frac{n\log r}{\alpha}
	\]
		edges, then~\ref{it:giii} yields $|B(G)|<\exp(-q\alpha+n\log r)=1$,
	which means that $G$ has no proper $r$-colouring. So in this case $G$ 
	has the desired properties. 
	Now suppose towards a contradiction that $G$ has at most $q$ edges. 
	
	\begin{claim}
		There are at most $n^2/6r$ pairs $e\in V^{(2)}$ such that the graph $G+e$
		violates~\ref{it:gi} or~\ref{it:gii}.
	\end{claim}
	
	\begin{proof} 
		A pair of nonadjacent vertices is excluded by~\ref{it:gi} if and only if 
		these two vertices have distance at most $g-1$ in $G$. Because of the maximum 
		degree condition, there are at most 
				\[
			\tfrac12 (C+\dots+C^{g-1})n
			\le 
			C^{g-1}n
			\le 
			\frac{n^2}{24r}
		\]
				such pairs. Analysing~\ref{it:gii} we observe that
		due to $\sum_{x\in V}d(x)=2e(G)\le 2q$ the set 
				\[
			U=\{x\in V\colon d(x)=C\}
		\]
				has at most the size $|U|\le 2q/C$. Therefore, there are at most
				\[
			|U|n
			\le 
			\frac{2qn}{C}
			=
			\frac{2(\log r)n^2}{\alpha C}
			\le
			\frac{n^2}{8r}
		\]
				pairs whose addition to $G$ would cause the failure of~\ref{it:gii}.
		Since $1/24+1/8=1/6$, the claim follows. 
	\end{proof}
	
	Let us now consider an arbitrary colouring $\phi\in B(G)$. There are at least 
	$r\binom{n/r}2>\frac{n^2}{3r}$ pairs of vertices receiving the same colour 
	with respect to $\phi$. Among them, there are by our claim at 
	least $\frac{n^2}{6r}>\alpha\binom n2$
	pairs that could be added to $G$ without harming~\ref{it:gi} or~\ref{it:gii}.
	Using a double counting argument we conclude that there is a pair $e$ such that
	$G+e$ satisfies~\ref{it:gi} and~\ref{it:gii}, and $e$ is monochromatic for at 
	least $\alpha|B(G)|$ colourings in $B(G)$. But now
		\[
		|B(G+e)|\le (1-\alpha)|B(G)|\le (1-\alpha)^{e(G+e)}r^n
	\]
		shows that the graph $G+e$ is good, contrary to the maximality of $G$.  
\end{proof}

The next two subsections deal with explicit constructions of graphs 
and hypergraphs with large chromatic number and large girth. 
In~\S\ref{subsec:cages} we already came quite close to seeing a number theoretic
example. Suppose that we change the assumptions of Theorem~\ref{thm:lps} to~$p$ 
being a quadratic residue modulo $q$. Then 
$G_{p, q}=\Cayley(\PSL(2, \FF_q), \olS_{p, q})$ is well-defined and an argument
similar to the one we have seen shows $\gth(G_{p, q})\ge 2\log_p q$.
Lubotzky, Philipps, and Sarnak~\cite{LPS88}*{p.263} have further established 
$\chi(G_{p, q})\ge (p+1)/2\sqrt{p}$. 
In particular, by choosing $p$ and $q$ appropriately, the chromatic number and 
girth of $G_{p, q}$ can both be made arbitrarily large. From now on, we confine 
ourselves to `combinatorial' constructions.

\subsection{Historical constructions}\label{subsec:32}
Most of the earliest protagonists in the study of graphs of large chromatic number 
and large girth were young researchers, who did not know much about each other's
work. Of course in the 1950s and 1960s, when these developments happened, 
information did usually not travel with the speed of light, and borders still 
meant something. 

The first relevant reference was written by Zykov at the age of 24. 
In~\cite{Zykov}*{\rn{Glava 3}, \S3} he compares two graph parameters, which 
he calls `rank' (\rn{rang}) and `density' (\rn{plotnostp1}). Today one would 
speak of the chromatic number and clique number,\footnote[1]{The clique number 
of a graph $G$, denoted by $\omega(G)$, is the largest natural number $d$ such 
that $G$ contains a clique of order $d$.} respectively. After 
observing the trivial estimate $\chi(G)\ge \omega(G)$ he shows that, sort of 
conversely, for all pairs of natural numbers $(r, d)$ 
with $r\ge d\ge 2$ there exists a graph $G$ such that $\chi(G)=r$ 
and $\omega(G)=d$. In particular, to $d=2$ there correspond 
triangle-free graphs of arbitrarily large chromatic number. 

For fixed $d\ge 2$ Zykov argues by induction on $r$, starting with the 
clique~$K_d$ as his base case.
Now suppose that for some $r\ge d$ a graph $G$ satisfying $\chi(G)=r$ 
and $\omega(G)=d$ has already been found. Let $G_1, \dots, G_r$ be 
vertex-disjoint copies of $G$. By a {\it transversal} we shall mean a set 
of $r$ vertices, one from each of these graphs. For every transversal~$T$ 
we take a new vertex $x_T$ and join it 
to the members of $T$ (see Figure~\ref{fig41}).

	\begin{figure}[h!]
	\begin{tikzpicture}
		
		\def\h{1}
		
		\draw [thick] (0,3*\h) -- (5.5,3*\h);
		\draw [thick] (0,2.2*\h) -- (5.5,2.2*\h);
		\draw [thick] (0,0) -- (5.5,0);
		
		\coordinate (a) at (.5,3*\h);
		\coordinate (c) at (3.5,0);
		\coordinate (b) at ($(a)!.265!(c)$);
		\coordinate (1) at (4.5,.4*\h);
		\coordinate (2) at (4.5, 1.1*\h);
		\coordinate (3) at (4.5, 1.8*\h);
		
		\qedge{(a)}{(b)}{(c)}{5pt}{.6pt}{black}{white, opacity=0};
	
		\coordinate (x) at (-2,1.3*\h);
		\draw (a)--(x)--(b)--(x)--(c);
				
		\foreach \i in {a,b,c,x, 1, 2,3} \fill (\i) circle (1.2pt); 
		
		\node at (6,3*\h) {$G_1$};
		\node at (6,2.2*\h) {$G_2$};
		\node at (6,0) {$G_r$};
		\node [left] at (x) {$x_T$};
		\node at (2.8,1.4*\h) {$T$};

	\end{tikzpicture}
	\caption{Zykov's construction}
	\label{fig41}
	\end{figure}
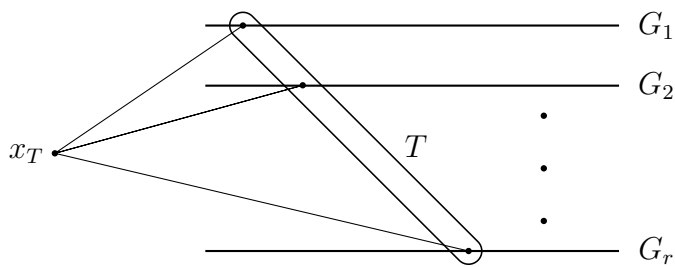
 
The resulting graph $G'$ clearly has clique number $d$. 
By $r$-colouring the graphs $G_1, \dots, G_r$
with the same $r$ colours, and assigning a new colour to all vertices $x_T$
we see the upper bound $\chi(G')\le r+1$. 
Now assume for the sake of contradiction that some proper $r$-colouring 
of~$G'$ existed. Without loss of generality, our set of colours is $[r]$. 
Due to $\chi(G)=r$ there is for every index $i\in [r]$ a vertex $x_i\in V(G_i)$
receiving the colour $i$. These vertices form a transversal $T=\{x_1, \dots, x_r\}$,
but there is no free colour for $x_T$. This proves $\chi(G')=r+1$ and the 
induction is complete. 

A few years after Zykov's work, Ungar, who was apparently unaware of it, 
posed the problem to construct triangle-free graphs of arbitrarily large 
chromatic number in the American mathematical monthly~\cite{UD54}. 
The editors received three solutions (including one from Ungar himself),
but only the submission of Descartes (a pseudonym of Tutte) got printed. 
Tutte's graphs have girth at least six; they are constructed recursively 
as follows. 

Start, for instance, with a cycle of length $6$, which has chromatic number $2$. 
Now suppose inductively that you already have a graph $G$ with $\chi(G)\ge r\ge 2$
and $\gth(G)\ge 6$. Set $n=|V(G)|$, take an independent set~$Y$ of size 
$(n-1)r+1$, and join each $n$-element subset of~$Y$ to its own copy of~$G$ 
by means of a matching (see Figure~\ref{fig42A}). 

\begin{figure}[h!]
		\begin{subfigure}[b]{0.47\textwidth}
			\centering
			\begin{tikzpicture}[scale=.5]
	
	\foreach \i in {-3,...,3}{
\coordinate (a\i) at (\i,1);
\coordinate (b\i) at (\i,5);
\fill (a\i) circle (1.8pt);	
}

\coordinate (c2) at (-3.5,4.6);
\coordinate (c1) at (-4.45,4.4);
\coordinate (c3) at (-2.55,4.8);

\coordinate (d2) at (3.5,4.6);
\coordinate (d1) at (4.45,4.4);
\coordinate (d3) at (2.55,4.8);

\draw (0,1) ellipse (3.8cm and 1.1cm);
\draw (0,1) ellipse (1.4cm and .5cm);
\draw (-2,1) ellipse (1.4cm and .5cm);
\draw (2,1) ellipse (1.4cm and .5cm);
\draw (0,5) ellipse (1.4cm and .5cm);
\draw [rotate around={10:(-3.5,4.6)}](-3.5,4.6) ellipse (1.4cm and .5cm);
\draw  [rotate around={-10:(3.5,4.6)}](3.5,4.6) ellipse (1.4cm and .5cm);

	\draw (b-1)--(b1);
	\draw (c1)--(c3);
	\draw (d1)--(d3);
	\foreach \i in {1,2,3}{
		\fill (c\i) circle (1.8pt);
		\fill (d\i) circle (1.8pt);
	}
\draw (c1)--(a-3);
\draw (c2)--(a-2);
\draw (c3)--(a-1);
\draw (d1)--(a3);
\draw (d2)--(a2);
\draw (d3)--(a1);
	\foreach \i in {-1,0,1} {
		\fill (b\i) circle (1.8pt);
		\draw (a\i)--(b\i);}
	
	\draw [thick, -stealth] (0,7)--(0,5.8);
	\draw [thick, -stealth] (-1.5,7) [out = 200, in=70] to (-3,5.4);
	\draw [thick, -stealth] (1.5,7) [out = -20, in=110] to (3,5.4);
	
	\node at (4.9,1) {$Y$};
	\node at (0,7.5) {copies of $G$};
	
	\end{tikzpicture}
\vskip -.1cm
\caption{Tutte's construction}\label{fig42A}
		\end{subfigure}
	\hfill
	\begin{subfigure}[b]{0.47\textwidth}
		\centering
		\begin{tikzpicture}[scale=.7]
		
			\coordinate (a1) at (-.5,1.5);
			\coordinate (a2) at (.5,1.5);
			\coordinate (c1) at (-2.1,4.5);
			\coordinate (c2) at (-1.1,4.7);
			
			\coordinate (d2) at (2.1,4.5);
			\coordinate (d1) at (1.1,4.7);

			\draw (0,1.5) ellipse (2.5cm and .6cm);
			\draw (-.5,1.5) ellipse (1.2cm and .3cm);
			\draw (.5,1.5) ellipse (1.2cm and .3cm);
			\draw [rotate around={10:(-1.6,4.6)}](-1.6,4.6) ellipse (1.1cm and .3cm);
				\draw [rotate around={-10:(1.6,4.6)}](1.6,4.6) ellipse (1.1cm and .3cm);

			\draw (c1)--(c2);
			\draw (d1)--(d2);
			\foreach \i in {1,2}{
				\fill (c\i) circle (1.5pt);
				\fill (d\i) circle (1.5pt);
				\fill (a\i) circle (1.5pt);
				\draw (c\i)--(a\i)--(d\i);
			}
					
			\node at (3.5,1.5) {$Y$};

		\end{tikzpicture}
		\vskip -.1cm
		\caption{A cycle of length 6}\label{fig42B}
	\end{subfigure}
	\caption{}	\end{figure}
 
It is easy to see that the resulting graph $G'$ satisfies $\gth(G')\ge 6$. Moreover, for every $r$-colouring of~$G'$ there needs to be a 
monochromatic $n$-set $X\subseteq Y$ (by the box principle), and the colour of $X$ 
is then unavailable for the copy of $G$ attached to $X$. 
Thus we have $\chi(G')\ge r+1$ and the induction continues. 

At this juncture, Jarik enters our story. As reported in~\cite{Ne09}, 
he enrolled at Charles University in Prague in the middle of the 1960s. 
Almost immediately he began to contemplate research problems in graph theory. 
This quickly led to his first publication~\cite{Ne66} written at the age 
of~$20$. Therein he studies the problem of generalising Tutte's construction
and manages to exclude cycles of lengths six and seven as well. Interestingly,
and perhaps even fortunately, the knowledge that in the meantime Erd\H{o}s had 
already proved Theorem~\ref{thm:41} had not arrived in Prague yet. 

To get some first ideas, suppose that for some integer $r\ge 2$ we already have 
a graph~$H$ with $\chi(H)\ge r$ and $\gth(H)\ge 8$. If we applied Tutte's 
construction directly to $H$, then the appearance of $6$-cycles would be hard to 
avoid (see Figure~\ref{fig42B}).

Jarik's plan to get around this difficulty is that he considers a `cleverly selected' 
independent set $I\subseteq V(H)$ and joins only the copies of $I$ to 
the subsets~$X\subseteq Y$. More explicitly, writing $|I|=n$ he takes again 
an independent set $Y$ of size $(n-1)r+1$ and for every $n$-element 
subset $X\subseteq Y$ he creates its own copy $(H_X, I_X)$ of the pair $(H, I)$ 
such that $Y$ and all $\binom{|Y|}{n}$ sets $V(H_X)$ are mutually disjoint. 
Now he joins every set~$X$ to the corresponding set~$I_X$ by a matching, 
thus arriving at a graph~$H'$ with $\gth(H')\ge 8$ (see Figure~\ref{fig43}).

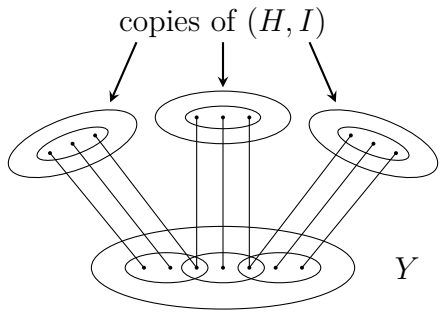
\begin{figure}[h!]
			\centering
			\begin{tikzpicture}[scale=0.5]
	
	\foreach \i in {-3,...,3}{
\coordinate (a\i) at (\i*.7,1);
\coordinate (b\i) at (\i*.7,5);
\fill (a\i) circle (1.5pt);	
}

\coordinate (c2) at (-4,4.3);
\coordinate (c1) at (-4.6,4.05);
\coordinate (c3) at (-3.4,4.52);

\coordinate (d2) at (4,4.3);
\coordinate (d1) at (4.6,4.05);
\coordinate (d3) at (3.4,4.52);

\draw (0,1) ellipse (3.5cm and 1.1cm);

\draw (0,1) ellipse (1.1cm and .4cm);
\draw (-1.5,1) ellipse (1.1cm and .4cm);
\draw (1.5,1) ellipse (1.1cm and .4cm);

\draw (0,5) ellipse (1cm and .3cm);
\draw [rotate around={20:(-4,4.3)}](-4,4.3) ellipse (1cm and .3cm);
\draw  [rotate around={-20:(4,4.3)}](4,4.3) ellipse (1cm and .3cm);

\draw (0,5) ellipse (1.8cm and .7cm);
\draw [rotate around={20:(-4,4.3)}](-4,4.3) ellipse (1.8cm and .7cm);
\draw  [rotate around={-20:(4,4.3)}](4,4.3) ellipse (1.8cm and .7cm);

	\foreach \i in {1,2,3}{
		\fill (c\i) circle (1.5pt);
		\fill (d\i) circle (1.5pt);
	}
\draw (c1)--(a-3);
\draw (c2)--(a-2);
\draw (c3)--(a-1);
\draw (d1)--(a3);
\draw (d2)--(a2);
\draw (d3)--(a1);
	\foreach \i in {-1,0,1} {
		\fill (b\i) circle (1.5pt);
		\draw (a\i)--(b\i);}
	
	\draw [thick, -stealth] (0,7)--(0,5.8);
	\draw [thick, -stealth] (-2.3,7) -- (-3,5.4);
	\draw [thick, -stealth] (2.3,7)-- (3,5.4);

	\node at (0,7.5) {copies of $(H,I)$};
	\node at (4.9,1) {$Y$};
	
	\end{tikzpicture}
	\caption{Jarik's idea to avoid $6$- and $7$-cycles} 
	\label{fig43}
	\end{figure}
 
The only problem we are facing now is that it is less clear 
whether $\chi(H')\ge r+1$ can still be proved. Given a proper $r$-colouring 
of $H'$ it remains true that there is a monochromatic $n$-set $X\subseteq Y$
and that the colour of $X$ is blocked on $I_X$. This would lead to a contradiction
if we could guarantee that for every proper $r$-colouring of $H$ all
colours had to appear on~$I$. In other words, we need to assume a strong form of 
the induction hypothesis, notably the existence of an appropriate pair $(H, I)$.
Thus the usual question arises whether this 
extra strength is maintainable in the induction. 
However, the most obvious candidate for the new set $I'$, namely the union of 
all sets $I_X$, does not seem viable.

Jarik solves this problem by adding an `inner induction' on a new parameter $k$. 
Given two integers $r\ge k\ge 1$ he considers the following statement.
\begin{quotation}
	$(J_{k, r})$ There is a pair $(H, I)$ consisting of a graph $H$ with $\chi(H)\ge r$
	and $\gth(H)\ge 8$, and an independent set $I\subseteq V(H)$ such that 
	for every proper $r$-colouring of $H$ at least $k$ colours appear on $I$. 
\end{quotation}

Notice that $J_{1, r}$ is equivalent to the existence of a graph $H$ with 
$\chi(H)\ge r$ and $\gth(H)\ge 8$.
Moreover, Jarik's modification of Tutte's construction establishes the implication 
\[
	J_{r, r} 
	\,\,\, \Longrightarrow \,\,\,
	J_{1, r+1}
	\qquad \text{for every integer $r\ge 2$.}
\]
So to complete the entire argument it suffices to prove 
\[
	J_{k, r}
	\,\,\, \Longrightarrow \,\,\,
	J_{k+1, r}
	\qquad
	\text{ whenever } 1\le k<r\,.
\]

To this end Jarik employs the following construction. 
Let $(H, I)$ be a pair exemplifying~$J_{k, r}$, set $n=|I|$, and 
let $H_1, \dots, H_n$ be vertex-disjoint copies of $H$. By a {\it transversal}
we shall again mean a set $T$ consisting of one vertex from each of these graphs. 
For every transversal $T$ let $(H_T, I_T)$ be a pair isomorphic to $(H, I)$
such that all graphs $H_T$ are mutually vertex-disjoint and vertex-disjoint 
to $H_1, \dots, H_n$. Next, we connect every set $I_T$ with a matching to the 
corresponding transversal $T$, thereby obtaining a graph $H_\star$. Finally, 
we let $I_\star$ be the union of the sets $I_T$ over all transversals $T$ 
(see Figure~\ref{fig44}). 

\usetikzlibrary{intersections}

	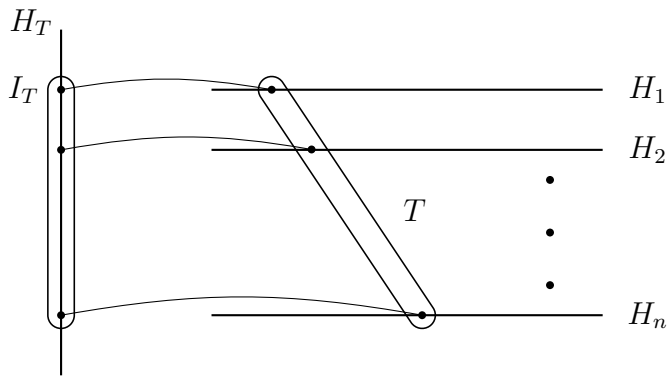
\begin{figure}[h!]
	\begin{tikzpicture}
		
		\def\h{1}
		
		\draw [thick] (0,3*\h) -- (5.2,3*\h);
		\draw [thick] (0,2.2*\h) -- (5.2,2.2*\h);
		\draw [thick] (0,0) -- (5.2,0);
		\draw [thick] (-2,3.8*\h)--(-2,-.8*\h);
		
		\coordinate (a) at (.8,3*\h);
		\coordinate (c) at (2.8,0);
		\coordinate (b) at ($(a)!.265!(c)$);
		\coordinate (1) at (4.5,.4*\h);
		\coordinate (2) at (4.5, 1.1*\h);
		\coordinate (3) at (4.5, 1.8*\h);
		\coordinate (x) at (-2, 3*\h);
		\coordinate (y) at (-2, 2.2*\h);
		\coordinate (z) at (-2, 0);
		
		\qedge{(a)}{(b)}{(c)}{5pt}{.6pt}{black}{red, opacity=0};
		\qedge{(x)}{(y)}{(z)}{5pt}{.6pt}{black}{red, opacity=0};
		
	\draw (a) edge [bend right=10] (x);
	\draw (b) edge [bend right=10] (y);
	\draw (c) edge [bend right=10] (z);
				
		\foreach \i in {a,b,c, x,y,z,1, 2,3} \fill (\i) circle (1.5pt); 
		
		\node at (5.8,3*\h) {$H_1$};
		\node at (5.8,2.2*\h) {$H_2$};
		\node at (5.8,0) {$H_n$};
		\node at (2.7,1.4*\h) {$T$};
		\node at (-2.5, 3*\h) {$I_T$};
		\node at (-2.4, 3.9*\h) {$H_T$};

	\end{tikzpicture}
	\caption{The proof of $J_{k,r}\Rightarrow J_{k+1,r}$.}
	\label{fig44}
	\end{figure}
 
We contend that the pair $(H_\star, I_\star)$ is as required by $J_{k+1, r}$.
The demands $\chi(H_\star)\ge r$ and $\gth(H_\star)\ge 8$ are clear, and $I_\star$ 
is obviously independent. Now we assume for the sake of contradiction that there 
is a proper $r$-colouring of $H_\star$ such that at most $k$ distinct colours appear
on~$I_\star$. Let $\alpha$ be any of these colours. Due to $\chi(H)\ge r$ there 
is for every $i\in [n]$ a vertex $x_i\in V(H_i)$ receiving the colour $\alpha$. 
The set $T=\{x_1, \dots, x_n\}$ is a transversal and due to our matchings the 
colour $\alpha$ cannot appear on $I_T$. Thus there is a proper $r$-colouring 
of $H_T$ such that less than $k$ colours occur on $I_T$. This contradiction 
to the choice of the pair $(H, I)$ concludes our description of Jarik's argument.  

Almost immediately after the appearance of this work, the 19-year old 
Lov\'asz discovered a general construction of hypergraphs with large chromatic number
and large girth~\cite{Lov68}. Let us briefly pause to explain the terms involved 
here. By a {\it proper colouring} of a hypergraph~$H$ we again mean a colouring 
of $V(H)$ without monochromatic edges, and the {\it chromatic number}~$\chi(H)$
is the least natural number $r$ such that some proper $r$-colouring of $H$ exists.
For $n\ge 2$ a {\it cycle of length n} in a hypergraph $H$ is a cyclic sequence 
$e_1v_1\dots e_nv_n$ consisting of distinct edges $e_1, \dots, e_n\in E(H)$ and 
distinct vertices $v_1, \dots, v_n$ such that $v_i\in e_i\cap e_{i+1}$ holds for 
every $i\in\ZZ/n\ZZ$. As expected, $\gth(H)$ denotes the least $n$ such that $H$ 
contains some cycle of length~$n$, if there exists any; otherwise we call $H$ a 
{\it forest} and set $\gth(H)=\infty$.   
A hypergraph is said to be {\it linear} if any two distinct edges intersect in 
at most one vertex. Notice that a hypergraph contains a cycle of length $2$ if and 
only if it is not linear. 

It would take us too far afield to describe the details of Lov\'asz' construction,
but it has one remarkable aspect that deserves being pointed out. 
There are now three parameters in the statement. Given $g$, $k$, and $r$
we seek a $k$-uniform hypergraph $H$ with $\gth(H)>g$ and $\chi(H)>r$. 
Lov\'asz obtains such hypergraphs by an outer induction on $g$, 
and in the induction step he performs an inner induction on $r$. While all this 
happens, the value of $k$ is not kept fixed. Rather, Lov\'asz exploits the 
possibility to obtain girth increments by looking at auxiliary $k'$-uniform 
hypergraphs, where $k'$ is quite huge in comparison to $k$. In particular, one cannot 
simply ``focus on the graph case'' when studying Lov\'asz's article. 
This idea of controlling girth by means of higher-order structures is still 
of key importance in current research and we shall encounter it again when 
talking about the girth Ramsey theorem later.
As it can be done without much effort, we would briefly like 
to illustrate how hypergraphs can assist us when constructing graphs of large 
chromatic number and large girth. In Tutte's construction, we can view the 
collection of all $n$-element subsets $X$ of $Y$ as a complete $n$-uniform 
hypergraph $K=K^{(n)}_{r(n-1)+1}$ of order $|Y|=r(n-1)+1$. Our use of the box 
principle corresponds to the fact that the chromatic number of $K$ 
exceeds~$r$. The problem that we cannot avoid $6$-cycles (see Figure~\ref{fig42B}) is 
caused by the fact that $K$ is not linear. 
If instead of $K$ we take a linear $n$-uniform hypergraph~$L$
with $\chi(L)>r$ and attach our copies of the previous graph $G$ only to the edges 
of $L$, then we can even maintain the condition $\gth(G)\ge 9$. As  
the linearity of~$L$ is equivalent to $\gth(L)\ge 3$, we see that edge-size 
can indeed be traded for girth. More generally, Tutte's 
construction shows that if for all~$k\ge 2$ we can construct $k$-uniform 
hypergraphs~$H$ of arbitrarily large chromatic number with $\gth(H)\ge g$, 
then there are graphs $G$ of arbitrarily large chromatic number with $\gth(G)\ge 3g$. 
Further properties and variants of Tutte's graphs were discovered by Kostochka 
and Jarik~\cite{KN99}.

Before moving on to a different hypergraph construction in the next subsection, 
we would like to mention that the problem of finding an `explicit, purely graph 
theoretic, hypergraph-free' construction of graphs with large chromatic 
number and large girth was popularised a lot by Jarik, until it was finally solved by 
his student K\v{r}\'{\i}\v{z}~\cite{Kriz}. A perhaps more transparent alternative 
construction has recently been provided by Alon et al.\ in~\cite{AKRWZ16}. 

\subsection{The partite construction method}\label{subsec:33}
Our next goal is to describe, in a very simple scenario, the partite construction 
method invented by Jarik and R\"{o}dl. The result we shall 
prove in this manner is originally due to Erd\H{o}s and Hajnal, who notice
in~\cite{EH66}*{Corollary~13.4} that Erd\H{o}s' probabilistic argument for the graph 
case generalises straightforwardly to hypergraphs. It is probably clear that 
R\"odl's proof we saw in \S\ref{subsec:41} transfers to hypergraphs as well.    

\begin{thm}[Erd\H{o}s \& Hajnal]\label{thm:EH}
	For all integers $g, k, r\ge 2$ there exists a $k$-uniform hypergraph~$H$
	such that $\gth(H)>g$ and $\chi(H)>r$.
\end{thm}

As in Lov\'asz's construction mentioned in the previous subsection, there is 
an induction on~$g$. To keep the exposition as simple as possible, we 
shall first explain how one would handle the case $g=2$ by partite 
construction. So given $k$ and $r$ we are aiming for a linear, $k$-uniform 
hypergraph $H$ such that $\chi(H)>r$. Without the linearity constraint, we 
could simply take the clique $G=K^{(k)}_{(k-1)r+1}$. It will be convenient 
to write $n=(k-1)r+1$ and to suppose $V(G)=[n]$ for notational simplicity.

The partite construction produces a sequence of so-called 
{\it pictures}, which in the present case are just $n$-partite $k$-uniform hypergraphs.
It is customary to draw the vertex classes of pictures, which are called {\it music 
lines}, horizontally; the hypergraph $G$ is then drawn vertically next to the picture
(see Figure~\ref{fig45}) so that a bijective correspondence between music lines and the
vertices of $G$ is set  up. In other words, the projection $\psi$ `to the left side' 
is a hypergraph homomorphism from the picture to~$G$. Due to $V(G)=[n]$ we can 
speak of the first, second, etc.\ music line of a picture. 

\usetikzlibrary {arrows.meta,bending,positioning}

	\begin{figure}[h!]
	
	\begin{tikzpicture}

			\foreach \j in {1,...,5}
	{	\draw (0,\j*.5)--(8.8,\j*.5);
		\foreach \i in {-4,...,10}
		{\coordinate (a\i\j) at (\i*.8,\j*.5);
		}
}
\qedge{(a13)}{(a14)}{(a15)}{5pt}{.6pt}{red!80!black}{red, opacity=.3};
\qedge{(a22)}{(a24)}{(a25)}{5pt}{.6pt}{red!80!black}{red, opacity=.3};
\qedge{(a31)}{(a34)}{(a35)}{5pt}{.6pt}{red!80!black}{red, opacity=.3};
\qedge{(a42)}{(a44)}{(a45)}{5pt}{.6pt}{red!80!black}{red, opacity=.3};
\qedge{(a51)}{(a54)}{(a55)}{5pt}{.6pt}{red!80!black}{red, opacity=.3};
\qedge{(a61)}{(a64)}{(a65)}{5pt}{.6pt}{red!80!black}{red, opacity=.3};
\qedge{(a72)}{(a73)}{(a74)}{5pt}{.6pt}{red!80!black}{red, opacity=.3};
\qedge{(a81)}{(a83)}{(a84)}{5pt}{.6pt}{red!80!black}{red, opacity=.3};
\qedge{(a91)}{(a93)}{(a94)}{5pt}{.6pt}{red!80!black}{red, opacity=.3};
\qedge{(a101)}{(a102)}{(a103)}{5pt}{.6pt}{red!80!black}{red, opacity=.3};

\draw [Stealth-, thick] (-2.6,1.5)--(-.5,1.5);
\node at (-1.5,1.7) {$\psi$};

		\foreach \i in {a13,a14,a15,a22,a24,a25,a31,a34,a35,a42,a43,a45,a51,a53,a55,a61,a62,a65,a74,a73,a72,a84,a83,a81,a92,a94,a91,a103,a102,a101} \fill [black] (\i) circle (1.5pt);
		
		\qedge{(a-41)}{(a-42)}{(a-45)}{5pt}{.6pt}{violet!80!black}{violet, opacity=0.3};
		\foreach \i in {1,...,5} \fill (a-4\i) circle (1.5pt);
		
		\node [violet] at (-4.6,1.5) {$G=K^{(3)}_5$};
		
	\end{tikzpicture}
	\caption{Picture zero for $k=3$, $r=2$, and $n=5$.}\label{fig45}
	
	\end{figure}
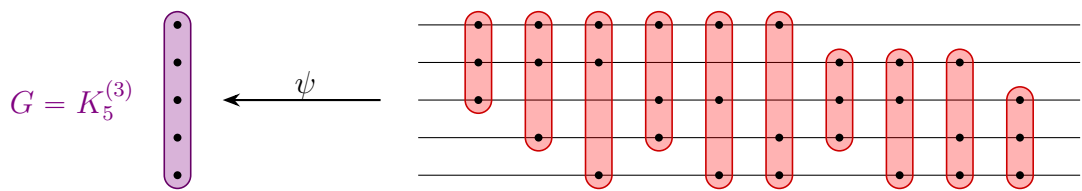
 
Every partite construction is initialised with its {\it picture zero}, typically 
denoted by $\Pi_0$. In the case at hand, picture zero is a matching consisting of 
$e(G)=\binom nk$ edges. Their vertices are to be positioned on the music lines in 
such a way that to every edge $e$ of $G$ there corresponds a unique edge of $\Pi_0$ 
projected to $e$ by $\psi$ (see Figure~\ref{fig45}). Clearly $\Pi_0$ has infinite girth and, 
in particular, it is linear. 

We shall now construct iteratively a sequence of linear pictures 
$\Pi_1, \dots, \Pi_n$. The last picture~$\Pi_n$ is going to be the desired 
linear $k$-uniform hypergraph, whose chromatic number exceeds $r$. In general, 
the construction of $\Pi_i$ will `process' the $i^{\mathrm{th}}$ music line. 

Let us first explain the formation of $\Pi_1$ (see Figure~\ref{fig46a}). If the 
first music line of $\Pi_0$ has~$n_1$ vertices, then the first 
music line of $\Pi_1$ has $r(n_1-1)+1$ vertices. 
Moreover, each set of~$n_1$ vertices from this music line 
is extended to its own copy of $\Pi_0$. These copies of $\Pi_0$ are to be drawn as 
disjointly as possible, so that copies corresponding to different sets intersect 
only on the first music line of $\Pi_1$. 
This ensures that distinct edges of $\Pi_1$ can only intersect on the first music line 
and, therefore, $\Pi_1$ is indeed linear. Notice that for every $r$-colouring 
of $\Pi_1$ there are $n_1$ vertices on the first music line receiving the same colour; 
the copy of $\Pi_0$ attached to these $n_1$ vertices has the property that its first 
music line is monochromatic. 

	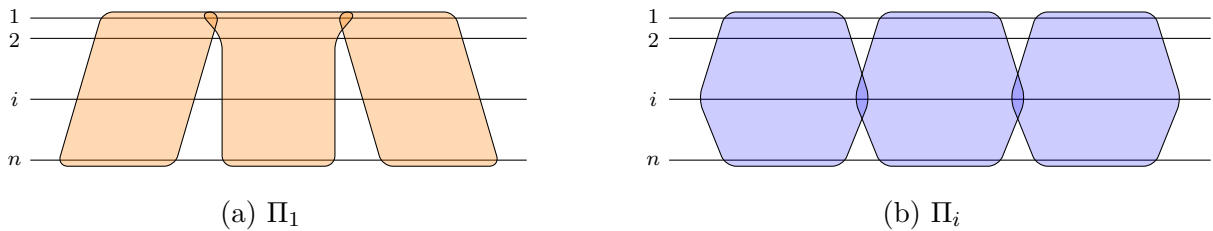
\begin{figure}[h!]
	
\begin{subfigure}[b]{.47\textwidth}
	\centering
		\begin{tikzpicture}

\def\h{.27}
\def\w{.3}

\draw (-4*\w,8*\h)--(18*\w,8*\h);
\draw (-4*\w,7*\h)--(18*\w,7*\h);
\draw (-4*\w,4*\h)--(18*\w,4*\h);
\draw (-4*\w,1*\h)--(18*\w,1*\h);

\fill [ orange, opacity=.3,rounded corners] (3.5*\w,8.3*\h) --(4.5*\w,7*\h)--(4.5*\w,0.7*\h)--(9.5*\w,0.7*\h)--(9.5*\w,7*\h)-- (10.5*\w,8.3*\h) -- cycle;
\fill [orange, opacity=.3, rounded corners] (16.8*\w,0.7*\h)--(11.6*\w,0.7*\h)--(9.6*\w,8.3*\h)--(14.8*\w,8.3*\h)--cycle;
\fill [orange, opacity=.3,rounded corners] (2.4*\w,0.7*\h)--(-2.8*\w,0.7*\h)--(-.8*\w,8.3*\h)--(4.4*\w,8.3*\h)--cycle;

\draw [rounded corners] (3.5*\w,8.3*\h) --(4.5*\w,7*\h)--(4.5*\w,0.7*\h)--(9.5*\w,0.7*\h)--(9.5*\w,7*\h)-- (10.5*\w,8.3*\h) -- cycle;
\draw [rounded corners] (16.8*\w,0.7*\h)--(11.6*\w,0.7*\h)--(9.6*\w,8.3*\h)--(14.8*\w,8.3*\h)--cycle;
\draw [rounded corners] (2.4*\w,0.7*\h)--(-2.8*\w,0.7*\h)--(-.8*\w,8.3*\h)--(4.4*\w,8.3*\h)--cycle;

\node at (-4.7*\w,8.1*\h) {\tiny $1$};
\node at (-4.7*\w,6.9*\h) {\tiny $2$};
\node at (-4.7*\w,4*\h) {\tiny $i$};
\node at (-4.7*\w,1*\h) {\tiny $n$};
		
	\end{tikzpicture}
\caption{$\Pi_1$}\label{fig46a}
\end{subfigure}
\hfill
\begin{subfigure}[b]{.47\textwidth}
	\centering
	\begin{tikzpicture}
		
		\def\h{.27}
		\def\w{.3}
		
		\draw (-12*\w,8*\h)--(12*\w,8*\h);
		\draw (-12*\w,7*\h)--(12*\w,7*\h);
		\draw (-12*\w,4*\h)--(12*\w,4*\h);
		\draw (-12*\w,1*\h)--(12*\w,1*\h);
		
	\fill [blue, opacity=.2,rounded corners] (-2.6*\w,0.7*\h)--(2.6*\w,0.7*\h)--(3.8*\w,3.6*\w)--(2.6*\w,8.3*\h)--(-2.6*\w,8.3*\h)--(-3.8*\w,3.6*\w)--cycle;
	\fill [blue, opacity=.2,rounded corners] (9.5*\w,0.7*\h)--(4.3*\w,0.7*\h)--(3.1*\w,3.6*\w)--(4.3*\w,8.3*\h)--(9.5*\w,8.3*\h)--(10.7*\w,3.6*\w)--cycle;	
	\fill [blue, opacity=.2,rounded corners] (-9.5*\w,0.7*\h)--(-4.3*\w,0.7*\h)--(-3.1*\w,3.6*\w)--(-4.3*\w,8.3*\h)--(-9.5*\w,8.3*\h)--(-10.7*\w,3.6*\w)--cycle;
		
	\draw [rounded corners] (-2.6*\w,0.7*\h)--(2.6*\w,0.7*\h)--(3.8*\w,3.6*\w)--(2.6*\w,8.3*\h)--(-2.6*\w,8.3*\h)--(-3.8*\w,3.6*\w)--cycle;
	\draw [rounded corners] (9.5*\w,0.7*\h)--(4.3*\w,0.7*\h)--(3.1*\w,3.6*\w)--(4.3*\w,8.3*\h)--(9.5*\w,8.3*\h)--(10.7*\w,3.6*\w)--cycle;
	\draw [rounded corners] (-9.5*\w,0.7*\h)--(-4.3*\w,0.7*\h)--(-3.1*\w,3.6*\w)--(-4.3*\w,8.3*\h)--(-9.5*\w,8.3*\h)--(-10.7*\w,3.6*\w)--cycle;

	\node at (-12.7*\w,8.1*\h) {\tiny $1$};
	\node at (-12.7*\w,6.9*\h) {\tiny $2$};
	\node at (-12.7*\w,4*\h) {\tiny $i$};
	\node at (-12.7*\w,1*\h) {\tiny $n$};
		
	\end{tikzpicture}
	\caption{$\Pi_i$}\label{fig46b}
\end{subfigure}
	\caption{The recursive construction of $\Pi_1, \Pi_2, \dots, \Pi_n$.
	The orange and blue shapes indicate copies of $\Pi_0$ and $\Pi_{i-1}$, respectively.}
	\label{fig46}
	\end{figure}
 
Now suppose inductively that for some $i\in [n]$ the linear picture $\Pi_{i-1}$ 
has already been defined and that it has the following property: for 
every $r$-colouring of $\Pi_{i-1}$ there is a copy of~$\Pi_0$ each of whose 
first $i-1$ music lines is monochromatic (but different music lines may 
have different colours). Let the $i^\mathrm{th}$ music line of $\Pi_{i-1}$
have $n_i$ vertices. Then the $i^\mathrm{th}$ music line of $\Pi_i$ is 
constructed to have $(n_i-1)r+1$ vertices and every set consisting of $n_i$ 
of them is extended to its own copy of~$\Pi_{i-1}$ (see Figure~\ref{fig46b}). Again we 
perform these extensions as disjointly as possible, thereby guaranteeing that~$\Pi_i$
is again linear. For every $r$-colouring of~$\Pi_i$ there is a copy of~$\Pi_{i-1}$
whose~$i^{\mathrm{th}}$ music line is monochromatic; so by our above hypothesis there 
is a copy of~$\Pi_0$ whose first~$i$ music lines are monochromatic.

Ultimately we reach a final picture~$\Pi_n$, which is a linear $k$-uniform hypergraph.
For every $r$-colouring of $\Pi_n$ there is a copy $\widetilde{\Pi}_0$ of picture zero 
all of whose music lines are monochromatic. The $n$ colours we see on the music lines 
of $\widetilde{\Pi}_0$ correspond, via the projection $\psi$, to a vertex colouring 
of the vertical hypergraph $G$. Because of $\chi(G)>r$ some edge of $G$ needs to 
be monochromatic with respect to this auxiliary colouring. The corresponding 
edge of $\widetilde{\Pi}_0$ is the desired monochromatic edge of $\Pi_n$. 
Thus we have indeed $\chi(\Pi_n)>r$. We leave it to the reader's curiosity to check 
that~$\Pi_n$ is not only linear, but also free of $3$-cycles (this fact is not going 
to used later).  

Before generalising this argument to larger girth, we would like to offer some brief 
remarks. The projection argument in the last paragraph essentially establishes the 
implication 
\[
	\chi(G)>r \,\,\, \Longrightarrow \chi(\Pi_n)>r\,.
\]
In principle, any other $k$-uniform hypergraph $G'$ with $\chi(G')>r$ could have
been employed vertically; the corresponding picture zero would again have~$e(G')$ 
edges, so that some of its naturally induced $k$-partite $k$-uniform subhypergraphs 
were edgeless. The freedom to do something smart vertically adds considerably to 
the power and flexibility of the partite construction method. It is often exploited 
very successfully in the current research literature (e.g., by Hubi\v{c}ka and 
Jarik~\cite{HN19}), but for the purposes of the current subsection there is no need 
for clever vertical decisions.  

Horizontally we appealed to the box principle when arguing that for 
every $r$-colouring of~$\Pi_i$ there is a copy of~$\Pi_{i-1}$ whose~$i^{\mathrm{th}}$ 
music line is monochromatic. We can view this step also as follows. 
The $i^{\mathrm{th}}$ music line
of $\Pi_{i-1}$ is essentially the same as an $n_i$-uniform edge. The $n_i$-uniform 
clique~$K^{(n_i)}_{(r-1)n_i+1}$ is our standard example of an $n_i$-uniform hypergraph
whose chromatic number exceeds $r$, and the copies of~$\Pi_{i-1}$ in Figure~\ref{fig46b}
should be thought of as corresponding to its edges. Any other choice of a $n_i$-uniform 
hypergraph that fails to be $r$-colourable would work here as well. 
This possibility certainly needs 
to be exploited when proving Theorem~\ref{thm:EH}, because as long as two copies 
of~$\Pi_1$ can intersect in more than one vertex it is difficult to avoid four-cycles 
in~$\Pi_2$. 

Having thus laid a solid foundation we can prove Theorem~\ref{thm:EH} rather easily. 
Fix $r\ge 2$ and assume, as an induction hypothesis, that for some $g\ge 2$ we already 
have a sequence of hypergraphs $(H^{(k)}_g)_{k\ge 2}$ such that $H^{(k)}_g$ 
is $k$-uniform, $\gth(H^{(k)}_g)>g$, and $\chi(H^{(k)}_g)>r$. 
Given any integer $k\ge 2$ we need to 
construct an appropriate hypergraph $H^{(k)}_{g+1}$. To this end we set $n=(k-1)r+1$ 
and run a partite construction, thereby generating a sequence of pictures 
$\Pi_0, \Pi_1, \dots, \Pi_n$.

We start with the same picture zero $\Pi_0$ as before (see Figure~\ref{fig45}).
Now suppose that for some positive integer $i\le n$ we have already obtained 
the picture $\Pi_{i-1}$ with $\gth(\Pi_{i-1})>g+1$. Let $n_i$ denote the number of 
vertices on the $i^\mathrm{th}$ music line of $\Pi_{i-1}$. 
Draw the hypergraph~$H^{(n_i)}_g$ horizontally and extend each of its edges to a 
separate copy of $\Pi_{i-1}$, thus obtaining the next picture $\Pi_i$. 
For clarity we point out that there are $|V(H^{(n_i)}_g)|$ vertices on 
the~$i^{\mathrm{th}}$ music line of~$\Pi_i$ and that 
$e(\Pi_i)=e(\Pi_{i-1})\cdot e(H^{(n_i)}_g)$. It is important to ensure that 
our $e(H^{(n_i)}_g)$ so-called {\it standard copies} of $\Pi_{i-1}$ (visualised 
by blue shapes in Figure~\ref{fig46b}) are only intersecting each other on 
the~$i^{\mathrm{th}}$ music line. 

We contend that $\gth(\Pi_i)>g+1$. Assume contrariwise that for some $n\in [2, g+1]$
there is an $n$-cycle $e_1v_1\dots e_nv_n$ in $\Pi_i$. For every $j\in\ZZ/n\ZZ$ let 
$f_j$ be the edge of $H^{(n_i)}_g$ whose extension led to the standard copy 
of~$\Pi_{i-1}$ containing $e_j$. By our disjointness requirement, if $f_j\ne f_{j+1}$, 
then $v_j\in f_j\cap f_{j+1}$. So unless the edges $f_1, \dots, f_n$ are identical, 
some of them form a cycle. Owing to $\gth(H^{(n_i)}_g)>g$ this shows that 
\begin{enumerate}[label=\nlabel]
	\item\label{it:pc1} either $f_1=\dots=f_n$;
	\item\label{it:pc2} or $f_1, \dots, f_n$ are distinct and $f_1v_1\dots f_nv_n$ 
	is a cycle in $H^{(n_i)}_g$.
\end{enumerate}  
But~\ref{it:pc1} contradicts $\gth(\Pi_{i-1})>g+1$ and~\ref{it:pc2} implies that 
all of $v_1, \dots, v_n$ are on the $i^{\mathrm{th}}$ music line of $\Pi_i$. Due to 
$v_1, v_n\in e_1$ it follows that $e_1$ intersects this music line at least twice,
which is absurd. We have thereby established $\gth(\Pi_i)>g+1$ and the partite 
construction goes on. 

As in the linear case we see that for every $r$-colouring of the last picture there 
is a copy of~$\Pi_0$ whose music lines are monochromatic, which in turn shows that
there is a monochromatic edge. This confirms $\chi(\Pi_n)>r$ and the proof of 
Theorem~\ref{thm:EH} by partite construction is complete. Another account of 
this argument can be found in the original source~\cite{NeRo79}. 

\subsection{A conjecture of Erd\H{o}s}\label{sec:EC}
We proceed with some results related to a famous problem of 
Erd\H{o}s~\cite{Erd68}.

\begin{conj}[Erd\H{o}s]\label{conj:E}
	Given any two natural numbers $g, r\ge 2$ there exists a natural number $k$ 
	such that every graph $G$ with $\chi(G)> k$ has a subgraph $F$ 
	with $\gth(F)>g$ and $\chi(F)> r$.
\end{conj}

The special case $g=3$ was solved in~\cite{Rodl77}, while for every $g\ge 4$ the 
conjecture is wide open. 

\begin{thm}[R\"{o}dl]\label{thm:rodl}
	Given $r\ge 2$ every graph whose chromatic number is sufficiently large 
	has a triangle-free subgraph whose chromatic number exceeds $r$. 
\end{thm}

The argument exploits that the chromatic number is submultiplicative. This was first 
observed by Zykov~\cite{Zykov}*{\rn{Teorema}~2} and can be proved using a product 
colouring. 

\begin{fact}[Zykov]\label{f:zykov}
	Let $G=(V, E)$ be a graph. If $E=\bigcup_{i\in I}E_i$, then 
		\[
		\pushQED{\qed} 
		\chi(G)\le \prod_{i\in I} \chi(V, E_i)\,. \qedhere
		\popQED
	\]
	\end{fact}

Now the idea of R\"{o}dl's proof is the following. Suppose that for some integer $n$
(that will later be allowed to grow) we consider a graph $G=(V, E)$ whose chromatic 
number is much bigger than $n$. Fix an arbitrary ordering $<$ of $V$, so that for 
every vertex $x\in V$ we can consider its {\it left neighbourhood} 
\[
	N_<(x)=\{y\in V\colon y<x \text{ and } xy\in E\}\,.
\]
There are two possibilities. Either
\begin{enumerate}[label=\nlabel]
	\item\label{it:nr1} $\chi(G[N_<(x)])\le n$ for every $x\in V$
	\item\label{it:nr2} or $\chi(G[N_<(x)])>n$ for some $x\in V$. 
\end{enumerate}

Let us first consider the case that~\ref{it:nr1} holds. Fix for 
every vertex $x\in V$ a proper $n$-colouring $f_x\colon N_<(x)\lra [n]$
of $G[N_<(x)]$. The sequence of colourings $(f_x)_{x\in V}$ can equivalently 
be described by a partition $E=\bigdcup_{i\in [n]} E_i$, where an edge $xy\in E$
with $y<x$ is put into a set $E_i$ if and only if $f_x(y)=i$. As the colourings $f_x$
are proper, the graphs $(V, E_i)$ are triangle-free. 
So if $\chi(V, E_i)>r$ holds for some $i\in [n]$, then we have found the desired 
subgraph of $G$; otherwise Fact~\ref{f:zykov} tells us $\chi(G)\le r^n$, so that 
the chromatic number of $G$ is `bounded'.

Intuitively the argument from the previous paragraph tells is that if $\chi(G)$
is sufficiently large, then only case~\ref{it:nr2} is relevant. But the same 
observation can then be applied to $G[N_<(x)]$ in place of $G$, thus starting 
an iteration. Given any number $m$ in advance, we can assume that $\chi(G)$ is 
so large that $m$ iteration steps are possible, which allows us to build a 
clique $K_m$ in $G$.
But clearly, if $m$ itself is chosen sufficiently large, then this clique contains 
a triangle-free subgraph whose chromatic number exceeds $r$. For further details 
on the proof of Theorem~\ref{thm:rodl} we refer to~\cite{Rodl77}.

The triangle-free subgraph provided by this proof is usually not induced. This is 
quite manifest in the `second case', where such a graph is found inside a big clique; 
but also if at some step along the iteration the first case occurs, the subgraph 
obtained after partitioning the edge set is typically non-induced. Nevertheless, 
it is natural to wonder whether, under some additional assumptions, even an induced 
triangle-free subgraph of large chromatic number can be found. 
For instance, Galvin and R\"{o}dl conjectured that it suffices to assume that, in 
addition to having extremely large chromatic number, the given graph is also 
$K_4$-free (see Jarik's graph theory 
textbook~\cite{jarik-book}*{p.293, Probl\'em S}), but this was 
refuted a couple of years ago in~\cite{CHMS}.    
 
\begin{thm}[Carbonero, Hompe, Moore \& Spirkl]\label{thm:CHMS}
	There are $K_4$-free graphs of arbitrarily large chromatic number all of 
	whose induced triangle-free subgraphs are $4$-colourable.
\end{thm}

\begin{proof}
	We start by orienting the graphs from Zykov's construction, which we saw 
	in~\S\ref{subsec:32}. This produces a sequence of digraphs $(D_2)_{r\ge 2}$,
	where $D_2$ consists of two vertices joined by an arc. If for some $r\ge 2$ 
	the digraph $D_r$ has just been constructed, we form $D_{r+1}$ as indicated 
	in Figure~\ref{fig41} and direct all `new' edges towards the vertices $x_T$. 
	We already know that the underlying graph of $D_r$ has chromatic number $r$. 
	Moreover, one checks easily that $D_r$ is acyclic and that for all 
	vertices $u, v\in V(D_r)$ there is at most one directed path from $u$ to $v$. 
	These are all properties of $D_r$ we need in the sequel. They guarantee that 
	the concatenation of two directed paths in $D_r$ is again a directed path, 
	i.e., there never arise problems due to repeated vertices. 
	
	Now let $D'_r$ be the digraph on $V(D_r)$ which has the following two kinds 
	of arcs:
		\begin{enumerate}
		\item[$(+)$] arcs $u\lra v$ such that in $D_r$ there is a 
			directed $u$-$v$-path whose length is congruent to~$+1$ modulo $3$; 
		\item[$(-)$] arcs $u\lra v$ such that in $D_r$ there is a 
			directed $v$-$u$-path whose length is congruent to~$-1$ modulo $3$. 
	\end{enumerate}
		The arcs of $D'_r$ corresponding to these two clauses are called {\it positive}
	and {\it negative}, respectively. It will turn out that the underlying graph 
	$G_r$ of $D'_r$ has the required properties. Since every arc of $D_r$ yields 
	a (positive) arc of $D'_r$, we have $\chi(G)\ge r$. Suppose next that 
	$u\lra v\lra w$ is a directed path in $D'_r$. By considering the corresponding 
	directed paths in $D_r$ one sees that 
		\begin{enumerate}
		\item[$\bullet$] if $u\lra v$, $v\lra w$ have the same sign, 
			then $w\lra u$ is an arc of $D'_r$ as well;
		\item[$\bullet$] and if $u\lra v$, $v\lra w$ have opposite signs,
			then $u\lra w$ cannot be an arc of $D'_r$.
	\end{enumerate}
		In particular, $D'_r$ contains no transitive tournament of order $3$
	and, therefore, $G_r$ is $K_4$-free. 
	
	Now let $H$ be a triangle-free induced subgraph of $G_r$. We need to exhibit 
	a proper $4$-colouring of $H$. Owing to Fact~\ref{f:zykov} it suffices to show 
	that the two subgraphs of $H$ corresponding to the positive 
	and negative arcs are bipartite. By the first of the above bullets, both of 
	these graphs have orientations without directed paths of length $2$, and it 
	is an easy exercise to show that all graphs admitting such orientations are 
	bipartite.
\end{proof}

This leaves the following problem open.

\begin{quest}[Davies]\label{quest:D}
	Do there exist $K_4$-free graphs of arbitrarily large chromatic number 
	all of whose induced triangle-free subgraphs are $3$-colourable?
\end{quest} 

Scott's research group~\cite{GIP} found a generalisation of Theorem~\ref{thm:CHMS} 
to arbitrary graphs instead of triangles. 

\begin{thm}[Gir\~{a}o, Illingworth, Powierski, Savery, Scott, Tamitegama \& Tan] 
	\label{thm:GIP}
	For every graph $F$ with at least one edge there exists a natural number~$c(F)$
	such that for every natural number $r$ there exists a graph $G$ with $\chi(G)>r$,
	$\omega(G)=\omega(F)$, 
	and the following property: all induced subgraphs of $G$ without induced 
	subgraphs isomorphic to $F$ have chromatic number at most $c(F)$. \qed
\end{thm}

For girth-enthusiasts the same authors also pose the following intriguing problem. 

\begin{conj}[Gir\~{a}o, Illingworth, Powierski, Savery, Scott, Tamitegama \& Tan]
	If $F$ is not a forest, then Theorem~\ref{thm:GIP} remains valid if we replace 
	the demand $\omega(G)=\omega(F)$ by $\gth(G)=\gth(F)$.
\end{conj}

Moreover, there is an optimistic conjecture of Jarik that would yield a positive 
answer to Question~\ref{quest:D}. 

\begin{conj}[Jarik]
	Theorem~\ref{thm:GIP} holds for $c(F)=\chi(F)$. 
\end{conj}

Currently, it is not even known whether $c(F)$ can be bounded by a function 
of $\chi(F)$. We conclude this subsection with a result of Erd\H{o}s, Galvin, 
and Hajnal~\cite{EGH}*{Theorem~10.8}, which implies that the natural 
generalisation of Conjecture~\ref{conj:E} to $3$-uniform hypergraphs is false. 
Its proof is somewhat similar to Tutte's construction we encountered 
in~\S\ref{subsec:32}.

\begin{thm}[Erd\H{o}s, Galvin \& Hajnal]
	For every natural number $r$ there exists a $3$-uniform hypergraph $H$ with 
	$\chi(H)\ge r$ such that every linear subhypergraph of $H$ is $2$-colourable. 
\end{thm}

\begin{proof} 
	Arguing by induction on $r$ we assume that such a hypergraph $H$ exists 
	for some $r\in\NN$ and explain how to construct an example for $r+1$. To this end 
	we take a set $Y$ of $r+1$ vertices and to every pair $yy'\in Y^{(2)}$ we 
	assign its own copy $H_{yy'}$ of $H$, so that $Y$ and all vertex sets 
	$V(H_{yy'})$ are disjoint. As indicated in Figure~\ref{fig47}, we also add 
	all edges of the form~$yy'z$, where $yy'\in Y^{(2)}$ and $z\in V(H_{yy'})$.
	
	\usetikzlibrary {arrows.meta,bending,positioning}

	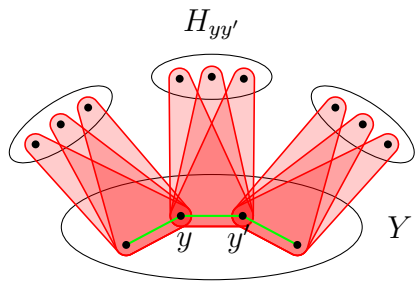
\begin{figure}[h!]
	
	\begin{tikzpicture}

			\foreach \i in {-6,...,6}
			{\coordinate (a\i) at (90+\i*7:3.5);
			\coordinate (b\i) at (90+\i*7:1.7);}

\draw (0,1.5) ellipse (2cm and .7cm);
	\draw [rotate around={-34:(a-5)}] (a-5) ellipse (.8cm and .3cm);
		\draw [rotate around={34:(a5)}] (a5) ellipse (.8cm and .3cm);
			\draw  (a0) ellipse (.8cm and .3cm);
		
		\qedge{(b-2)}{(a-6)}{(b-6)}{4pt}{.6pt}{red}{red, opacity=.2};
			\qedge{(b-2)}{(a-5)}{(b-6)}{4pt}{.6pt}{red}{red, opacity=.2};
				\qedge{(b-2)}{(a-4)}{(b-6)}{4pt}{.6pt}{red}{red, opacity=.2};
				
					\qedge{(b6)}{(a6)}{(b2)}{4pt}{.6pt}{red}{red, opacity=.2};
				\qedge{(b6)}{(a5)}{(b2)}{4pt}{.6pt}{red}{red, opacity=.2};
				\qedge{(b6)}{(a4)}{(b2)}{4pt}{.6pt}{red}{red, opacity=.2};
					\qedge{(b2)}{(a-1)}{(b-2)}{4pt}{.6pt}{red}{red, opacity=.2};
				\qedge{(b2)}{(a0)}{(b-2)}{4pt}{.6pt}{red}{red, opacity=.2};
				\qedge{(b2)}{(a1)}{(b-2)}{4pt}{.6pt}{red}{red, opacity=.2};

		\draw [thick, green] (b-6)--(b-2)--(b2)--(b6);
		
			\foreach \i in {a-6,a-5,a-4,a-1,a0,a1,a4,a5,a6,b-6,b-2,b2,b6} \fill  (\i) circle (1.5pt);
			
			\node at ($(b-2)+(-.05,-.3)$ ){$y'$};
			\node at ($(b2)+(.05,-.35)$ ){$y$};
			\node at (2.5,1.5) {$Y$};
			\node at (0,4.2) {$H_{yy'}$};

	\end{tikzpicture}
	\caption{Construction of $H_\star$.}\label{fig47}
	
	\end{figure}
 	
	For every proper $r$-colouring of the resulting hypergraph $H_\star$ there 
	need to exist two distinct vertices $y, y'\in Y$ of the same colour. 
	But this colour is then unavailable for the vertices of~$H_{yy'}$, so that
	$\chi(H)\ge r$ yields a contradiction; this proves $\chi(H_\star)\ge r+1$. 
	
	Now let $H'$ be any linear spanning subhypergraph of $H_\star$. In order to 
	find the desired proper $2$-colouring of $H'$ we start by assigning the 
	colour {\it blue} to all vertices in $Y$ and the colour {\it yellow} to all 
	vertices $z$ for which $Y\cup \{z\}$ spans an edge of $H'$. Since $H'$ 
	is linear, there can be at most one such vertex $z\in V(H_{yy'})$ for 
	each pair $yy'\in Y^{(2)}$. By our induction hypothesis there are proper  
	blue/yellow colourings of the sets $V(H_{yy'})$ and by switching colours 
	if necessary we can ensure that the vertices which have already been coloured 
	yellow create no conflicts. 
\end{proof}

\subsection{Vertex colourings and Ramsey theory}\label{subsec:35}
Returning to the definition of the chromatic number we can also investigate what 
happens when instead of demanding only a monochromatic edge we want to find a 
larger monochromatic substructure.  
The partition symbols introduced by Erd\H{o}s and Rado~\cite{ER56} provide a 
systematic and concise notation for the kind of statement we have in mind. 

For instance, given a graph or hypergraph $H$ and $r\in\NN$ a lower bound of 
the form $\chi(H)>r$ is written in the form 
\begin{equation}\label{eq:1747}
	H \lra (e)^v_r\,,
\end{equation}
where $v$ and $e$ abbreviate the words `vertex' and `edge', respectively. The 
general pattern is that 
\[
	\text{source} \lra (\text{target})^{A}_{r}
\]
indicates the following statement: If all subobjects of the source 
symbolised by $A$ are coloured with $r$ colours, then some subobject of the 
source isomorphic to the target is monochromatic in the sense that all its 
copies of $A$ have the same colour. Generalising~\eqref{eq:1747} we may thus 
consider for any two graphs (or $k$-uniform hypergraphs) $F$ and $H$ and every
number of colours~$r$ the statement
\begin{equation}\label{eq:1751}
	H\lra (F)^v_r\,.
\end{equation}
It means that for every colouring $f\colon V(H)\lra [r]$ there is an induced 
subgraph of $H$ isomorphic to $F$ whose vertices have the same colour. The negation 
of this statement is indicated by crossing out the arrow. 
E.g., an upper bound $\chi(H)\le r$ can be expressed by $H\nlra (e)^v_r$. 

In connection with~\eqref{eq:1751} the first question one may ask is whether 
given a graph $F$ and~${r\in\NN}$ there always exists a graph $H$ such 
that $H\lra (F)^v_r$ holds. 
This was first settled by Folkman~\cite{Folk}, whose construction was 
called a ``gem of combinatorial ingenuity'' in a review by Graham. Nevertheless,
we resist the temptation of repeating the argument here, because later 
Jarik and R\"{o}dl~\cite{NR76b} found an even more beautiful trick, which gives 
this result almost for free: they take a linear $|V(F)|$-uniform hypergraph $G$ 
with $\chi(G)>r$ and replace 
the edges of~$G$ by copies of~$F$, thereby generating the desired graph $H$. 

This construction has a further interesting property. The system $\ccH$ of all 
copies of $F$ in $H$ corresponding to the edges of $G$ satisfies, in an obvious sense, 
the partition relation $\ccH\lra(F)^v_r$. Moreover, any two distinct copies of~$F$ 
in~$\ccH$ are either disjoint or they intersect in a single vertex.\footnote[1]{The 
reason for introducing $\ccH$ here is that $H$ can also contain other, unintended
copies of $F$. Their possible intersection patterns depend on the structure of $F$,
but it does not seem worthwhile to work out further details.} 
One can gain even more control over the system $\ccH$ by starting with a 
hypergraph $G$ of large girth (cf.\ Theorem~\ref{thm:EH}). In this manner we arrive
at the following conclusion (see~\cite{NR76b}). 

\begin{thm}\label{thm:1535}
	For every graph $F$ and all $r, n\in \NN$ there exists 
	a graph~$H$ together with a system $\ccH$ of induced copies of $F$ in $H$
	such that 
	\begin{enumerate}[label=\rmlabel]
		\item $\ccH\lra (F)^v_r$ and    
  		\item for every $\ccN\subseteq \ccH$ with $|\ccN|\le n$ there exists an 
			enumeration $\ccN=\{F_1, \ldots, F_{|\ccN|}\}$ with the property that for 
			every $j\in [2, |\ccN|]$ the sets $\bigcup_{i<j}V(F_i)$ and $V(F_j)$
			have at most one vertex in common. \qed
	\end{enumerate}
\end{thm}

Clearly, the same argument works for hypergraphs instead of graphs as well. 
In the special case $F=K_2$ Theorem~\ref{thm:1535} reduces to Erd\H{o}s's 
Theorem~\ref{thm:41}, and both are optimal in the same sense. 
That is, Theorem~\ref{thm:1535} describes all configurations of 
copies of~$F$ that need to be present in systems $\ccH$ satisfying $\ccH\lra (F)^v_r$
for sufficiently large $r$. For a precise statement along these lines we refer 
to the work of Daskin, Hoshen, Krivelevich, and Zhukovskii~\cite{DHKZ}.
   
\subsection{Infinite graphs}\label{subsec:36}
Some unexpected new phenomena arise when one tries to generalise Theorem~\ref{thm:41}
to infinite graphs. Our discussion presupposes some elementary background in set theory 
as it can be found, e.g., in the early chapters of the texts by Jech~\cite{Jech} or 
Kunen~\cite{Kunen}. Sometimes we shall mention certain partition relations involving 
cardinal numbers. Standard references on this topic are the book by Erd\H{o}s,
Hajnal, M\'{a}t\'{e}, and Rado~\cite{EHMR} and the more recent survey by Hajnal and 
Larson~\cite{HL} in the handbook of set theory.

The chromatic number of an infinite graph can be any finite or infinite cardinal. 
However, since cycles are necessarily finite, the girth of an infinite 
graph is still in $\NN_{\ge 3}\dcup\{\infty\}$. It follows immediately from 
Theorem~\ref{thm:41} by taking disjoint unions that graphs of countably infinite 
chromatic number can have arbitrarily large girth. 
Thus the first `new' question is whether 
triangle-free graphs with uncountable chromatic number exist. 
Erd\H{o}s and Rado~\cite{ER59} gave an affirmative answer. 
Shortly afterwards, they realised that an idea of Specker~\cite{Speck}
yields a different construction~\cite{ER60} with the optimal quantitative 
dependence between~$|V(G)|$ and~$\chi(G)$.

\begin{thm}[Erd\H{o}s \& Rado]\label{thm:er60}
	For every infinite cardinal $\kappa$ there exists a triangle-free graph on $\kappa$
	vertices with chromatic number $\kappa$. 
\end{thm}     

\begin{proof}
	Let $G$ be the graph on 
	$\kappa^{(3)}=\{\{\alpha_0, \alpha_1, \alpha_2\}\colon 
	\alpha_0<\alpha_1<\alpha_2<\kappa\}$ which has for every increasing sequence 
	$\alpha_0<\alpha_1<\alpha_2<\alpha_3<\alpha_4<\alpha_5<\kappa$ an edge 
	from $\{\alpha_0, \alpha_1, \alpha_3\}$ to~$\{\alpha_2, \alpha_4, \alpha_5\}$.
	A short finitary consideration discloses that $G$ contains no triangles. 
	
	Now assume for the 
	sake of contradiction that for some cardinal $\chi<\kappa$ there is a proper 
	{$\chi$-colouring} $f$ of $G$. By the uniform construction of our graphs 
	we can suppose $\kappa=\chi^+$ if $\chi$ is infinite. 
	This allows us to assign to every 
	pair of ordinals $\{\alpha_0, \alpha_1\}$ with $\alpha_0<\alpha_1<\kappa$
	an auxiliary colour $g(\{\alpha_0, \alpha_1\})<\chi$ such 
	that $f(\{\alpha_0, \alpha_1, \alpha\})=g(\{\alpha_0, \alpha_1\})$ holds 
	for arbitrarily large ordinals $\alpha<\kappa$. Iterating this once more we find a 
	map $h\colon \kappa\lra\chi$ such that for each $\alpha_0<\kappa$ there 
	are unboundedly many ordinals $\alpha<\kappa$ 
	with $g(\{\alpha_0, \alpha\})=h(\alpha_0)$. Finally, there is a 
	colour $\chi_\star<\chi$ such that $h(\alpha)=\chi_\star$ holds for arbitrarily
	large $\alpha<\kappa$. 
	
	Unravelling these stipulations, we find successively 
	six ordinals $\alpha_0<\dots<\alpha_5<\kappa$ such that 
		\[
		\chi_\star
		=
		h(\alpha_0)
		=
		g(\{\alpha_0, \alpha_1\})
		=
		h(\alpha_2)
		=
		f(\{\alpha_0, \alpha_1, \alpha_3\})
		=
		g(\{\alpha_2, \alpha_4\})
		=
		f(\{\alpha_2, \alpha_4, \alpha_5\})\,,
	\]
		which means that the edge from $\{\alpha_0, \alpha_1, \alpha_3\}$
	to $\{\alpha_2, \alpha_4, \alpha_5\}$ is monochromatic.
\end{proof}

A few years later, Erd\H{o}s and Hajnal~\cite{EH64}*{Theorem 7} extended this 
result to larger odd cycles. 

\begin{thm}[Erd\H{o}s and Hajnal]\label{thm:EHodd}
	For every positive integer $k$ and every cardinal $\kappa$ there is a 
	$\{C_3, C_5, \dots, C_{2k+1}\}$-free graph $G$ with $\chi(G)\ge \kappa$.
\end{thm}  

Their original proof used so-called shift graphs, which are defined as follows. 
Given a cardinal $\lambda$ and an integer $k\ge 2$ the {\it shift graph} 
$\Sh_k(\lambda)$ has vertex set $\lambda^{(k)}$ and for all ordinals 
$\alpha_0<\dots<\alpha_k<\lambda$ it has an edge from $\{\alpha_0, \dots, \alpha_{k-1}\}$
to $\{\alpha_1, \dots, \alpha_k\}$. It is a finitary matter to check that 
$\Sh_{k+1}(\lambda)$ is always $\{C_3, C_5, \dots, C_{2k+1}\}$-free. Moreover, 
if $\lambda$ is chosen so large that the partition relation 
$\lambda\lra(k+2)^{k+1}_\kappa$
holds, then $\chi(\Sh_{k+1}(\lambda))>\kappa$. 
This argument yields Theorem~\ref{thm:EHodd} with an iterated exponential dependence 
between $\chi(G)$ and $|V(G)|$. 
Shortly afterwards Erd\H{o}s and Hajnal~\cite{EH66}*{Theorem~7.4} found 
a different construction achieving $|V(G)|=\chi(G)$. Similar to the proof of 
Theorem~\ref{thm:er60}, these graphs have vertex set~$\kappa^{(2k^2+1)}$ and there 
is a rule assigning an edge to each increasing sequence of $4k^2+2$ ordinals below 
$\kappa$. In general, graphs on a set of the shape $\kappa^{(m)}$ whose edges are 
determined by certain order patterns are called {\it type graphs}. They have turned 
out to be useful in many other contexts as well, see e.g.\ \cites{Pr86, KS05}. Their 
finite counterparts appear prominently in some of Jarik's and R\"odl's early work on 
structural Ramsey theory~\cite{NR76}; finite type graphs keep being used 
(e.g.~\cite{pisier}) and investigated (e.g.\ \cite{Avart}) until today. 

Concerning even cycles, Erd\H{o}s thought for a long time that graphs of uncountable 
chromatic number and girth $5$ exist and merely awaited their discovery. Thus he was
quite surprised when together with Hajnal~\cite{EH66}*{Corollary~5.6} he proved that,
actually, the chromatic number of $C_4$-free graphs is always at most countable---the 
natural generalisation of Theorem~\ref{thm:41} to infinite graphs is false. In fact, 
they obtained the following much stronger statement. 

\begin{thm}[Erd\H{o}s \& Hajnal]\label{thm:EHC4}
	For every natural number $n$ every graph $G$ of uncountable chromatic number 
	contains the bipartite graph $K_{n, \aleph_1}$.
\end{thm}

\begin{proof}
	Arguing indirectly we consider for fixed $n$ a counterexample $G$ such that 
	$\kappa=|V(G)|$ is minimal. 
	Call a subset $M\subseteq V(G)$ {\it closed} if there is no vertex 
	$x\in V(G)\sm M$ with at least~$n$ neighbours in $G$.	
	Due to $K_{n, \aleph_1}\not\subseteq G$ every set $X\subseteq V(G)$
	has a closed superset $M$ with $|M|\le |X|+\aleph_0$. 
	This allows us to express $V(G)$ as a union of a continuous increasing chain 
	$\langle M_i\colon i<\cofi(\kappa)\rangle$ of closed sets $M_i$ with $|M_i|<\kappa$. 
	
	We shall construct inductively an increasing chain 
	$\langle f_i\colon i<\cofi(\kappa)\rangle$
	of proper colourings $f_i\colon M_i\lra\omega$ of the graphs $G[M_i]$. Only the 
	successor step is interesting. So suppose that for some $i<\cofi(\kappa)$ we have 
	just selected $f_i$. By the minimality of $\kappa$, there is a proper 
	$\omega$-colouring~$g$ of $G[M_{i+1}\sm M_i]$. 
	Let $\omega=\bigdcup_{m<\omega} A_m$ be a partition 
	of $\omega$ into infinitely many sets of size $n$. 
	Now for every vertex $x\in M_{i+1}$ there is 
	a free colour $f_{i+1}(x)\in A_{g(x)}$, because~$M_i$ is closed. 
	Thus the desired extension 
	$f_{i+1}\supseteq f_i$ does indeed exist.
\end{proof}

Despite the fact that Theorem~\ref{thm:41} does not extend to the transfinite world, 
we can still ponder the same question that motivated us in~\S\ref{subsec:41}. Which 
graphs $F$ appear in all graphs of uncountable chromatic number? 
For finite graphs $F$, the results we have seen so far yield a complete solution. 
By Theorem~\ref{thm:EHC4} all finite bipartite graphs $F$ have this property. On the 
other hand, each non-bipartite graph contains an odd cycle and Theorem~\ref{thm:EHodd} 
yields a negative answer. We summarise this paragraph as follows. 

\vbox{  
\begin{cor}\label{cor:bip}
	For every finite graph $F$, the following statements are equivalent.
		\begin{enumerate}[label=\rmlabel]
			\item $F$ is bipartite.
			\item The chromatic number of every $F$-free graph is at most $\aleph_0$.
			\item There is an absolute bound on the chromatic number of $F$-free 
				graphs. \qed
		\end{enumerate}
\end{cor}
}

There is a substantial body of work on the possibilities for the family of finite 
subgraphs of a graph with uncountable chromatic number. Referring the interested 
reader to a survey by Komj\'ath~\cite{kom-survey} we will only focus on one specific 
result here (see~\cite{EHS}*{Theorem 3} or Thomassen~\cite{Tho83} for an 
alternative proof). 

\begin{thm}[Erd\H{o}s, Hajnal \& Shelah]\label{thm:EHS}
	Every graph of uncountable chromatic number contains odd cycles of all 
	sufficiently large lengths. 
\end{thm}

\begin{proof}	Without loss of generality we can assume that the graph $G$ under consideration 
	is connected. Let $x\in V(G)$ be arbitrary. For each $n<\omega$ let $D_n$ be the 
	set of vertices at distance~$n$ from~$x$. Since $V(G)=\bigdcup_{n<\omega} D_n$,
	there exists some $n<\omega$ such that the chromatic number of the induced 
	subgraph $G[D_n]$ is uncountable. For every edge $uv$ connecting two vertices 
	in $D_n$ there is a $u$-$v$-path $P_{uv}$ whose inner vertices are not in $D_n$
	and whose length is some even number $2m_{uv}$ with $m_{uv}\le n$. Indeed, such 
	a path can be found by going from $u$ to $x$ in $n$ steps, in another $n$ steps
	to $v$, and removing all detours. By Fact~\ref{f:zykov} there exists some $m\le n$
	such that the spanning subgraph $H$ of $G[D_n]$ whose edges $uv$ satisfy $m_{uv}=m$
	has uncountable chromatic number. For every $k\ge 2$ Theorem~\ref{thm:EHC4} yields 
	a copy of $C_{2k}$ in~$H$. Replacing one edge of such a cycle by a path of 
	length $2m$ we obtain $C_{2m+2k-1}\subseteq G$.  
\end{proof}

\begin{cor}
	For all graphs $G$, $G'$ of uncountable chromatic number there is a 
	finite graph~$F$
	with $\chi(F)=3$ such that both $G$ and $G'$ have subgraphs isomorphic to $F$. 
\end{cor}

\begin{proof}
	Every sufficiently large odd cycle $F$ has this property. 
\end{proof}

As noted in~\cite{kom-survey} it is unknown whether this holds for $4$ instead of $3$ as well.

\begin{quest}
	Is it true that any two graphs of uncountable chromatic number have a finite
	subgraph of chromatic number four in common?
\end{quest}

\subsection{Obligatory hypergraphs}\label{subsec:37}
Much less is known about the analogous questions for hypergraphs. 
For concreteness we shall only consider the $3$-uniform case here. 
The subject begins with an unfortunate oversight, which caused
Erd\H{o}s and Hajnal to believe for a while that no $3$-uniform hypergraph
of uncountable chromatic number could be 
linear\footnote[1]{In~\cite{EH66}*{Theorem~12.1} the assumption $\alpha=\beta^+$ 
is missing.}.
The argument they had in mind was supposed to be similar to the proof of 
Theorem~\ref{thm:EHC4}. However, it only shows that linear $3$-uniform hypergraphs 
on $\aleph_1$ vertices are indeed $\aleph_0$-colourable. In joint work with 
Rothschild~\cite{EHRoth}*{Theorem~2} they then found the following counterexample. 
Set $\lambda=(2^\omega)^+$ and consider the $3$-uniform hypergraph $H$ 
on $\lambda^{(2)}$ whose edges are all triples of the form 
$\bigl\{\{\alpha, \beta\}, \{\alpha, \gamma\}, \{\beta, \gamma\}\bigr\}$, 
where $\alpha<\beta<\gamma<\lambda$. Clearly, $H$ is linear and the partition 
relation $\lambda\lra (3)^2_\omega$ entails $\chi(H)\ge \aleph_1$.

A finite $3$-uniform hypergraph $F$ is called {\it obligatory} if it is contained
in every $3$-uniform hypergraph whose chromatic number is uncountable. 
Define for every $n\ge 2$ the cycle $C_n^{(3)}$ to be the hypergraph with
$2n$ vertices $x_i$, $y_i$ and $n$ edges $x_ix_{i+1}y_i$, where $i\in \ZZ/n\ZZ$
(see Figure~\ref{fig48}). 
We have just seen that the cycle $C_2^{(3)}$ is not obligatory and by a result of
Erd\H{o}s, Galvin, and Hajnal~\cite{EGH}*{Theorem~11.6} neither is $C_3^{(3)}$.

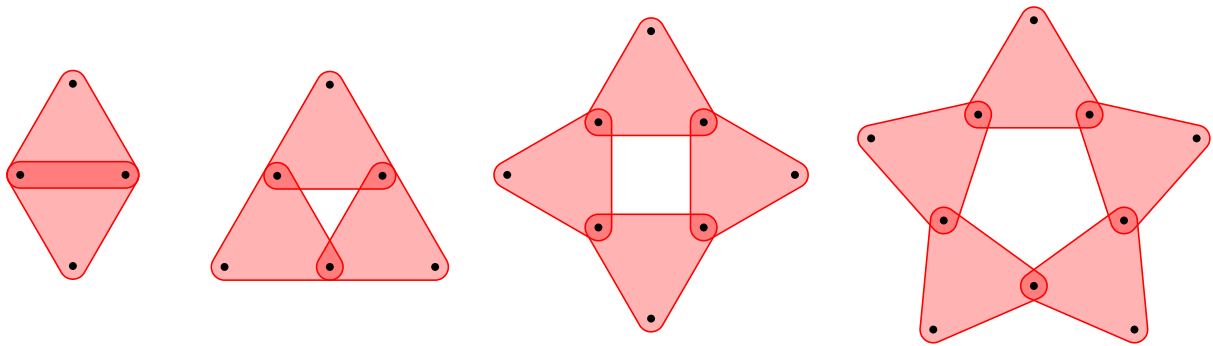
\begin{figure}[h!]
			\def\b{.7}
			\def\h{1.732*\b}
	\begin{subfigure}[c]{.16\textwidth}
		\centering
	\begin{tikzpicture}

		\coordinate (a12) at (0,0);
		\coordinate (a21) at (-\b, 1.732*\b);
		\coordinate (a22) at (\b, 1.732*\b);
		\coordinate (a31) at (0, 3.464*\b);
		
	\qedge{(a12)}{(a21)}{(a22)}{5pt}{.6pt}{red}{red, opacity=.3};
			\qedge{(a21)}{(a31)}{(a22)}{5pt}{.6pt}{red}{red, opacity=.3};
		
		\foreach \i in {12, 21, 22, 31} \fill (a\i) circle (1.5pt);
		
	\end{tikzpicture}
	\end{subfigure}
\hfill
	\begin{subfigure}[c]{.2\textwidth}
	\centering
	\begin{tikzpicture}
		\coordinate (a12) at (0,0);
		\coordinate (a11) at (-2*\b,0);
		\coordinate (a13) at (2*\b,0);
		\coordinate (a21) at (-\b, 1.732*\b);
		\coordinate (a22) at (\b, 1.732*\b);
		\coordinate (a31) at (0, 3.464*\b);
		
		\qedge{(a11)}{(a21)}{(a12)}{5pt}{.6pt}{red}{red, opacity=.3};
		\qedge{(a12)}{(a22)}{(a13)}{5pt}{.6pt}{red}{red, opacity=.3};
		\qedge{(a21)}{(a31)}{(a22)}{5pt}{.6pt}{red}{red, opacity=.3};
		
		\foreach \i in {11, 12, 13, 21, 22, 31} \fill (a\i) circle (1.5pt);
		
	\end{tikzpicture}
\end{subfigure}
\hfill
	\begin{subfigure}[c]{.26\textwidth}
	\centering
	\begin{tikzpicture}
		\coordinate (a11) at (0,-\b-\h);
		\coordinate (a21) at (-\b,-\b);
		\coordinate (a22) at (\b,-\b);
		\coordinate (a31) at (-\b-\h, 0);
		\coordinate (a32) at (\b+\h, 0);
		\coordinate (a41) at (-\b, \b);
			\coordinate (a42) at (\b, \b);
			\coordinate (a51) at (0, \b+\h);
		
		\qedge{(a11)}{(a21)}{(a22)}{5pt}{.6pt}{red}{red, opacity=.3};
		\qedge{(a31)}{(a41)}{(a21)}{5pt}{.6pt}{red}{red, opacity=.3};
		\qedge{(a42)}{(a32)}{(a22)}{5pt}{.6pt}{red}{red, opacity=.3};
		\qedge{(a41)}{(a51)}{(a42)}{5pt}{.6pt}{red}{red, opacity=.3};
		
		\foreach \i in {11, 21, 22, 31, 32, 41,42,51} \fill (a\i) circle (1.5pt);
		
	\end{tikzpicture}
\end{subfigure}
\hfill
	\begin{subfigure}[c]{.3\textwidth}
	\centering
	\begin{tikzpicture}
		\foreach \i in {1,...,5}{
		\coordinate (a\i) at (-90 + \i*72:1.8*\b);
		\coordinate (b\i) at (90+\i*72:3.25*\b);
	}
		
		\qedge{(a1)}{(a2)}{(b4)}{5pt}{.6pt}{red}{red, opacity=.3};
		\qedge{(a2)}{(a3)}{(b5)}{5pt}{.6pt}{red}{red, opacity=.3};
		\qedge{(a3)}{(a4)}{(b1)}{5pt}{.6pt}{red}{red, opacity=.3};
		\qedge{(a4)}{(a5)}{(b2)}{5pt}{.6pt}{red}{red, opacity=.3};
		\qedge{(a5)}{(a1)}{(b3)}{5pt}{.6pt}{red}{red, opacity=.3};

		\foreach \i in {1,...,5} {
			\fill (a\i) circle (1.5pt);
			\fill (b\i) circle (1.5pt);}
		
	\end{tikzpicture}
\end{subfigure}
\caption{The cycles $C_2^{(3)}$, $C_3^{(3)}$, $C_4^{(3)}$, and $C_5^{(3)}$.}
\label{fig48}
\end{figure}

Komj\'ath~\cite{kom01} proved that every obligatory hypergraph is $3$-partite.
Moreover the class of obligatory hypergraphs is closed under taking disjoint 
unions and one-point amalgamations; consequently, all forests are obligatory. 
Until very recently no further examples of obligatory hypergraphs were known 
and it was open whether, consistently or even provably, a hypergraph is 
obligatory if and only if it is a forest. 

This possibility was recently ruled out in~\cite{OH}, where the following examples 
are proposed. For every positive integer $n$
let $H^{(3)}_n$ be the hypergraph with $n^2+2n$ vertices $x_i$, $y_i$, $z_{ij}$
and~$n^2$ edges $x_iy_jz_{ij}$ (where $i,j\in [n]$). Thus $H^{(3)}_n$ arises from
the bipartite graph $K_{n, n}$ by adding a new vertex to every edge (see Figure~\ref{fig49}).

\begin{thm}
	For every natural number $n$ the hypergraph $H^{(3)}_n$ is obligatory. \qed
\end{thm}   

In particular, for every even $n\ge 4$ the cycle $C^{(3)}_n$ is obligatory.

\begin{figure}[h!]
		\def\h{1.3} 	\begin{subfigure}[c]{.05\textwidth}
		\centering
	\begin{tikzpicture}
	\coordinate (a) at (0 ,0);
	\coordinate (b) at (0,\h);		
	\coordinate (c) at (0,2*\h);
	
	\qedge{(a)}{(b)}{(c)}{4pt}{.6pt}{red}{red, opacity=.3};
	
	\foreach \i in {a,b,c} \fill (\i) circle (1.5pt);
		
	\end{tikzpicture}
	\end{subfigure}
\hfill
\begin{subfigure}[c]{.28\textwidth}
	\centering
\begin{tikzpicture}
		\def\w{.8}
		\def\x{1.9}  			
	\coordinate (a22) at (-\w, \h);
	\coordinate (a23) at (\w, \h);
	\coordinate (a32) at (-\w, 2*\h);
	\coordinate (a33) at (\w, 2*\h);

	\coordinate (a11) at  ($(a33)!\x!(a22)$);
	\coordinate (a14) at  ($(a32)!\x!(a23)$);
		\coordinate (a12) at  ($(a32)!\x!(a22)$);
	\coordinate (a13) at  ($(a33)!\x!(a23)$);
	
	\qedge{(a12)}{(a22)}{(a32)}{4pt}{.6pt}{red}{red, opacity=.3};
	\qedge{(a13)}{(a23)}{(a33)}{4pt}{.6pt}{red}{red, opacity=.3};
	\qedge{(a11)}{(a22)}{(a33)}{4pt}{.6pt}{red}{red, opacity=.3};
	\qedge{(a14)}{(a23)}{(a32)}{4pt}{.6pt}{red}{red, opacity=.3};

	\foreach \i in {12,13,22,23,32,33,11, 14} \fill (a\i) circle (1.5pt);				
\end{tikzpicture}

\end{subfigure}
\hfill
\begin{subfigure}[c]{.6\textwidth}
	\centering
	\begin{tikzpicture}
	
		\def\w{1.3}
		\def\x{1.9} 
\coordinate (a31) at (-.8*\w, 2*\h);
\coordinate (a32) at (0, 2*\h);
\coordinate (a33) at (.8*\w, 2*\h);
\coordinate (a21) at (-\w, \h);
\coordinate (a22) at (0, \h);
\coordinate (a23) at (\w, \h);

			\coordinate (a11) at ($(a33)!\x!(a21)$);
			\coordinate (a12) at ($(a32)!\x!(a21)$);
			\coordinate (a13) at ($(a31)!\x!(a21)$);
			\coordinate (a14) at ($(a33)!\x!(a22)$);
			\coordinate (a15) at ($(a32)!\x!(a22)$);
			\coordinate (a16) at ($(a31)!\x!(a22)$);
			\coordinate (a17) at ($(a33)!\x!(a23)$);
			\coordinate (a18) at ($(a32)!\x!(a23)$);
			\coordinate (a19) at ($(a31)!\x!(a23)$);

			\qedge{(a33)}{(a21)}{(a11)}{4pt}{.6pt}{red}{red, opacity=.3};
			\qedge{(a32)}{(a21)}{(a12)}{4pt}{.6pt}{red}{red, opacity=.3};
			\qedge{(a31)}{(a21)}{(a13)}{4pt}{.6pt}{red}{red, opacity=.3};
			\qedge{(a33)}{(a22)}{(a14)}{4pt}{.6pt}{red}{red, opacity=.3};
			\qedge{(a32)}{(a22)}{(a15)}{4pt}{.6pt}{red}{red, opacity=.3};
			\qedge{(a31)}{(a22)}{(a16)}{4pt}{.6pt}{red}{red, opacity=.3};
			\qedge{(a33)}{(a23)}{(a17)}{4pt}{.6pt}{red}{red, opacity=.3};
			\qedge{(a32)}{(a23)}{(a18)}{4pt}{.6pt}{red}{red, opacity=.3};
			\qedge{(a31)}{(a23)}{(a19)}{4pt}{.6pt}{red}{red, opacity=.3};

		\foreach \i in {11,12,13,14,15,16,17,18,19,21,22,23,31,32,33} \fill (a\i) circle (1.5pt);				
	\end{tikzpicture}
\end{subfigure}
\caption{The hypergraphs $H^{(3)}_1$, $H^{(3)}_2$, and $H_3^{(3)}$.}
\label{fig49}
\end{figure}
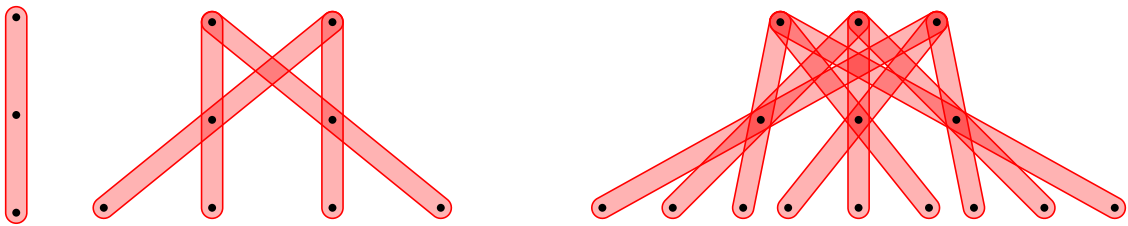

Let us finally introduce a related concept, which seems equally interesting. 
We call a {$3$-uniform} hypergraph $F$ {\it linearly obligatory} if every 
linear $3$-uniform hypergraph of uncountable chromatic number has a subhypergraph 
isomorphic to $F$. It has been shown by Hajnal and Komj\'ath~\cite{HK08}
that for $n\ne 2, 3, 5$ the cycle $C_n^{(3)}$ is linearly obligatory. 
This is complemented by a result of Komj\'ath~\cite{kom08}, which asserts 
that consistently there exists a linear hypergraph of uncountable chromatic 
number containing neither $C_3^{(3)}$ nor $C_5^{(3)}$.
Of course, every obligatory hypergraph is linearly obligatory as well, 
but the reverse implication is consistently false. This follows from results 
of Hajnal and Komj\'ath in~\cite{HK08}. 

\section{The girth Ramsey Theorem}\label{sec:grt}

\subsection{The induced Ramsey theorem}\label{subsec:61}
The question how the results in~\S\ref{subsec:35} generalise from vertex colourings 
to edge colourings motivated a lot of research in structural Ramsey theory during the 
last five decades. For graphs (and linear hypergraphs) a satisfactory understanding
has been reached only very recently~\cite{girth}, but for general hypergraphs there 
is still room for further investigations. In the remaining pages of this survey we 
can hardly do more than to scratch the surface of this fascinating area. 

We commence with the simplest existence question: given a graph $F$ and a number of 
colours $r$, does there exist a graph $H$ such that 
\[
	H\lra (F)^e_r\,?
\] 
This would mean that for every $r$-colouring of $H$ there is a monochromatic induced 
copy of~$F$ in~$H$. An affirmative answer has been obtained independently at about the 
same time by Deuber~\cite{Deuber75}, by Erd\H{o}s, Hajnal, and P\'osa~\cite{EHP75},
and by R\"{o}dl in his master thesis~\cites{Rodl73, Rodl76}. 

\begin{thm}[Induced Ramsey theorem for graphs]\label{thm:61}
	Given a graph $F$ and a number of colours $r$ there exists a graph $H$ such that 
	no matter how the edges of $H$ get coloured with $r$ colours, there is always a 
	monochromatic induced copy of $F$ in $H$. 
\end{thm}

Today several further proofs of this result are known, the most transparent 
of which are based on the partite construction method~\cite{NR81}, which we have 
already encountered in~\S\ref{subsec:33}. Here one starts with the observation that 
without the requirement that the monochromatic copy of $F$ needs to be induced one 
could simply take a sufficiently large clique. Indeed, the theorem of 
Ramsey~\cite{Ramsey30} allows us to fix an integer $n$ which is so large that for 
every $r$-colouring of $E(K_n)$ there is a monochromatic copy of $K_{|V(F)|}$ and,
a fortiori, a monochromatic (usually non-induced) copy of $F$. We shall now run 
a partite construction over $G=K_n$. Its pictures are $n$-partite graphs $\Pi$ 
accompanied by graph homomorphisms $\psi\colon \Pi\lra G$. 
Picture zero, denoted again by $\Pi_0$, consists of lots of vertex-disjoint copies 
of $F$, one for every copy of $F$ in~$G$. Thus it looks somewhat like 
Figure~\ref{fig45}, but with copies of $F$ instead of edges. 

Let us recall that in~\S\ref{subsec:33} we were colouring vertices, and in each of 
the pictures constructed after $\Pi_0$ one music line was processed. In some sense 
the entire construction reflected the fact that the vertex set of a picture is 
the disjoint union of its music lines. Now we are colouring edges, and the entire 
edge set of a picture can be expressed as a disjoint union of certain bipartite 
graphs, namely the preimages of the edges of $G$ with respect to the 
projection $\psi$. These bipartite graphs are called the {\it constituents} of the 
picture. For every picture $\Pi$ and every edge $e\in E(G)$ the 
constituent $\psi^{-1}(e)$ is denoted by $\Pi^e$.

Preparing the partite construction we fix an enumeration $E(G)=\{e(1), \dots, e(N)\}$,
where, in the present case, $N=\binom n2$. Starting with picture zero we intend to
define recursively a sequence of pictures $(\Pi_i)_{0\le i\le N}$, where in the 
formation of $\Pi_i$ we want to `process' the~$i^{\mathrm{th}}$ constituent of the 
previous picture. In~\S\ref{subsec:33} this `processing' involved an appeal to an 
induction hypothesis (or to the fact that hypergraph cliques have arbitrarily large 
chromatic number). In general, the r\^{o}le of such statements is played by so-called 
{\it partite lemmata}. For edge-colourings of graphs the simplest partite lemma 
imaginable reads as follows. 

\begin{lemma}\label{lem:62}
	For every bipartite graph $B=(X_B, Y_B, E_B)$ and every number of colours $r$
	there exists a bipartite graph $H=(X_H, Y_H, E_H)$ with the following property:
	no matter how~$E_H$ gets $r$-coloured, there exist sets $X'_B\subseteq X_H$
	and $Y'_B\subseteq Y_H$ such that the induced subgraph $H[X'_B, Y'_B]$ 
	is monochromatic and isomorphic to $B$.   
\end{lemma}
  
In practice one usually abbreviates the conclusion of this lemma to `there is 
a monochromatic {\it partite} copy of $B$'.\footnote[1]{In principle, there can 
also be monochromatic copies of $B$ not respecting the bipartite structure; 
but they are useless for the partite construction.} Postponing the proof of 
Lemma~\ref{lem:62} to a later moment, we proceed with our explanation 
how one proves Theorem~\ref{thm:61} by means of the partite construction method. 

Recall that we already have chosen picture zero and now we want to define a 
sequence of further pictures $\Pi_1, \dots, \Pi_N$. When for some $i\in [N]$ the 
picture~$\Pi_{i-1}$ has just been constructed, we apply Lemma~\ref{lem:62} to its 
constituent $B_i=\Pi_{i-1}^{e(i)}$, thus obtaining some bipartite graph $H_i$.
Now we extend all partite copies of $B_i$ in $H_i$ to its own copy of $\Pi_{i-1}$
and, as usual, while doing so we ensure that distinct standard copies of $\Pi_{i-1}$
generated in this manner are as disjoint as possible (see Figure~\ref{fig61}). 
In other words they are only allowed to intersect in the constituent~$\Pi_i^{e(i)}$ 
of the resulting picture $\Pi_i$. This completes our description 
of $\Pi_1, \dots, \Pi_N$.

\usetikzlibrary {arrows.meta}
\usetikzlibrary{decorations.pathreplacing}

\begin{figure}[ht]
	\centering	
	\def\h{2}
	\def\w{1}
	
			\begin{tikzpicture}[scale=1]
				\coordinate  (a3) at (-\w,\h);
				\coordinate (a4) at (\w,\h);
				\coordinate (b3) at (-\w,-\h);
				\coordinate (b4) at (\w,-\h);
					\coordinate  (a1) at (-3.9*\w,\h);
				\coordinate (a2) at (-1.9*\w,\h);
				\coordinate (b1) at (-3.9*\w,-\h);
				\coordinate (b2) at (-1.9*\w,-\h);
				\coordinate  (a5) at (3.9*\w,\h);
				\coordinate (a6) at (1.9*\w,\h);
				\coordinate (b5) at (3.9*\w,-\h);
				\coordinate (b6) at (1.9*\w,-\h);
				\coordinate (x3) at (-1.45*\w, .29*\h);
				\coordinate (x4) at (1.45*\w, .29*\h);
				\coordinate (y3) at (-1.45*\w,- .29*\h);
				\coordinate (y4) at (1.45*\w,- .29*\h);
				\coordinate (x1) at (-4.35*\w, .29*\h);
				\coordinate (x2) at (-1.45*\w, .29*\h);
				\coordinate (y1) at (-4.35*\w,- .29*\h);
				\coordinate (y2) at (-1.45*\w,- .29*\h);
					\coordinate (x6) at (4.35*\w, .29*\h);
				\coordinate (x5) at (1.45*\w, .29*\h);
				\coordinate (y6) at (4.35*\w,- .29*\h);
				\coordinate (y5) at (1.45*\w,- .29*\h);
				
		\fill [green!70!black] (-6*\w,0) ellipse (2pt and 17pt);
			\draw (-6*\w,\h)--(-6*\w,-\h);
			
			\draw [thick](-4.7*\w, .295*\h)--(4.7*\w, .295*\h);
			\draw [thick](-4.7*\w, -.295*\h)--(4.7*\w, -.295*\h);
			
			\draw [thick, red!80!black, decorate,
			decoration = {brace}] (4.9*\w, .295*\h) --  (4.9*\w, -.295*\h);
			
		\draw (a3)--(a4) [out=-65, in=65] to (b4)--(b3)[out=115, in=-115] to
		(a3);
			 \draw (a1)--(a2) [out=-65, in=65] to (b2)--(b1)[out=115, in=-115] to	(a1);
			  \draw (a6)--(a5) [out=-65, in=65] to (b5)--(b6)[out=115, in=-115] to	(a6);
			  
	\fill [red, opacity=.3](x3)--(x4) [out=-82, in=82] to (y4)--(y3)[out=100, in=-100] to (x3);
	\draw [thick,red!80!black](x3)--(x4) [out=-82, in=82] to (y4)--(y3)[out=98, in=-98] to (x3);
		\fill [red, opacity=.3](x1)--(x2) [out=-82, in=82] to (y2)--(y1)[out=100, in=-100] to (x1);
	\draw [thick,red!80!black](x1)--(x2) [out=-82, in=82] to (y2)--(y1)[out=98, in=-98] to (x1);
		\fill [red, opacity=.3](x5)--(x6) [out=-82, in=82] to (y6)--(y5)[out=100, in=-100] to (x5);
	\draw [thick,red!80!black](x5)--(x6) [out=-82, in=82] to (y6)--(y5)[out=100, in=-100] to (x5);

	\draw [thick, red!80!black, -Stealth] (-2.6*\w, -1.4*\h) -- (-3*\w, 0); 
	\draw [thick, red!80!black, -Stealth] (0, -1.4*\h) -- (0, 0); 
	\draw [thick, red!80!black, -Stealth] (2.6*\w, -1.4*\h) -- (3*\w, 0); 
		\draw [thick,  -Stealth] (-1.6*\w, 1.25*\h) -- (-2.8*\w, .7*\h); 
	\draw [thick,  -Stealth] (0, 1.25*\h) -- (0, .7*\h); 
	\draw [thick,  -Stealth] (1.6*\w, 1.25*\h) -- (2.8*\w, .7*\h);

\node [green!70!black] at (-6.5*\w,0) {$e(i)$};
\node at (-6*\w,-1.2*\h) {$G=K_n$};
			
\node [red!80!black] at (0,-1.6*\h) {partite copies of $B_i=\Pi_{i-1}^{e(i)}$ in $H_i$};	
\node at (0,1.4*\h) {standard copies of $\Pi_{i-1}$};			

\node [red!80!black] at (5.4*\w,0) {$H_i$};

		\end{tikzpicture}

				\caption{The construction of $\Pi_i$.}
				\label{fig61}

	\end{figure}
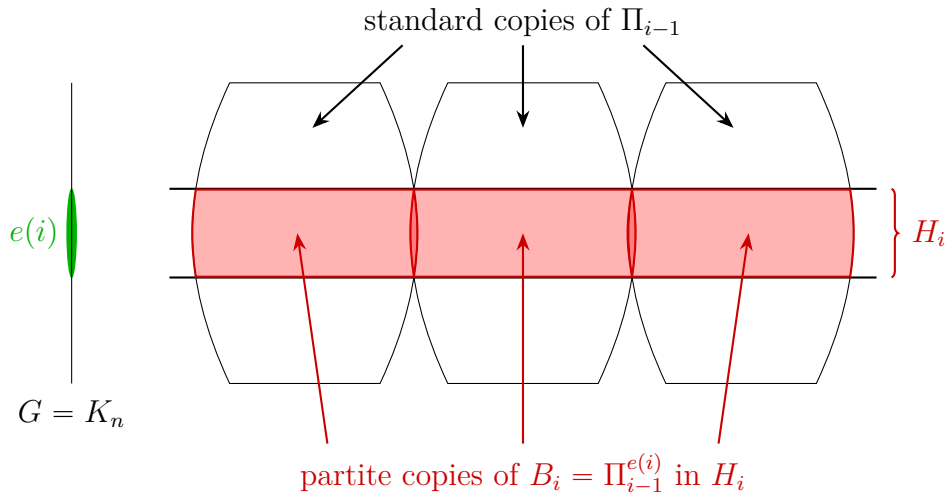 
Here is their most important property: Whenever $i\in [N]$ and 
${f\colon E(\Pi_i)\lra [r]}$ is a colouring, there is an induced 
copy~$\widetilde{\Pi}_0$ of picture zero such that the 
constituents $\widetilde{\Pi}_0^{e(1)}, \dots, \widetilde{\Pi}_0^{e(i)}$ are 
monochromatic. As usual, this can be shown by a straightforward induction on~$i$. 

Let us now check that the final picture $H=\Pi_N$ is as required 
by Theorem~\ref{thm:61}. Given any colouring $f\colon E(H)\lra [r]$ the result 
of the previous paragraph yields an induced copy $\widetilde{\Pi}_0$ of picture zero 
all of whose constituents are monochromatic. The colour pattern we see on these 
constituents projects to an auxiliary colouring $f_\star\colon E(G)\lra [r]$. 
By our sufficiently large choice of $G$ there is a (presumably non-induced) 
copy of $F$ in $G$ which is monochromatic with respect to $f_\star$. 
Now the corresponding copy of $F$ in $\widetilde{\Pi}_0$ is induced in $H$ and 
monochromatic with respect to $f$. The only step in the proof of Theorem~\ref{thm:61}
still missing is that we need to address the partite lemma. 

\begin{proof}[Proof of Lemma~\ref{lem:62}]
	For all integers $m\ge t\ge 1$ let $B(m, t)$ be the bipartite graph with 
	vertex classes $[m]^{(t)}$ and $[m]$ whose edges are all pairs $Aa$ with 
	$a\in A$, where $A\in [m]^{(t)}$ and $a\in [m]$. For every bipartite graph $B$
	there exist integers $m\ge t\ge 1$ such that~$B(m, t)$ contains a partite 
	copy of~$B$. Thus it suffices to prove the partite lemma for $B=B(m, t)$. 
	
	Given $m$, $r$, and $t$ one can show that for every 
	sufficiently large integer $m_\star$ and $t_\star=r(t-1)+1$ the bipartite 
	graph $H=B(m_\star, t_\star)$ is as required for $B=B(m, t)$ and $r$ colours. 
	The main idea here is that every 
	$r$-colouring of $E(H)$ induces an auxiliary $(r^{t_\star})$-colouring 
	of $[m_\star]^{(t_\star)}$ recording for every vertex $A\in [m_\star]^{(t_\star)}$
	the colour pattern we see on its neighbourhood. Ramsey's theorem yields arbitrarily 
	large subsets $Z\subseteq [m_\star]$ such that $Z^{(t_\star)}$ is monochromatic 
	with respect to this auxiliary colouring. The common auxiliary colour of the 
	vertices in $Z^{(t_\star)}$ can be viewed as a colouring $[t_\star]\lra [r]$, 
	which has a monochromatic $t$-subset $T\subseteq [t_\star]$ owing to the box 
	principle. Using $Z$ and $T$ one can now build the desired monochromatic 
	partite copy of $B$ in $H$.   	
\end{proof}

Full details on the material presented so far can be found in~\cite{NR81}.
After the first proofs of Theorem~\ref{thm:61} had been discovered, it 
was an open problem for a few years to extend the result to hypergraphs.
Eventually the following statement has been proved independently
by Abramson and Harrington~\cite{AH78}, and by Jarik and R\"{o}dl~\cite{NR77}.

\begin{thm}[Induced Ramsey theorem for hypergraphs]\label{thm:63}
	For every $k$-uniform hypergraph~$F$ and every number of colours $r$ there 
	exists a $k$-uniform hypergraph~$H$ such that $H\lra (F)^e_r$. Explicitly,
	this partition symbol means that for every $r$-colouring of $E(H)$ there 
	exists a monochromatic induced copy of $F$ in $H$.
\end{thm} 

When one tries to adapt the above proof by partite construction to the hypergraph
setting, the only step that is not immediately clear is how one establishes the 
natural generalisation of the partite lemma. In this statement we view every 
$k$-partite, $k$-uniform hypergraph $B$ as being equipped with a distinguished vertex 
partition $V(B)=V_1(B)\dcup \dots\dcup V_k(B)$ such that $|V_i(B)\cap e|=1$ holds 
for all $i\in [k]$ and $e\in E(B)$. As in the case of bipartite graphs, {\it partite
copies} are required to respect this partite structure. Here is the partite lemma 
required for the proof of Theorem~\ref{thm:63}.

\begin{lemma}
	Given a $k$-partite, $k$-uniform hypergraph $B$ and a number of colours $r$
	there exists a $k$-partite, $k$-uniform hypergraph $H$ such that for every 
	$r$-colouring of $E(H)$ there exists a monochromatic, induced, partite copy 
	of $B$. 
\end{lemma}    
 
Jarik and R\"odl~\cite{NR82} found an extremely elegant proof of this lemma 
based on the Hales-Jewett theorem~\cite{HJ63} (see also Shelah~\cite{Sh329}). 
The idea is that we want to take $H=B^n$ for a certain Hales-Jewett number~$n$. 
More precisely, we first fix an integer~$n$ which is so large that for every 
$r$-colouring of the Hales-Jewett cube $E(B)^n$ there is a monochromatic 
combinatorial line. Now we set $V_i(H)=V_i(B)^n$ for every $i\in [k]$, and for 
every $n$-tuple in~$E(B)^n$ we put the expected edge into $H$. It can then be 
confirmed straightforwardly that the combinatorial lines in $E(B)^n$ yield 
induced partite copies of $B$ in $H$. This is explained, for instance, in each 
of the references~\cites{BNRR, NR82, NR87, girth}.
         
\subsection{Three theorems}\label{subsec:62}
The construction by means of which we proved Theorem~\ref{thm:61} has several 
desirable properties going beyond $H\lra (F)^e_r$, two of which we would like 
to point out. First, the graphs $F$ and $H$ have the same clique number. This 
can easily be seen by an argument called ``induction along the partite construction''.
Since picture zero is just a disjoint union of copies of $F$, we 
have $\omega(\Pi_0)=\omega(F)$. Moreover, the disjointness requirement in the 
formation of each new picture $\Pi_i$ yields $\omega(\Pi_i)=\omega(\Pi_{i-1})$ for 
every $i\in [N]$, so that altogether we have indeed 
$\omega(H)=\omega(\Pi_N)=\omega(\Pi_0)=\omega(F)$.   

The second property deals with the system of copies of $F$ constructed along the 
way. Given two graphs (or $k$-uniform hypergraphs) $F$ and $H$ we 
write $\binom HF$ for the set of all induced copies of $F$ in $H$. 
With every picture $\Pi_i$ encountered in the partite construction 
we want to associate a system of copies $\ccP_i\subseteq \binom{\Pi_i}F$. 
The system $\ccP_0$ is defined in such a way that $\Pi_0$ is its disjoint union, 
and for every (not necessarily induced) copy of $F$ in $G$ there is a copy in~$\ccP_0$
projecting to it. When for some $i\in [N]$ the system $\ccP_{i-1}$ has just 
been determined, we let $\ccP_i$ be the union of all copies of $\ccP_{i-1}$ 
corresponding to the standard copies of $\Pi_{i-1}$ in $\Pi_i$. 
Roughly speaking, the final system $\ccP_N\subseteq \binom{\Pi_N}F$ consists 
of all copies of $F$ which are `relevant' for the verification of $\Pi_N\lra (F)^e_r$, 
so that in an obvious sense we have $\ccP_N\lra (F)^e_r$. 
An easy induction along the partite construction reveals that any two distinct copies 
in~$\ccP_N$ are either disjoint, or they intersect in a single vertex, or they intersect in two vertices joined by an edge. Summarising the discussion so far, 
we have shown the following.

\begin{prop}\label{prop:65}
	For every graph $F$ and every number of colours $r$, there exists a graph $H$
	together with a system $\ccH\subseteq \binom HF$ such that 
		\begin{enumerate}[label=\rmlabel]
		\item\label{it:65i} $\ccH\lra (F)^e_r$;
		\item\label{it:65ii} $\omega(H)=\omega(F)$;
		\item\label{it:65iii} and any two distinct copies in $\ccH$ are either disjoint, 
			or they intersect in a vertex, or they intersect in an edge. \qed
	\end{enumerate}
\end{prop}  

This raises several questions. In view of the topic of this survey, the perhaps 
most immediate one is whether in~\ref{it:65ii} the clique number can be replaced 
by girth (provided that $F$ is not a forest). Such an assertion would certainly 
require a different construction, because even for $F=C_5$ the graph $H$ we 
produced contains lots of four-cycles. For more than a decade, this was a common 
problem of all known proofs of the induced Ramsey theorem. Erd\H{o}s~\cite{Erd75} 
asked whether a graph $H$ with $H\lra (C_5)^e_2$ and $\gth(H)=5$ exists, and expected 
a negative answer. This was due to the fact that, at that time, he believed in some 
kind of meta-conjecture that edge-colourings of finite graphs display phenomena 
similar to vertex-colourings of uncountable graphs. Thus he took the fact that 
no $C_4$-free graphs of uncountable chromatic number exist as an indication that at 
least some graph of girth $5$ should have no Ramsey graph of girth $5$. 
Jarik and R\"odl~\cite{NR87} refuted this suspicion. Their argument is capable 
of controlling cycles of lengths $5$, $6$, and $7$ as well. However, it was 
always clear that excluding $8$-cycles is horrendously difficult. The problem 
remained a central goal of R\"{o}dl's research programme for almost forty years, 
until it was recently solved in~\cite{girth}. 

\begin{thm}[Girth Ramsey theorem, first version]\label{thm:66}
	For every graph $F$ which is not a forest and every number of colours $r$ 
	there exists a graph $H$ such that $H\lra (F)^e_r$ and $\gth(H)=\gth(F)$. \qed
\end{thm} 

Let us next point to another question suggested by Proposition~\ref{prop:65}.
Its clause~\ref{it:65iii} gives complete control over the possible intersections 
of two copies in~$\ccH$. In the nontrivial case $e(F), r\ge 2$ it 
certainly needs to happen from time to time that two copies in~$\ccH$ share an 
edge---otherwise we could colour the copies in~$\ccH$ one by one without making 
any of them monochromatic. In the spirit of Theorem~\ref{thm:1535} it would be even 
more satisfactory to control the possible intersection patterns of more than 
two copies. The best one could hope for is that locally the Ramsey system of 
copies has a forest-like structure in the following sense. 

\begin{dfn}\label{dfn:67}
	Given a graph $F$ we call a set $\ccN$ of graphs isomorphic 
	to $F$ a {\it forest of copies of $F$} if there exists 
	an enumeration 
	$\ccN=\bigl\{F_1, \ldots, F_{|\ccN|}\bigr\}$
	such that for every $j\in [2, |\ccN|]$ the 
	set $z_j=\bigl(\bigcup_{i<j}V(F_i)\bigr)\cap V(F_j)$ satisfies  
		\begin{enumerate}[label=\rmlabel]
		\item\label{it:67i} either $|z_j|\le 1$
		\item\label{it:67ii} or $z_j\in \bigl(\bigcup_{i<j}E(F_i)\bigr)\cap E(F_j)$.
	\end{enumerate}
	We denote the union of a forest of copies $\ccN$ 
	by $\bigcup\ccN$; explicitly, this is the graph
	with vertex set $\bigcup_{F_\star\in \ccN}V(F_\star)$ and edge 
	set $\bigcup_{F_\star\in \ccN}E(F_\star)$. 
	A graph $G$ is said to be a {\it partial $F$-forest} if it is an induced 
	subgraph of $\bigcup\ccN$ for some forest $\ccN$ of copies of~$F$.
\end{dfn}

The following result from~\cite{girth} analyses the local structure 
of Ramsey graphs completely. 

\begin{thm}[Girth Ramsey theorem, second version]\label{thm:19}
	For every graph $F$ and all $r, n\in \NN$ there exists a graph~$H$ 
	with $H\lra (F)^e_r$ such that every set $X\subseteq V(H)$
	whose size it at most~$n$ induces a partial $F$-forest in $H$. \qed
\end{thm}

The proof of Theorem~\ref{thm:19} constructs $H$ together with a distinguished 
system of copies $\ccH\subseteq\binom HF$, which satisfies, in particular, 
the partition relation $\ccH\lra (F)^e_r$. A further interesting claim can be made 
about this system $\ccH$, which seems to be stronger than the conclusion 
that $H$ is locally a partial $F$-forest. Namely, $\ccH$ itself has a comparable 
property. But before making this precise we should emphasise a bizarre difference 
between ordinary forests and $F$-forests. 
Everybody knows that the former are closed under taking subgraphs. 
Subsets of $F$-forest, on the other hand, can fail to be $F$-forests themselves (see 
Figure~\ref{fig62}).

\begin{figure}[ht]
	\centering

	\begin{tikzpicture}[scale=.5]
	
	\coordinate (x1) at (-3, 0);
	\coordinate (x0) at (3,0);
	\coordinate (x2) at (0,5);

\fill [blue!20!white, rounded corners = 8pt, opacity=.5] (.4,5.1) to[out=140,in=60] (-4,4) to[out=-110, in=150] (-2.9,-.4) --cycle;

\draw [thick, rounded corners = 8pt] (.4,5.1) to[out=140,in=60] (-4,4) to[out=-110, in=150] (-2.9,-.4) --cycle;

\fill [blue!20!white, rounded corners = 8pt, opacity=.5] (0,5.4) -- (-3.36,-.2) -- (3.36,-.2)--cycle;

\draw [thick, rounded corners = 8pt] (0,5.4) -- (-3.36,-.2) -- (3.36,-.2)--cycle;

\fill [blue!20!white, rounded corners = 8pt, opacity=.5] (-.4,5.1) to[out=40,in=120] (4,4) to[out=-70, in=30] (2.9,-.4) --cycle;

\draw [thick, rounded corners = 8pt] (-.4,5.1) to[out=40,in=120] (4,4) to[out=-70, in=30] (2.9,-.4) --cycle;

\fill [blue!20!white, rounded corners = 8pt, opacity=.5] (-3.3,.3) to[out=-90,in=175] (0,-3) to[out=5, in=-90] (3.3,.3) --cycle;

\draw [thick, rounded corners = 8pt] (-3.3,.3) to[out=-90,in=175] (0,-3) to[out=5, in=-90] (3.3,.3) --cycle;

		\draw [green!70!black, line width=.5mm] (x0) -- (x1) -- (x2) -- (x0);
	
	\foreach \i in {0,1,2}{
		\fill (x\i) circle (3pt);}
	
	\node at (0,2) {$F$};
	\node at (-2.8,3.3) {$F_0$};
	\node at (2.8,3.3) {$F_1$};
	\node at (0,-1.5) {$F_2$};
	
	\node at (-.05,5.95) {$x_2$};
	\node at (-3.7,-.5) {$x_1$};
	\node at (3.8,-.5) {$x_0$};
	
	\end{tikzpicture}

	\caption{The subforest $\{F_0, F_1, F_2\}$ fails to be a forest.}
	\label{fig62} 
\end{figure}
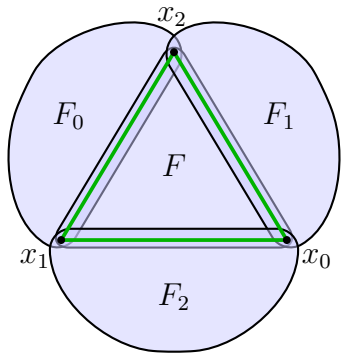 
 
In the example we have chosen $F$ is any graph containing a triangle~$x_0x_1x_2$. 
For every index $i\in\ZZ/3\ZZ$ the copy $F_i$ of $F$ has the edge~$x_{i+1}x_{i+2}$ 
but nothing else in common with $F$. Except for these intersections the copies 
in $\ccN=\{F, F_0, F_1, F_2\}$ are mutually disjoint. This enumeration exemplifies 
that~$\ccN$ is a forest of copies. However its subset $\ccN^-=\ccN\sm\{F\}$ fails 
to be such a forest. For instance, for the enumeration $\ccN^-=\{F_0, F_1, F_2\}$ the 
set $z_2=\bigl(V(F_0)\cup V(F_1)\bigr)\cap V(F_2)=\{x_0, x_1\}$ is certainly 
not in case~\ref{it:67i} of Definition~\ref{dfn:67} and, as it 
fails to be an edge of $F_0$ or $F_1$, it does not satisfy~\ref{it:67ii} either.
By symmetry a similar problem arises when one enumerates~$\ccN^-$ in any other 
way. 

What this shows is that given a graph $F$ we cannot ask for Ramsey systems $\ccH$
such that all `small' subsets of $\ccH$ are forests 
of copies. We can still demand, however, that every `small' 
subset $\ccN\subseteq \ccH$ is contained in a forest of copies that is not much 
larger than $\ccN$. The following result from~\cite{girth} makes this precise.   

\begin{thm}[Girth Ramsey theorem, third version] \label{thm:1522}
	Given a graph $F$ and $r, n\in \NN$ there exists a graph~$H$
	together with a system of copies $\ccH\subseteq\binom{H}{F}$ 
	satisfying not only $\ccH\lra (F)^e_r$ but also the following statement:
	For every $\ccN\subseteq \ccH$ with $|\ccN|\in [2, n]$ there exists 
	a set $\ccX\subseteq \ccH$ such that $|\ccX|\le |\ccN|-2$ and $\ccN\cup\ccX$
	is a forest of copies. \qed
\end{thm}
     
It deserves to be pointed out that the upper bound $|\ccX|\le |\ccN|-2$ 
is best possible. Roughly this is because it requires $|\ccN|-2$ triangles
to triangulate an $|\ccN|$-gon. If $F=K_3$ and there is some $C_{|\ccN|}$ in $H$
such that every copy in $\ccN$ contains a unique edge of this cycle, then the 
copies in $\ccX$ need to triangulate the cycle (see Figure~\ref{fig63}). 

\begin{figure}[ht]
	\centering

		\begin{subfigure}[b]{0.3\textwidth}
			\centering
	\begin{tikzpicture}[scale=.6]
	
	\def\r{1.8cm};
	\coordinate (x0) at (142:\r);
	\coordinate (x1) at (214: \r);
	\coordinate (x2) at (286: \r);
	\coordinate (x3) at (358: \r);
	\coordinate (x4) at (70: \r);
	\coordinate (x5) at (106:3.5cm);
	\coordinate (x6) at (178:3.5cm);
	\coordinate (x7) at (250:3.5cm);
	\coordinate (x8) at (322:3.5cm);
	\coordinate (x9) at (34:3.5cm);

	\draw [thick] (x0)--(x1) --(x2) --(x3)--(x4) -- cycle;
	\draw [thick] (x4) -- (x5) -- (x0) -- (x6) -- (x1) -- (x7) -- (x2) -- (x8) -- (x3) -- (x9) -- cycle;

	\end{tikzpicture}

	\caption{A cycle of triangles}
	\label{fig63a} 

	\end{subfigure}
	\hfill    
	\begin{subfigure}[b]{0.5\textwidth}
		\centering
		
			\begin{tikzpicture}[scale=.6]
			
			\def\r{1.8cm};
			\coordinate (x0) at (142:\r);
			\coordinate (x1) at (214:\r);
			\coordinate (x2) at (286: \r);
			\coordinate (x3) at (358: \r);
			\coordinate (x4) at (70: \r);
			\coordinate (x5) at (106:3.5cm);
			\coordinate (x6) at (178:3.5cm);
			\coordinate (x7) at (250:3.5cm);
			\coordinate (x8) at (322:3.5cm);
			\coordinate (x9) at (34:3.5cm);

			\draw [thick] (x0)--(x1) --(x2) --(x3)--(x4) -- cycle;
			\draw [thick] (x4) -- (x5) -- (x0) -- (x6) -- (x1) -- (x7) -- (x2) -- (x8) -- (x3) -- (x9) -- cycle;
			\draw [thick] (x3) -- (x0) -- (x2);

			\end{tikzpicture}
				\caption{Adding further triangles creates a forest}
				\label{fig63b} 
				\end{subfigure}    
		\caption{The necessity of $\ccX$ in Theorem~\ref{thm:1522}}	\label{fig63}
		\vspace{-1em}
	\end{figure} 
 
\subsection{Ideas.} \label{subsec:63}
There is not much we can say about the proof of any version of the girth 
Ramsey theorem in a few pages. A general theme is that it is difficult to 
isolate special cases, which are simpler than the general result. It is rather the 
other way around: One has to develop several further concepts, such as trains, 
Roman $\Gth$, and German $\GTH$, which allow the formulation of even more general 
Ramsey theoretic statements (see, e.g.,~\cite{girth}*{\S4.4 and~\S10.3}), which 
can then be proved by an induction scheme resembling a transfinite induction 
up to $\omega^\omega$. 

As in~\S\ref{subsec:33} the proof cannot be understood if one just wants to focus 
on the graph case. In fact, each of our three versions of the girth Ramsey theorem
holds for linear hypergraphs instead of graphs as well, and the proof requires this 
level of generality for roughly the same reason we have already seen. 

Jarik and R\"odl~\cite{NR87} discovered that the proof of Theorem~\ref{thm:63}
we have outlined in~\S\ref{subsec:61} can be used for maintaining linearity. 

\begin{thm}\label{thm:631}
	Given a linear, $k$-uniform hypergraph $F$ and a number of colours $r$ 
	there exists a linear, $k$-uniform hypergraph $H$ such that $H\lra (F)^e_r$.
\end{thm}

There is no problem with the partite lemma, because for every linear, $k$-partite,
$k$-uniform hypergraph $B$ all Hales-Jewett powers $B^n$ are linear as well. 
What requires some thought when proving Theorem~\ref{thm:631} is that 
no $2$-cycles are introduced in the amalgamation steps (see Figure~\ref{fig61}).  
In~\cite{NR87} Jarik and R\"odl accomplish this by studying the possible intersection
patterns of partite copies of $B$ corresponding to combinatorial lines very carefully. 
More recently (see e.g.~\cites{BNRR, girth}) a different approach to this issue 
became popular. One first runs the partite construction under the additional 
assumption that $H$ is a $k$-partite, $k$-uniform hypergraph as well. This has the 
advantage that vertically we do not have to use Ramsey's theorem. Instead, it is 
preferable to use the Hales-Jewett partite lemma not only horizontally, but also 
vertically. 
Accordingly we end up getting a $k$-partite Ramsey hypergraph again, and in this 
case it is much easier to check that linearity is preserved. So the result is that 
we have a new partite lemma for linear, $k$-partite, $k$-uniform hypergraphs, called 
the {\it clean partite lemma}. In comparison to the Hales-Jewett partite lemma, its 
main advantage is that it generates systems of partite copies satisfying 
clause~\ref{it:65iii} of Proposition~\ref{prop:65}. This renders it rather obvious 
that linearity is preserved when we want to prove Theorem~\ref{thm:631} by a 
partite construction using Ramsey's theorem vertically and the clean partite lemma
horizontally. 

The reason why we have spent so much time on this somewhat subtle point in a 
proof variant of Theorem~\ref{thm:631} is that such usages of the partite 
construction method as a `cleaning device' occur all over the place in the 
proof of the girth Ramsey theorem. Whenever we obtain a Ramsey theoretic result 
with `complicated possible intersections' of copies, we try to clean it by running 
the partite construction once more. Of course this plan also imposes some restrictions 
on the proof strategy:
concepts we introduce and additional properties we acquire can be considered 
useful only when they are `indestructible by partite constructions'. 
For instance, the extension lemma (cf. \cite{girth}*{Lemma~9.1}) and the 
German $\GTH$ iterability lemma (cf. \cite{girth}*{Proposition~9.14}) implement 
this theme. 

Besides Ramsey's theorem, the Hales-Jewett partite lemma, and constructions 
derivable from them by means of the partite construction method, the proof of the 
girth Ramsey theorem also involves a different procedure for obtaining new 
constructions from known ones, called the extension process. The basic idea 
was again pioneered by Jarik and R\"{o}dl, who used it in their work on $C_4$-free
Ramsey graphs~\cite{NR87} mentioned in the previous subsection. A major step in 
their argument is the following $C_4$-free  partite lemma.

\begin{lemma}[Jarik and R\"{o}dl]\label{lem:632}
	For every $C_4$-free bipartite graph $B$ and every number colours~$r$ there
	exists a $C_4$-free bipartite graph $H$ such that for every $r$-colouring 
	of $E(H)$ there exists a monochromatic, induced, partite copy of $B$. 
\end{lemma}

The proof of this lemma has certain similarities with the proof of Lemma~\ref{lem:62}
we sketched in~\S\ref{subsec:61}. Attempting to emphasise the common features of 
both proofs we define for every bipartite graph $B=(X, Y, E)$ with the properties that 
\begin{enumerate}
	\item[$\bullet$] all vertices in $X$ have the same degree $t\ge 2$ 
	\item[$\bullet$] and no two vertices in $X$ have the same neighbourhood 
\end{enumerate}
the $t$-uniform {\it neighbourhood hypergraph} $F=\NH(B)$ by setting
\[
	V(F)=Y
	\quad \text{ and } \quad 
	E(F)=\{N(x)\colon x\in X\}\,.
\]
For instance, the neighbourhood hypergraph of the bipartite graph $B(m, t)$
defined in the proof of Lemma~\ref{lem:62} is the $t$-uniform clique $K_m^{(t)}$.
Roughly speaking the proof of Lemma~\ref{lem:62} consists of the four steps 
\[
	B(m, t)
	\overset{\NH}{\xrightarrow{\hspace*{1cm}}} 
	K_m^{(t)}
	\overset{\mathrm{extension}}{\xrightarrow{\hspace*{2cm}}} 
	K_{|Z|}^{(t_\star)}
	\overset{\mathrm{Ramsey}}{\xrightarrow{\hspace*{2cm}}} 
	K_{m_\star}^{(t_\star)} 
	\overset{\NH^{-1}}{\xrightarrow{\hspace*{1cm}}} 
	B(m_\star, t_\star)\,,
\]
where the first and last arrow indicate the formation of the neighbourhood 
hypergraph and its inverse operation, respectively; the second `extension' 
arrow yields a $t_\star$-uniform hypergraph, where our choice $t_\star=(t-1)r+1$
prepares an application of the box principle; finally, the third arrow indicates 
an application of Ramsey's theorem with $r^{t_\star}$ colours. 

When proving Lemma~\ref{lem:632} we start with some $C_4$-free bipartite graph 
$B=(X, Y, E)$ instead of $B(m, t)$. Without loss of generality we can assume 
that all vertices in $X$ have the same degree $t\ge 2$. Thus $B$ has a $t$-uniform 
neighbourhood hypergraph $F=\NH(B)$. The assumption $C_4\not\subseteq B$ implies 
that $F$ is linear. Without going into a lot of detail here, one then forms 
a linear, $t_\star$-uniform `extension' $M$ of $F$. Instead of Ramsey's theorem  
we employ Theorem~\ref{thm:631}, thus getting a linear, $t_\star$-uniform 
hypergraph $N$
such that $N\lra (M)^e_{r^{t_\star}}$. Finally, the linearity of $N$ implies that 
the bipartite graph $H=\NH^{-1}(N)$ is again $C_4$-free. 

Based on the plan
\[
	B
	\overset{\NH}{\xrightarrow{\hspace*{1cm}}} 
	F
	\overset{\mathrm{extension}}{\xrightarrow{\hspace*{2cm}}} 
	M
	\overset{\mathrm{Thm}~\ref{thm:631}}{\xrightarrow{\hspace*{2cm}}} 
	N 
	\overset{\NH^{-1}}{\xrightarrow{\hspace*{1cm}}} 
	H
\]
it is not too difficult to work out how one needs to define 
the `extension'~$M$ in such a way that $H$ will be as required by Lemma~\ref{lem:632}.
In any case, the curious reader can find full details in~\cite{NR87}.

The abstract version of the extension process defined and studied in the 
proof of the girth Ramsey theorem deals with structures called {\it pretrains}:
these are pairs $(H, \equiv)$ consisting of a hypergraph $H$ and an equivalence 
relation $\equiv$ on $E(H)$. The {\it wagons} of a pretrain $(H, \equiv)$ are the 
equivalence classes of $\equiv$. For instance, with every bipartite 
graph $B=(X, Y, E)$ we can associate a pretrain by declaring two edges to be 
equivalent if and only if they intersect on $X$. The wagons of this pretrain
are stars whose centres are in $X$. In the proofs of Lemma~\ref{lem:62} and 
Lemma~\ref{lem:632} we used the box principle in order to find Ramsey objects 
for the wagons and we applied Ramsey's theorem or Theorem~\ref{thm:631} to the 
hypergraphs describing how the wagons intersect each other. More generally, 
when we have two constructions $\Phi$, $\Psi$ applicable to hypergraphs we 
can similarly define a construction $\Ext(\Phi, \Psi)$ applicable to (certain) 
pretrains (see~\cite{girth}*{Section~6}). 

The way in which the extension process enters the proof of the girth Ramsey 
theorem is quite unrelated to Lemma~\ref{lem:632}. Suppose that we want to perform
any partite construction over a linear hypergraph $G$. Now any two constituents 
of our pictures will either be vertex-disjoint, or they share a unique music 
line. Suppose further that in each step of the construction the partite lemma 
we use delivers a system of partite copies satisfying 
Proposition~\ref{prop:65}\ref{it:65iii}. These mild assumptions already cause 
severe limitations as to how the constituents can `develop' 
in the course of the construction. In picture zero, every constituent is a 
perfect matching (augmented by some isolated vertices). The constituents of the 
next picture are either disjoint unions of such matchings or they arise from 
such unions by identifying some vertices on a common music line, so that they 
look like Figure~\ref{fig64a}. Similarly, the most general constituent of the 
next picture is shown in Figure~\ref{fig64b}.   
     
\begin{figure}[ht]
	\centering	
	
		\begin{subfigure}[b]{0.39\textwidth}
		\centering
	
			\begin{tikzpicture}[scale=.55]
	
\draw (-4.5,1) -- (4.5,1);
\draw (-4.5,0)--(4.5,0);
\draw (-4.5,-1)--(4.5,-1);
\phantom{\draw [ultra thick](-4,-1.35)--(3,1.35);}

\draw [red!80!black, thick ] (0,1)--(0,-1);
\draw [red!80!black, thick ] (1,1)--(1,-1);
\draw [red!80!black, thick ] (-1,1)--(-1,-1);

\draw [red!80!black, thick ] (-1,1)--(-2,-1);
\draw [red!80!black, thick ] (-2,1)--(-3,-1);
\draw [red!80!black, thick ] (-3,1)--(-4,-1);

\draw [red!80!black, thick ] (1,1)--(2,-1);
\draw [red!80!black, thick ] (2,1)--(3,-1);
\draw [red!80!black, thick ] (3,1)--(4,-1);

\draw [rounded corners, green!40!black, ultra thick] (-1.2, 1.15)--(-1.2,-1.15)--(1.2,-1.15)--(1.2,1.15)--cycle;

\draw [rounded corners, green!40!black,ultra thick] (-3.15, 1.15)--(-4.25,-1.15)--(-1.85,-1.15)--(-0.75,1.15)--cycle;

\draw [rounded corners, green!40!black, ultra thick] (3.15, 1.15)--(4.25,-1.15)--(1.85,-1.15)--(0.75,1.15)--cycle;

			\end{tikzpicture}
		 \caption{}
		\label{fig64a} 		
	\end{subfigure}
\hfill    
\begin{subfigure}[b]{0.6\textwidth}
\centering
			\begin{tikzpicture}[scale=.55]

			\draw (-8.5,1) -- (8.5,1);
			\draw (-8.5,0)--(8.5,0);
			\draw (-8.5,-1)--(8.5,-1);
			
	\draw [red!80!black, thick ] (.5,1)--(.5,-1);
	\draw [red!80!black, thick ] (-.5,1)--(-.5,-1);		
	\draw [red!80!black, thick ] (-.5,1)--(-1.5,-1);
	\draw [red!80!black, thick ] (-1.5,1)--(-2.5,-1);		
	\draw [red!80!black, thick ] (.5,1)--(1.5,-1);
	\draw [red!80!black, thick ] (1.5,1)--(2.5,-1);
			
\draw [shift=({5,0}), red!80!black, thick ] (.5,1)--(.5,-1);
\draw [shift=({5,0}), red!80!black, thick ] (-.5,1)--(-.5,-1);		
\draw [shift=({5,0}), red!80!black, thick ] (-.5,1)--(-1.5,-1);
\draw [shift=({5,0}), red!80!black, thick ] (-1.5,1)--(-2.5,-1);		
\draw [shift=({5,0}), red!80!black, thick ] (.5,1)--(1.5,-1);
\draw [shift=({5,0}), red!80!black, thick ] (1.5,1)--(2.5,-1);

\draw [shift=({-5,0}), red!80!black, thick ] (.5,1)--(.5,-1);
\draw [shift=({-5,0}), red!80!black, thick ] (-.5,1)--(-.5,-1);		
\draw [shift=({-5,0}), red!80!black, thick ] (-.5,1)--(-1.5,-1);
\draw [shift=({-5,0}), red!80!black, thick ] (-1.5,1)--(-2.5,-1);		
\draw [shift=({-5,0}), red!80!black, thick ] (.5,1)--(1.5,-1);
\draw [shift=({-5,0}), red!80!black, thick ] (1.5,1)--(2.5,-1);

		\draw [rounded corners, green!40!black, ultra thick] (-.75, 1.15)--(-.75,-1.15)--(.75,-1.15)--(.75,1.15)--cycle;
		\draw [rounded corners, green!40!black,ultra thick] (-1.75, 1.15)--(-2.85,-1.15)--(-1.2,-1.15)--(-0.25,1.15)--cycle;		
		\draw [rounded corners, green!40!black, ultra thick] (1.75, 1.15)--(2.85,-1.15)--(1.2,-1.15)--(0.25,1.15)--cycle;
			
		\draw [shift=({5,0}), rounded corners, green!40!black, ultra thick] (-.75, 1.15)--(-.75,-1.15)--(.75,-1.15)--(.75,1.15)--cycle;
		\draw [shift=({5,0}), rounded corners, green!40!black,ultra thick] (-1.75, 1.15)--(-2.85,-1.15)--(-1.2,-1.15)--(-0.25,1.15)--cycle;		
		\draw [shift=({5,0}), rounded corners, green!40!black, ultra thick] (1.75, 1.15)--(2.85,-1.15)--(1.2,-1.15)--(0.25,1.15)--cycle;
			
		\draw [shift=({-5,0}), rounded corners, green!40!black, ultra thick] (-.75, 1.15)--(-.75,-1.15)--(.75,-1.15)--(.75,1.15)--cycle;
		\draw [shift=({-5,0}), rounded corners, green!40!black,ultra thick] (-1.75, 1.15)--(-2.85,-1.15)--(-1.2,-1.15)--(-0.25,1.15)--cycle;		
		\draw [shift=({-5,0}), rounded corners, green!40!black, ultra thick] (1.75, 1.15)--(2.85,-1.15)--(1.2,-1.15)--(0.25,1.15)--cycle;	
			
			\draw [rounded corners, violet!80!blue, ultra thick] (-1.9,1.35)--(-3.2,-1.35)--(3.2,-1.35)--(1.9,1.35)--cycle;
			\draw [shift=({5,0}), rounded corners, violet!80!blue, ultra thick] (-1.9,1.35)--(-3.2,-1.35)--(3.2,-1.35)--(1.9,1.35)--cycle;
			\draw [shift=({-5,0}), rounded corners, violet!80!blue, ultra thick] (-1.9,1.35)--(-3.2,-1.35)--(3.2,-1.35)--(1.9,1.35)--cycle;
			
		\end{tikzpicture}
	\caption{}
	\label{fig64b} 		
\end{subfigure}

				\caption{Two 3-uniform trains}
				\label{fig64}

	\end{figure}
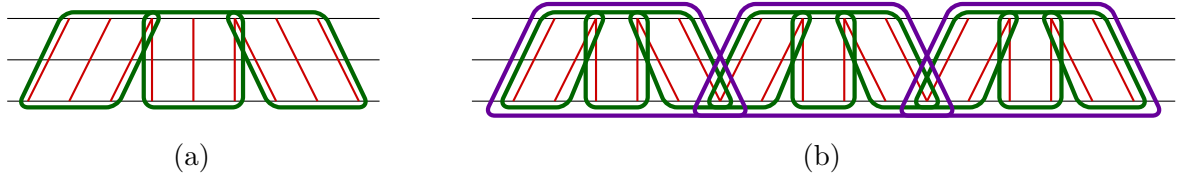 
In general, hypergraphs of this form are called {\it trains}. Officially a train 
is a hypergraph equipped with a nested sequence of equivalence relations 
satisfying some rules on intersections of edges. Since the constituents of pictures
are trains, it suffices to study partite lemmata applicable to trains. These can 
be obtained by iterative applications of the extension process. For further ideas 
and details we refer to~\cite{girth}.  

\subsection{Infinite structural Ramsey theory} \label{subsec:64}
We would finally like to talk about some results on the question whether 
the induced Ramsey theorem generalises to the transfinite setting. Given 
a (finite or infinite) graph $F$ and a cardinal $\mu$ one would like to 
have a graph $H$ such that $H\lra (F)^e_\mu$. An early result of Hajnal 
and Komj\'ath~\cites{HK88, HK92} shows that, consistently, such a graph $H$ does 
not always exist. Notably, they showed that adding a Cohen real also adds
a bipartite graph $F$ on $\aleph_1$ vertices such that $H\nlra (F)^e_2$ 
holds for all graphs~$H$ in the generic extension. Later Komj\'ath~\cite{kom94} 
found a surprisingly simple proof that any non-trivial forcing whose conditions 
form a set adds an uncountable graph $F$ such that for some cardinal~$\mu$ there 
is no graph $H$ with $H\lra (F)^e_\mu$. This is complemented by a deep result of Shelah~\cite{Sh289}, which is proved by means of a difficult proper class forcing. 

\begin{thm}[Shelah]
	It is consistent with $\mathrm{ZFC}$ that for every graph $F$ and every 
	cardinal~$\mu$ there exists a graph $H$ such that $H\lra (F)^e_\mu$. \qed
\end{thm} 

Careful readers will have observed that the aforementioned negative consistency 
results involve uncountable graphs $F$ only. This leaves some room for $\mathrm{ZFC}$
theorems addressing `small graphs'. Building on the ideas in~\cite{EHP75} and 
transferring them into a partite setting, Hajnal~\cite{Hajnal} clarified the 
situation for finite graphs. 

\begin{thm}[Hajnal]\label{thm:Hajnal}
	For every finite graph $F$ and every cardinal $\mu$ there exists a graph~$H$
	such that $H\lra (F)^e_\mu$. \qed
\end{thm}

But what about countable graphs? Here the case of finitely many colours 
was already addressed in~\cite{EHP75}.

\begin{thm}[Erd\H{o}s, Hajnal \& P\'osa]
	For every countable graph $F$ and every natural number $r$ there is a graph $H$
	such that $H\lra (F)^e_r$. \qed
\end{thm}

For infinitely many colours the problem is open and 
Shelah~\cite{Sh666}*{Question 8.12} calls it a ``mystery''.

\begin{quest}
	Is it provable, in $\mathrm{ZFC}$, that for every countable graph $F$ 
	there exists a graph~$H$ such that $H\lra (F)^e_\omega$?
\end{quest}

We conclude with an old problem of Erd\H{o}s related to 
Proposition~\ref{prop:65}\ref{it:65ii}, whose original source we have forgotten. 
But it is restated in~\cite{Sh666}*{Question~8.11}.

\begin{quest}[Erd\H{o}s]
	Does there provably exist a $K_4$-free graph $H$ such that 
	$H\lra (K_3)^e_\omega$?
\end{quest}

By Shelah~\cite{Sh289} the existence of such graphs is consistent (even 
for arbitrarily many colours). It would also be interesting to derive a 
positive answer from $\mathrm{GCH}$ or from $V=L$. 

Of course the real question is whether 
Theorem~\ref{thm:Hajnal} remains valid when we add the demand $\omega(H)=\omega(F)$.
It is certainly impossible to achieve such a result for girth instead of the clique 
number. For instance, if $F$ fails to be bipartite, then every every $C_4$-free 
graph~$H$ satisfies $H\nlra (F)^e_\omega$; this is because Theorem~\ref{thm:EHC4}
yields $\chi(H)\le \aleph_0$, wherefore $H$ is a union of countably many 
bipartite graphs. However, it still seems conceivable that for every 
cardinal~$\mu$ there could be a graph $H$ such that $H\lra (F)^e_\mu$ and 
the shortest odd cycles in $F$ and $H$ have the same length. 

More generally, one would hope to find a transfinite analogue of 
Theorem~\ref{thm:1522}.
So given a finite graph $F$ the question is which finite configurations 
of copies of $F$ need to be present in systems $\ccH$ with $\ccH\lra (F)^e_\mu$, 
when $\mu$ gets arbitrarily large. This kind of `transfinite girth Ramsey theory' is 
certainly a very challenging subject. Nevertheless, there are no convincing
reasons to believe that it is more difficult than finite girth Ramsey theory.

\subsection*{Acknowledgements} It is a great pleasure to 
thank {\sc Joanna Polcyn} for the wonderful graphical illustrations,
and guest editor {\sc Vojt\v{e}ch R\"{o}dl} for the invitation to 
contribute to this volume. Furthermore, we would like to thank 
{\sc Sevda Guliyeva}~\cite{Sevda}, {\sc Max Pitz}, 
and {\sc Vojt\v{e}ch R\"odl} for interesting discussions.


\begin{bibdiv}
\begin{biblist}
\bib{AH78}{article}{
   author={Abramson, Fred G.},
   author={Harrington, Leo A.},
   title={Models without indiscernibles},
   journal={J. Symbolic Logic},
   volume={43},
   date={1978},
   number={3},
   pages={572--600},
   issn={0022-4812},
   review={\MR{503795 (80a:03045)}},
   doi={10.2307/2273534},
}

\bib{AHL}{article}{
   author={Alon, Noga},
   author={Hoory, Shlomo},
   author={Linial, Nathan},
   title={The Moore bound for irregular graphs},
   journal={Graphs Combin.},
   volume={18},
   date={2002},
   number={1},
   pages={53--57},
   issn={0911-0119},
   review={\MR{1892433}},
   doi={10.1007/s003730200002},
}

\bib{AKRWZ16}{article}{
   author={Alon, Noga},
   author={Kostochka, Alexandr},
   author={Reiniger, Benjamin},
   author={West, Douglas B.},
   author={Zhu, Xuding},
   title={Coloring, sparseness and girth},
   journal={Israel J. Math.},
   volume={214},
   date={2016},
   number={1},
   pages={315--331},
   issn={0021-2172},
   review={\MR{3540616}},
   doi={10.1007/s11856-016-1361-2},
}


\bib{AR}{article}{
   author={Alon, Noga},
   author={Ruzsa, Imre Z.},
   title={Non-averaging subsets and non-vanishing transversals},
   journal={J. Combin. Theory Ser. A},
   volume={86},
   date={1999},
   number={1},
   pages={1--13},
   issn={0097-3165},
   review={\MR{1682960}},
   doi={10.1006/jcta.1998.2926},
}
	
\bib{Asch71}{article}{
   author={Aschbacher, Michael},
   title={The nonexistence of rank three permutation groups of degree $3250$
   and subdegree $57$},
   journal={J. Algebra},
   volume={19},
   date={1971},
   pages={538--540},
   issn={0021-8693},
   review={\MR{291266}},
   doi={10.1016/0021-8693(71)90087-1},
}

\bib{Avart}{article}{
   author={Avart, Christian},
   author={Kay, Bill},
   author={Reiher, Chr.},
   author={R\"odl, Vojt\v ech},
   title={The chromatic number of finite type-graphs},
   journal={Journal of Combinatorial Theory Series B},
   volume={122},
   date={2017},
   pages={877--896},
   issn={0095-8956},
   review={\MR{3575234}},
   doi={10.1016/j.jctb.2016.10.004},
}

\bib{BI73}{article}{
   author={Bannai, Eiichi},
   author={Ito, Tatsuro},
   title={On finite Moore graphs},
   journal={J. Fac. Sci. Univ. Tokyo Sect. IA Math.},
   volume={20},
   date={1973},
   pages={191--208},
   issn={0040-8980},
   review={\MR{323615}},
}

\bib{BG}{article}{
   author={Baumert, Leonard D.},
   author={Gordon, Daniel M.},
   title={On the existence of cyclic difference sets with small parameters},
   conference={
      title={High primes and misdemeanours: lectures in honour of the 60th
      birthday of Hugh Cowie Williams},
   },
   book={
      series={Fields Inst. Commun.},
      volume={41},
      publisher={Amer. Math. Soc., Providence, RI},
   },
   date={2004},
   pages={61--68},
   review={\MR{2075647}},
}

\bib{BCW}{article}{
   author={Behzad, Mehdi},
   author={Chartrand, Gary},
   author={Wall, Curtiss E.},
   title={On minimal regular digraphs with given girth},
   journal={Fund. Math.},
   volume={69},
   date={1970},
   pages={227--231},
   issn={0016-2736},
   review={\MR{285448}},
   doi={10.4064/fm-69-3-227-231},
}

\bib{Benson}{article}{
   author={Benson, Clark T.},
   title={Minimal regular graphs of girths eight and twelve},
   journal={Canadian J. Math.},
   volume={18},
   date={1966},
   pages={1091--1094},
   issn={0008-414X},
   review={\MR{197342}},
   doi={10.4153/CJM-1966-109-8},
}

\bib{BNRR}{article}{
   author={Bhat, Vindya},
   author={Ne\v set\v ril, Jaroslav},
   author={Reiher, Chr.},
   author={R\"odl, Vojt\v ech},
   title={A Ramsey class for Steiner systems},
   journal={J. Combin. Theory Ser. A},
   volume={154},
   date={2018},
   pages={323--349},
   issn={0097-3165},
   review={\MR{3718069}},
}
	
\bib{Biggs}{book}{
   author={Biggs, Norman},
   title={Algebraic graph theory},
   series={Cambridge Mathematical Library},
   edition={2},
   publisher={Cambridge University Press, Cambridge},
   date={1993},
   pages={viii+205},
   isbn={0-521-45897-8},
   review={\MR{1271140}},
}

\bib{BR}{article}{
	author={Blakley, G. R.},
	author={Roy, Prabir},
	title={A H\"older type inequality for symmetric matrices with nonnegative
					entries},
	journal={Proc. Amer. Math. Soc.},
	volume={16},
	date={1965},
	pages={1244--1245},
	issn={0002-9939},
	review={\MR{0184950}},
}

\bib{Boll-Ex}{book}{
   author={Bollob\'{a}s, B\'{e}la},
   title={Extremal graph theory},
   note={Reprint of the 1978 original},
   publisher={Dover Publications, Inc., Mineola, NY},
   date={2004},
   pages={xx+488},
   isbn={0-486-43596-2},
   review={\MR{2078877}},
}

\bib{Boll-Mod}{book}{
   author={Bollob\'{a}s, B\'{e}la},
   title={Modern graph theory},
   series={Graduate Texts in Mathematics},
   volume={184},
   publisher={Springer-Verlag, New York},
   date={1998},
   pages={xiv+394},
   isbn={0-387-98488-7},
   review={\MR{1633290}},
   doi={10.1007/978-1-4612-0619-4},
}
	
	
\bib{Bondy}{article}{
   author={Bondy, J. A.},
   title={Counting subgraphs: a new approach to the Caccetta-H\"{a}ggkvist
   conjecture},
   note={Graphs and combinatorics (Marseille, 1995)},
   journal={Discrete Math.},
   volume={165/166},
   date={1997},
   pages={71--80},
   issn={0012-365X},
   review={\MR{1439261}},
   doi={10.1016/S0012-365X(96)00162-8},
}


\bib{CH78}{article}{
   author={Caccetta, Louis},
   author={H\"{a}ggkvist, R.},
   title={On minimal digraphs with given girth},
   conference={
      title={Proceedings of the Ninth Southeastern Conference on
      Combinatorics, Graph Theory, and Computing},
      address={Florida Atlantic Univ., Boca Raton, Fla.},
      date={1978},
   },
   book={
      series={Congress. Numer., XXI},
      publisher={Utilitas Math., Winnipeg, MB},
   },
   date={1978},
   pages={181--187},
   review={\MR{527946}},
}

\bib{Cameron}{book}{
   author={Cameron, Peter J.},
   title={Permutation groups},
   series={London Mathematical Society Student Texts},
   volume={45},
   publisher={Cambridge University Press, Cambridge},
   date={1999},
   pages={x+220},
   isbn={0-521-65302-9},
   isbn={0-521-65378-9},
   review={\MR{1721031}},
   doi={10.1017/CBO9780511623677},
}

\bib{CHMS}{article}{
   author={Carbonero, Alvaro},
   author={Hompe, Patrick},
   author={Moore, Benjamin},
   author={Spirkl, Sophie},
   title={A counterexample to a conjecture about triangle-free induced
   subgraphs of graphs with large chromatic number},
   journal={J. Combin. Theory Ser. B},
   volume={158},
   date={2023},
   number={part 2},
   part={part 2},
   pages={63--69},
   issn={0095-8956},
   review={\MR{4484828}},
   doi={10.1016/j.jctb.2022.09.001},
}

\bib{CSS}{article}{
   author={Chudnovsky, Maria},
   author={Seymour, Paul},
   author={Sullivan, Blair},
   title={Cycles in dense digraphs},
   journal={Combinatorica},
   volume={28},
   date={2008},
   number={1},
   pages={1--18},
   issn={0209-9683},
   review={\MR{2399005}},
   doi={10.1007/s00493-008-2331-z},
}

\bib{CS83}{article}{
   author={Chv\'{a}tal, V.},
   author={Szemer\'{e}di, E.},
   title={Short cycles in directed graphs},
   journal={J. Combin. Theory Ser. B},
   volume={35},
   date={1983},
   number={3},
   pages={323--327},
   issn={0095-8956},
   review={\MR{735200}},
   doi={10.1016/0095-8956(83)90059-X},
}


\bib{CFS10}{article}{
   author={Conlon, David},
   author={Fox, Jacob},
   author={Sudakov, Benny},
   title={An approximate version of Sidorenko's conjecture},
   journal={Geom. Funct. Anal.},
   volume={20},
   date={2010},
   number={6},
   pages={1354--1366},
   issn={1016-443X},
   review={\MR{2738996}},
   doi={10.1007/s00039-010-0097-0},
}

\bib{CKLL18}{article}{
   author={Conlon, David},
   author={Kim, Jeong Han},
   author={Lee, Choongbum},
   author={Lee, Joonkyung},
   title={Some advances on Sidorenko's conjecture},
   journal={J. Lond. Math. Soc. (2)},
   volume={98},
   date={2018},
   number={3},
   pages={593--608},
   issn={0024-6107},
   review={\MR{3893193}},
   doi={10.1112/jlms.12142},
}

\bib{CL17}{article}{
   author={Conlon, David},
   author={Lee, Joonkyung},
   title={Finite reflection groups and graph norms},
   journal={Adv. Math.},
   volume={315},
   date={2017},
   pages={130--165},
   issn={0001-8708},
   review={\MR{3667583}},
   doi={10.1016/j.aim.2017.05.009},
}

\bib{CL21}{article}{
   author={Conlon, David},
   author={Lee, Joonkyung},
   title={Sidorenko's conjecture for blow-ups},
   journal={Discrete Anal.},
   date={2021},
   pages={Paper No. 2, 13},
   review={\MR{4237083}},
   doi={10.19086/da},
}

\bib{Dalfo}{article}{
   author={Dalf\'{o}, C.},
   title={A survey on the missing Moore graph},
   journal={Linear Algebra Appl.},
   volume={569},
   date={2019},
   pages={1--14},
   issn={0024-3795},
   review={\MR{3901732}},
   doi={10.1016/j.laa.2018.12.035},
}

\bib{Dam}{article}{
   author={Damerell, R. M.},
   title={On Moore graphs},
   journal={Proc. Cambridge Philos. Soc.},
   volume={74},
   date={1973},
   pages={227--236},
   issn={0008-1981},
   review={\MR{318004}},
   doi={10.1017/s0305004100048015},
}

\bib{Davenport}{book}{
   author={Davenport, Harold},
   title={Multiplicative number theory},
   series={Graduate Texts in Mathematics},
   volume={74},
   edition={3},
   note={Revised and with a preface by Hugh L. Montgomery},
   publisher={Springer-Verlag, New York},
   date={2000},
   pages={xiv+177},
   isbn={0-387-95097-4},
   review={\MR{1790423}},
}

\bib{Deuber75}{article}{
   author={Deuber, W.},
   title={Generalizations of Ramsey's theorem},
   conference={
      title={Infinite and finite sets (Colloq., Keszthely, 1973; dedicated
      to P. Erd\H os on his 60th birthday), Vol. I},
   },
   book={
      publisher={North-Holland, Amsterdam},
   },
   date={1975},
   pages={323--332. Colloq. Math. Soc. J\'anos Bolyai, Vol. 10},
   review={\MR{0369127 (51 \#5363)}},
}

\bib{DHKZ}{article}{
	author={Diskin, Sahar},
	author={Hoshen, Ilay},
	author={Krivelevich, Michael},
	author={Zhukovskii, Maksim},
	title={On vertex Ramsey graphs with forbidden subgraphs},
	eprint={2211.13966},
}

\bib{DHP}{article}{
   author={Dunkum, Molly},
   author={Hamburger, Peter},
   author={P\'{o}r, Attila},
   title={Destroying cycles in digraphs},
   journal={Combinatorica},
   volume={31},
   date={2011},
   number={1},
   pages={55--66},
   issn={0209-9683},
   review={\MR{2847876}},
   doi={10.1007/s00493-011-2589-4},
}

\bib{Erd59}{article}{
   author={Erd\H{o}s, P.},
   title={Graph theory and probability},
   journal={Canadian J. Math.},
   volume={11},
   date={1959},
   pages={34--38},
   issn={0008-414X},
   review={\MR{102081}},
   doi={10.4153/CJM-1959-003-9},
}

\bib{Erd68}{article}{
   author={Erd\H{o}s, P.},
   title={Problems and results in chromatic graph theory},
   conference={
      title={Proof Techniques in Graph Theory},
      address={Proc. Second Ann Arbor Graph Theory Conf., Ann Arbor, Mich.},
      date={1968},
   },
   book={
      publisher={Academic Press, New York-London},
   },
   date={1969},
   pages={27--35},
   review={\MR{252273}},
}

\bib{Erd75}{article}{
   author={Erd\H{o}s, P.},
   title={Problems and results on finite and infinite graphs},
   conference={
      title={Recent advances in graph theory},
      address={Proc. Second Czechoslovak Sympos., Prague},
      date={1974},
   },
   book={
      publisher={Academia, Prague},
   },
   date={1975},
   pages={183--192. (loose errata)},
   review={\MR{0389669}},
}

\bib{EGH}{article}{
   author={Erd\H{o}s, P.},
   author={Galvin, F.},
   author={Hajnal, A.},
   title={On set-systems having large chromatic number and not containing
   prescribed subsystems},
   conference={
      title={Infinite and finite sets (Colloq., Keszthely, 1973; dedicated
      to P. Erd\H{o}s on his 60th birthday), Vols. I, II, III},
   },
   book={
      series={Colloq. Math. Soc. J\'{a}nos Bolyai, Vol. 10},
      publisher={North-Holland, Amsterdam-London},
   },
   date={1975},
   pages={425--513},
   review={\MR{398876}},
}
	
\bib{EH64}{article}{
   author={Erd\H{o}s, P.},
   author={Hajnal, A.},
   title={Some remarks on set theory. IX. Combinatorial problems in measure
   theory and set theory},
   journal={Michigan Math. J.},
   volume={11},
   date={1964},
   pages={107--127},
   issn={0026-2285},
   review={\MR{171713}},
}

\bib{EH66}{article}{
   author={Erd\H{o}s, P.},
   author={Hajnal, A.},
   title={On chromatic number of graphs and set-systems},
   journal={Acta Math. Acad. Sci. Hungar.},
   volume={17},
   date={1966},
   pages={61--99},
   issn={0001-5954},
   review={\MR{193025}},
   doi={10.1007/BF02020444},
}

\bib{EHMR}{book}{
   author={Erd\H{o}s, Paul},
   author={Hajnal, Andr\'{a}s},
   author={M\'{a}t\'{e}, Attila},
   author={Rado, Richard},
   title={Combinatorial set theory: partition relations for cardinals},
   series={Studies in Logic and the Foundations of Mathematics},
   volume={106},
   publisher={North-Holland Publishing Co., Amsterdam},
   date={1984},
   pages={347},
   isbn={0-444-86157-2},
   review={\MR{795592}},
}

\bib{EHP75}{article}{
   author={Erd{\H{o}}s, P.},
   author={Hajnal, A.},
   author={P{\'o}sa, L.},
   title={Strong embeddings of graphs into colored graphs},
   conference={
      title={Infinite and finite sets (Colloq., Keszthely, 1973; dedicated
      to P. Erd\H os on his 60th birthday), Vol. I},
   },
   book={
      publisher={North-Holland, Amsterdam},
   },
   date={1975},
   pages={585--595. Colloq. Math. Soc. J\'anos Bolyai, Vol. 10},
   review={\MR{0382049 (52 \#2937)}},
}

\bib{EHRoth}{article}{
   author={Erd\H{o}s, P.},
   author={Hajnal, A.},
   author={Rothchild, B.},
   title={``On chromatic number of graphs and set-systems'' (Acta Math.
   Acad. Sci. Hungar. {\bf 17} (1966), 61--99) by Erd\H{o}s and Hajnal},
   conference={
      title={Cambridge Summer School in Mathematical Logic},
      address={Cambridge},
      date={1971},
   },
   book={
      series={Lecture Notes in Math., Vol. 337},
      publisher={Springer, Berlin-New York},
   },
   date={1973},
   pages={531--538},
   review={\MR{387103}},
}

\bib{EHS}{article}{
   author={Erd\H{o}s, P.},
   author={Hajnal, A.},
   author={Shelah, S.},
   title={On some general properties of chromatic numbers},
   conference={
      title={Topics in topology},
      address={Proc. Colloq., Keszthely},
      date={1972},
   },
   book={
      series={Colloq. Math. Soc. J\'{a}nos Bolyai, Vol. 8},
      publisher={North-Holland, Amsterdam-London},
   },
   date={1974},
   pages={243--255},
   review={\MR{357194}},
}

\bib{ER56}{article}{
   author={Erd\H{o}s, P.},
   author={Rado, R.},
   title={A partition calculus in set theory},
   journal={Bull. Amer. Math. Soc.},
   volume={62},
   date={1956},
   pages={427--489},
   issn={0002-9904},
   review={\MR{81864}},
   doi={10.1090/S0002-9904-1956-10036-0},
}

\bib{ER59}{article}{
   author={Erd\H{o}s, P.},
   author={Rado, R.},
   title={Partition relations connected with the chromatic number of graphs},
   journal={J. London Math. Soc.},
   volume={34},
   date={1959},
   pages={63--72},
   issn={0024-6107},
   review={\MR{101845}},
   doi={10.1112/jlms/s1-34.1.63},
}

\bib{ER60}{article}{
   author={Erd\H{o}s, P.},
   author={Rado, R.},
   title={A construction of graphs without triangles having preassigned
   order and chromatic number},
   journal={J. London Math. Soc.},
   volume={35},
   date={1960},
   pages={445--448},
   issn={0024-6107},
   review={\MR{0140433 (25 \#3853)}},
}

\bib{ES63}{article}{
   author={Erd\H{o}s, Paul},
   author={Sachs, Horst},
   title={Regul\"{a}re Graphen gegebener Taillenweite mit minimaler Knotenzahl},
   language={German},
   journal={Wiss. Z. Martin-Luther-Univ. Halle-Wittenberg Math.-Natur.
   Reihe},
   volume={12},
   date={1963},
   pages={251--257},
   issn={0138-1504},
   review={\MR{165515}},
}

\bib{EJS}{article}{
   author={Exoo, Geoffrey},
   author={Jajcay, Robert},
   title={Dynamic cage survey},
   journal={Electron. J. Combin.},
   volume={DS16},
   date={2008},
   number={Dynamic Surveys},
   pages={48},
   review={\MR{4336218}},
}
	
\bib{FH}{article}{
   author={Feit, Walter},
   author={Higman, Graham},
   title={The nonexistence of certain generalized polygons},
   journal={J. Algebra},
   volume={1},
   date={1964},
   pages={114--131},
   issn={0021-8693},
   review={\MR{170955}},
   doi={10.1016/0021-8693(64)90028-6},
}

\bib{Fitch}{article}{
   author={Fitch, Matthew},
   title={Rational exponents for hypergraph Tur\'{a}n problems},
   journal={J. Comb.},
   volume={10},
   date={2019},
   number={1},
   pages={61--86},
   issn={2156-3527},
   review={\MR{3890916}},
   doi={10.4310/joc.2019.v10.n1.a3},
}

\bib{Folk}{article}{
   author={Folkman, Jon},
   title={Graphs with monochromatic complete subgraphs in every edge
   coloring},
   journal={SIAM J. Appl. Math.},
   volume={18},
   date={1970},
   pages={19--24},
   issn={0036-1399},
   review={\MR{268080}},
   doi={10.1137/0118004},
}
	
\bib{GIP}{article}{
	author={Gir\~{a}o, A.},
	author={Illingworth, F.},
	author={Powierski, E.},
	author={Savery, M.},
	author={Scott, A.},
	author={Tamitegama, Y.},
	author={Tan, J.},
	title={Induced subgraphs of induced subgraphs of large chromatic number},
	eprint={2203.03612},
}
		
\bib{Gowers}{article}{
   author={Gowers, W. T.},
   title={Probabilistic combinatorics and the recent work of Peter Keevash},
   journal={Bull. Amer. Math. Soc. (N.S.)},
   volume={54},
   date={2017},
   number={1},
   pages={107--116},
   issn={0273-0979},
   review={\MR{3584100}},
   doi={10.1090/bull/1553},
}

\bib{GV}{article}{
   author={Grzesik, Andrzej},
   author={Volec, Jan},
   title={Degree conditions forcing directed cycles},
   journal={Int. Math. Res. Not. IMRN},
   date={2023},
   number={11},
   pages={9711--9753},
   issn={1073-7928},
   review={\MR{4597217}},
   doi={10.1093/imrn/rnac114},
}

\bib{Sevda}{article}{
	author={Guliyeva, Sevda},
	title={Sykow's Beweis des Satzes von Tur\'an},
	note={Bachelor thesis written at the University of Hamburg},
	year={2019},
}

\bib{Hajnal}{article}{
   author={Hajnal, A.},
   title={Embedding finite graphs into graphs colored with infinitely many
   colors},
   journal={Israel J. Math.},
   volume={73},
   date={1991},
   number={3},
   pages={309--319},
   issn={0021-2172},
   review={\MR{1135220}},
   doi={10.1007/BF02773844},
}

\bib{HK88}{article}{
   author={Hajnal, A.},
   author={Komj\'{a}th, P.},
   title={Embedding graphs into colored graphs},
   journal={Trans. Amer. Math. Soc.},
   volume={307},
   date={1988},
   number={1},
   pages={395--409},
   issn={0002-9947},
   review={\MR{936824}},
   doi={10.2307/2000770},
}

\bib{HK92}{article}{
   author={Hajnal, A.},
   author={Komj\'{a}th, P.},
   title={Corrigendum to: ``Embedding graphs into colored graphs''},
   journal={Trans. Amer. Math. Soc.},
   volume={332},
   date={1992},
   number={1},
   pages={475},
   issn={0002-9947},
   review={\MR{1140915}},
   doi={10.2307/2154043},
}
	
\bib{HK08}{article}{
   author={Hajnal, A.},
   author={Komj\'{a}th, P.},
   title={Obligatory subsystems of triple systems},
   journal={Acta Math. Hungar.},
   volume={119},
   date={2008},
   number={1-2},
   pages={1--13},
   issn={0236-5294},
   review={\MR{2400791}},
   doi={10.1007/s10474-007-6231-2},
}
	
\bib{HL}{article}{
   author={Hajnal, Andr\'{a}s},
   author={Larson, Jean A.},
   title={Partition relations},
   conference={
      title={Handbook of set theory. Vols. 1, 2, 3},
   },
   book={
      publisher={Springer, Dordrecht},
   },
   date={2010},
   pages={129--213},
   review={\MR{2768681}},
   doi={10.1007/978-1-4020-5764-9\_3},
}
	
\bib{HJ63}{article}{
   author={Hales, A. W.},
   author={Jewett, R. I.},
   title={Regularity and positional games},
   journal={Trans. Amer. Math. Soc.},
   volume={106},
   date={1963},
   pages={222--229},
   issn={0002-9947},
   review={\MR{143712}},
   doi={10.2307/1993764},
}

\bib{Ham81a}{article}{
   author={Hamidoune, Yahya Ould},
   title={An application of connectivity theory in graphs to factorizations
   of elements in groups},
   journal={European J. Combin.},
   volume={2},
   date={1981},
   number={4},
   pages={349--355},
   issn={0195-6698},
   review={\MR{638410}},
   doi={10.1016/S0195-6698(81)80042-X},
}

\bib{HW}{book}{
   author={Hardy, G. H.},
   author={Wright, E. M.},
   title={An introduction to the theory of numbers},
   edition={6},
   note={Revised by D. R. Heath-Brown and J. H. Silverman;
   With a foreword by Andrew Wiles},
   publisher={Oxford University Press, Oxford},
   date={2008},
   pages={xxii+621},
   isbn={978-0-19-921986-5},
   review={\MR{2445243}},
}
		
\bib{HS60}{article}{
   author={Hoffman, A. J.},
   author={Singleton, R. R.},
   title={On Moore graphs with diameters $2$ and $3$},
   journal={IBM J. Res. Develop.},
   volume={4},
   date={1960},
   pages={497--504},
   issn={0018-8646},
   review={\MR{140437}},
   doi={10.1147/rd.45.0497},
}

\bib{Hoory}{article}{
	author={Hoory, Shlomo},
	title={On the girth of graph lifts},
	eprint={2401.01238},
}

\bib{HN19}{article}{
   author={Hubi\v{c}ka, Jan},
   author={Ne\v{s}et\v{r}il, Jaroslav},
   title={All those Ramsey classes (Ramsey classes with closures and
   forbidden homomorphisms)},
   journal={Adv. Math.},
   volume={356},
   date={2019},
   pages={106791, 89},
   issn={0001-8708},
   review={\MR{4001036}},
   doi={10.1016/j.aim.2019.106791},
}

\bib{Jacobi}{book}{
   author={Jacobi, Carl Gustav Jacob},
   title={Fundamenta nova theoriae functionum ellipticarum},
   date={1829},
   note={Regiomonti, sumtibus fratrum Borntraeger},
}

\bib{Jech}{book}{
   author={Jech, Thomas},
   title={Set theory},
   series={Springer Monographs in Mathematics},
   note={The third millennium edition, revised and expanded},
   publisher={Springer-Verlag, Berlin},
   date={2003},
   pages={xiv+769},
   isbn={3-540-44085-2},
   review={\MR{1940513}},
}

	
\bib{Ted}{article}{
   author={Kaczynski, T. J.},
   title={Mathematical Notes: Another Proof of Wedderburn's Theorem},
   journal={Amer. Math. Monthly},
   volume={71},
   date={1964},
   number={6},
   pages={652--653},
   issn={0002-9890},
   review={\MR{1532764}},
   doi={10.2307/2312328},
}

\bib{Kart}{article}{
   author={K\`arteszi, Francesco},
   title={Piani finiti ciclici come risoluzioni di un certo problema di
   minimo},
   language={Italian},
   journal={Boll. Un. Mat. Ital. (3)},
   volume={15},
   date={1960},
   pages={522--528},
   review={\MR{145511}},
}

\bib{kom94}{article}{
   author={Komj\'{a}th, P.},
   title={Ramsey theory and forcing extensions},
   journal={Proc. Amer. Math. Soc.},
   volume={121},
   date={1994},
   number={1},
   pages={217--219},
   issn={0002-9939},
   review={\MR{1169039}},
   doi={10.2307/2160385},
}

\bib{kom01}{article}{
   author={Komj\'{a}th, P.},
   title={Some remarks on obligatory subsystems of uncountably chromatic
   triple systems},
   note={Paul Erd\H{o}s and his mathematics (Budapest, 1999)},
   journal={Combinatorica},
   volume={21},
   date={2001},
   number={2},
   pages={233--238},
   issn={0209-9683},
   review={\MR{1832448}},
   doi={10.1007/s004930100021},
}
	
\bib{kom08}{article}{
   author={Komj\'{a}th, P.},
   title={An uncountably chromatic triple system},
   journal={Acta Math. Hungar.},
   volume={121},
   date={2008},
   number={1-2},
   pages={79--92},
   issn={0236-5294},
   review={\MR{2463251}},
   doi={10.1007/s10474-008-7179-6},
}

\bib{kom-survey}{article}{
   author={Komj\'{a}th, P.},
   title={The chromatic number of infinite graphs---a survey},
   journal={Discrete Math.},
   volume={311},
   date={2011},
   number={15},
   pages={1448--1450},
   issn={0012-365X},
   review={\MR{2800970}},
   doi={10.1016/j.disc.2010.11.004},
}

\bib{KS05}{article}{
   author={Komj{\'a}th, P.},
   author={Shelah, Saharon},
   title={Finite subgraphs of uncountably chromatic graphs},
   journal={J. Graph Theory},
   volume={49},
   date={2005},
   number={1},
   pages={28--38},
   issn={0364-9024},
   review={\MR{2130468 (2005k:05096)}},
   doi={10.1002/jgt.20060},
}
		


\bib{KN99}{article}{
   author={Kostochka, A. V.},
   author={Ne\v{s}et\v{r}il, J.},
   title={Properties of Descartes' construction of triangle-free graphs with
   high chromatic number},
   journal={Combin. Probab. Comput.},
   volume={8},
   date={1999},
   number={5},
   pages={467--472},
   issn={0963-5483},
   review={\MR{1731981}},
   doi={10.1017/S0963548399004022},
}

\bib{Kriz}{article}{
   author={K\v{r}\'{\i}\v{z}, Igor},
   title={A hypergraph-free construction of highly chromatic graphs without
   short cycles},
   journal={Combinatorica},
   volume={9},
   date={1989},
   number={2},
   pages={227--229},
   issn={0209-9683},
   review={\MR{1030376}},
   doi={10.1007/BF02124683},
}

\bib{Kunen}{book}{
   author={Kunen, Kenneth},
   title={Set theory},
   series={Studies in Logic (London)},
   volume={34},
   publisher={College Publications, London},
   date={2011},
   pages={viii+401},
   isbn={978-1-84890-050-9},
   review={\MR{2905394}},
}
	

\bib{LUW}{article}{
   author={Lazebnik, F.},
   author={Ustimenko, V. A.},
   author={Woldar, A. J.},
   title={New upper bounds on the order of cages},
   note={The Wilf Festschrift (Philadelphia, PA, 1996)},
   journal={Electron. J. Combin.},
   volume={4},
   date={1997},
   number={2},
   pages={Research Paper 13, approx. 11},
   review={\MR{1444160}},
   doi={10.37236/1328},
}
	
\bib{Lee}{article}{
	author={Lee, Joonkyung}, 
	title={On some graph densities in locally dense graphs}, 
		note={Random Structures \& Algorithms},
	volume={58},
	date={2021},
    number={2},
    pages={322--344},
    issn={1042-9832},
	doi={10.1002/rsa.20974},
}

\bib{LS21}{article}{
   author={Lee, Joonkyung},
   author={Sch\"{u}lke, Bjarne},
   title={Convex graphon parameters and graph norms},
   journal={Israel J. Math.},
   volume={242},
   date={2021},
   number={2},
   pages={549--563},
   issn={0021-2172},
   review={\MR{4282091}},
   doi={10.1007/s11856-021-2112-6},
}

\bib{LS11}{article}{
	author={Li, J.~L.~Xiang},
	author={Szegedy, Balazs},
	title={On the logarithmic calculus and Sidorenko's conjecture},
	eprint={1107.1153},
}

\bib{Longyear}{article}{
   author={Longyear, Judith Q.},
   title={Regular $d$-valent graphs of girth $6$ and $2(d^{2}-d+1)$
   vertices},
   journal={J. Combinatorial Theory},
   volume={9},
   date={1970},
   pages={420--422},
   issn={0021-9800},
   review={\MR{278983}},
}

\bib{Lov68}{article}{
   author={Lov\'{a}sz, L.},
   title={On chromatic number of finite set-systems},
   journal={Acta Math. Acad. Sci. Hungar.},
   volume={19},
   date={1968},
   pages={59--67},
   issn={0001-5954},
   review={\MR{220621}},
   doi={10.1007/BF01894680},
}

\bib{Lov11}{article}{
   author={Lov\'{a}sz, L.},
   title={Subgraph densities in signed graphons and the local
   Simonovits-Sidorenko conjecture},
   journal={Electron. J. Combin.},
   volume={18},
   date={2011},
   number={1},
   pages={Paper 127, 21},
   review={\MR{2811096}},
   doi={10.37236/614},
}

\bib{LPS88}{article}{
   author={Lubotzky, A.},
   author={Phillips, R.},
   author={Sarnak, P.},
   title={Ramanujan graphs},
   journal={Combinatorica},
   volume={8},
   date={1988},
   number={3},
   pages={261--277},
   issn={0209-9683},
   review={\MR{963118}},
   doi={10.1007/BF02126799},
}

\bib{MS10}{article}{
   author={Ma\v{c}aj, Martin},
   author={\v{S}ir\'{a}\v{n}, Jozef},
   title={Search for properties of the missing Moore graph},
   journal={Linear Algebra Appl.},
   volume={432},
   date={2010},
   number={9},
   pages={2381--2398},
   issn={0024-3795},
   review={\MR{2599868}},
   doi={10.1016/j.laa.2009.07.018},
}


\bib{Ma07}{article}{
   author={Mantel, W.},
   title={Vraagstuk {\rm XXVIII}},
   journal={Wiskundige Opgaven},
   volume={10},
   date={1907},
   pages={60--61},
}

\bib{Ne66}{article}{
   author={Ne\v{s}et\v{r}il, Jaroslav},
   title={$K$-\rn{khromaticheskie grafy bez tsiklov dliny} $\leq 7$},
   language={Russian},
   journal={Comment. Math. Univ. Carolinae},
   volume={7},
   date={1966},
   pages={373--376},
   issn={0010-2628},
   review={\MR{201346}},
}

\bib{jarik-book}{book}{
   author={Ne\v{s}et\v{r}il, Jaroslav},
   title={Theorie graf\r{u}},
   note={Vyd.~1, St\'atn\'i Nakladatelstv\'i Technick\'e Literatury, Praha},
   date={1979},
}

\bib{Ne09}{article}{
   author={Ne\v{s}et\v{r}il, Jaroslav},
   title={A surprising permanence of old motivations (a not-so-rigid story)},
   journal={Discrete Math.},
   volume={309},
   date={2009},
   number={18},
   pages={5510--5526},
   issn={0012-365X},
   review={\MR{2567953}},
   doi={10.1016/j.disc.2008.04.055},
}

\bib{NR76b}{article}{
   author={Ne\v{s}et\v{r}il, J.},
   author={R\"{o}dl, V.},
   title={Partitions of vertices},
   journal={Comment. Math. Univ. Carolinae},
   volume={17},
   date={1976},
   number={1},
   pages={85--95},
   issn={0010-2628},
   review={\MR{412044}},
}

\bib{NR76}{article}{
   author={Ne\v{s}et\v{r}il, J.},
   author={R\"{o}dl, V.},
   title={The Ramsey property for graphs with forbidden complete subgraphs},
   journal={J. Combinatorial Theory Ser. B},
   volume={20},
   date={1976},
   number={3},
   pages={243--249},
   review={\MR{0412004 (54 \#133)}},
}	

\bib{NR77}{article}{
   author={Ne\v{s}et\v{r}il, J.},
   author={R\"{o}dl, V.},
   title={Partitions of finite relational and set systems},
   journal={J. Combinatorial Theory Ser. A},
   volume={22},
   date={1977},
   number={3},
   pages={289--312},
   review={\MR{0437351 (55 \#10283)}},
}

\bib{NeRo79}{article}{
   author={Ne\v{s}et\v{r}il, J.},
   author={R\"{o}dl, V.},
   title={A short proof of the existence of highly chromatic hypergraphs
   without short cycles},
   journal={J. Combin. Theory Ser. B},
   volume={27},
   date={1979},
   number={2},
   pages={225--227},
   issn={0095-8956},
   review={\MR{546865}},
   doi={10.1016/0095-8956(79)90084-4},
}

\bib{NR81}{article}{
   author={Ne\v{s}et\v{r}il, J.},
   author={R\"{o}dl, V.},
   title={Simple proof of the existence of restricted Ramsey graphs by means
   of a partite construction},
   journal={Combinatorica},
   volume={1},
   date={1981},
   number={2},
   pages={199--202},
   issn={0209-9683},
   review={\MR{625551 (83a:05101)}},
   doi={10.1007/BF02579274},
}

\bib{NR82}{article}{
   author={Ne\v{s}et\v{r}il, J.},
   author={R\"{o}dl, V.},
   title={Two proofs of the Ramsey property of the class of finite
   hypergraphs},
   journal={European J. Combin.},
   volume={3},
   date={1982},
   number={4},
   pages={347--352},
   issn={0195-6698},
   review={\MR{687733 (85b:05134)}},
   doi={10.1016/S0195-6698(82)80019-X},
}

\bib{NR87}{article}{
   author={Ne\v{s}et\v{r}il, J.},
   author={R\"{o}dl, V.},
   title={Strong Ramsey theorems for Steiner systems},
   journal={Trans. Amer. Math. Soc.},
   volume={303},
   date={1987},
   number={1},
   pages={183--192},
   issn={0002-9947},
   review={\MR{896015 (89b:05127)}},
   doi={10.2307/2000786},
}

\bib{Peluse}{article}{
   author={Peluse, Sarah},
   title={An asymptotic version of the prime power conjecture for perfect
   difference sets},
   journal={Math. Ann.},
   volume={380},
   date={2021},
   number={3-4},
   pages={1387--1425},
   issn={0025-5831},
   review={\MR{4297189}},
   doi={10.1007/s00208-021-02188-5},
}

\bib{tyh}{article}{
   author={Polcyn, Joanna},
   author={Reiher, Chr.},
   author={R\"{o}dl, Vojt\v{e}ch},
   author={Sch\"{u}lke, Bjarne},
   title={On Hamiltonian cycles in hypergraphs with dense link graphs},
   journal={J. Combin. Theory Ser. B},
   volume={150},
   date={2021},
   pages={17--75},
   issn={0095-8956},
   review={\MR{4250648}},
   doi={10.1016/j.jctb.2021.04.001},
}

\bib{Pr86}{article}{
   author={Preiss, David},
   author={R{\"o}dl, Vojt{\v{e}}ch},
   title={Note on decomposition of spheres in Hilbert spaces},
   journal={J. Combin. Theory Ser. A},
   volume={43},
   date={1986},
   number={1},
   pages={38--44},
   issn={0097-3165},
   review={\MR{859294 (87k:05083)}},
   doi={10.1016/0097-3165(86)90020-8},
}

\bib{Ramsey30}{article}{
   author={Ramsey, Frank Plumpton},
   title={On a problem of formal logic},
   journal={Proceedings London Mathematical Society},
   volume={30},
   date={1930},
   number={1},
   pages={264--286},
   issn={0024-6115},
   review={\MR{1576401}},
   doi={10.1112/plms/s2-30.1.264},
}

\bib{Raz07}{article}{
   author={Razborov, Alexander A.},
   title={Flag algebras},
   journal={J. Symbolic Logic},
   volume={72},
   date={2007},
   number={4},
   pages={1239--1282},
   issn={0022-4812},
   review={\MR{2371204}},
   doi={10.2178/jsl/1203350785},
}
	
	
\bib{Raz13}{article}{
   author={Razborov, Alexander A.},
   title={On the Caccetta-H\"{a}ggkvist conjecture with forbidden subgraphs},
   journal={J. Graph Theory},
   volume={74},
   date={2013},
   number={2},
   pages={236--248},
   issn={0364-9024},
   review={\MR{3090720}},
   doi={10.1002/jgt.21707},
}

\bib{OH}{article}{
	author={Reiher, Chr.},
	title={Obligatory hypergraphs},
	note={Unpublished manuscript (5 pages)},
} 
	
\bib{girth}{article}{ 
	author={Reiher, Chr.},
	author={R\"{o}dl, Vojt\v{e}ch},
	title={The girth Ramsey theorem}, 
	eprint={2308.15589},
	note={Submitted}, }

\bib{pisier}{article}{
	author={Reiher, Chr.},
   author={R\"{o}dl, Vojt\v{e}ch},
   author={Sales, Marcelo},
   title={Colouring versus density in integers and Hales-Jewett cubes},
   eprint={2311.08556},
   note={Submitted}, }

\bib{Rodl73}{unpublished}{
	author={R\"{o}dl, V.}, 
	title={The dimension of a graph and generalized Ramsey numbers}, 
	note={Master's Thesis, Charles University, Praha, Czechoslovakia},
	date={1973},
}

\bib{Rodl76}{article}{
    author = {R\"{o}dl, V.},
    title = {A generalization of the Ramsey theorem},
    conference={
    		title={Graphs, Hypergraphs, Block Syst.},
			address={Proc. Symp. comb. Anal., Zielona Gora},
			date={1976},
	},
    date={1976},
    pages={211--219},
    review={\,Zbl. 0337.05133},
}
	
\bib{Rodl77}{article}{
   author={R\"{o}dl, V.},
   title={On the chromatic number of subgraphs of a given graph},
   journal={Proc. Amer. Math. Soc.},
   volume={64},
   date={1977},
   number={2},
   pages={370--371},
   issn={0002-9939},
   review={\MR{469806}},
   doi={10.2307/2041460},
}

\bib{Rodl90}{article}{
   author={R\"{o}dl, V.},
   title={On Ramsey families of sets},
   journal={Graphs Combin.},
   volume={6},
   date={1990},
   number={2},
   pages={187--195},
   issn={0911-0119},
   review={\MR{1073689}},
   doi={10.1007/BF01787730},
}

\bib{Sachs}{article}{
   author={Sachs, H.},
   title={Regular graphs with given girth and restricted circuits},
   journal={J. London Math. Soc.},
   volume={38},
   date={1963},
   pages={423--429},
   issn={0024-6107},
   review={\MR{158390}},
   doi={10.1112/jlms/s1-38.1.423},
}
	
		
\bib{SS20}{article}{
   author={Seymour, Paul},
   author={Spirkl, Sophie},
   title={Short directed cycles in bipartite digraphs},
   journal={Combinatorica},
   volume={40},
   date={2020},
   number={4},
   pages={575--599},
   issn={0209-9683},
   review={\MR{4150883}},
   doi={10.1007/s00493-019-4065-5},
}

\bib{Sh289}{article}{
   author={Shelah, Saharon},
   title={Consistency of positive partition theorems for graphs and models},
   conference={
      title={Set theory and its applications},
      address={Toronto, ON},
      date={1987},
   },
   book={
      series={Lecture Notes in Math.},
      volume={1401},
      publisher={Springer, Berlin},
   },
   date={1989},
   pages={167--193},
   review={\MR{1031773}},
   doi={10.1007/BFb0097339},
}

\bib{Sh329}{article}{
   author={Shelah, Saharon},
   title={Primitive recursive bounds for van der Waerden numbers},
   journal={J. Amer. Math. Soc.},
   volume={1},
   date={1988},
   number={3},
   pages={683--697},
   issn={0894-0347},
   review={\MR{929498}},
   doi={10.2307/1990952},
}

\bib{Sh666}{article}{
   author={Shelah, Saharon},
   title={On what I do not understand (and have something to say). I},
   note={Saharon Shelah's anniversary issue},
   journal={Fund. Math.},
   volume={166},
   date={2000},
   number={1-2},
   pages={1--82},
   issn={0016-2736},
   review={\MR{1804704}},
   doi={10.4064/fm-166-1-2-1-82},
}

\bib{Shen00}{article}{
   author={Shen, Jian},
   title={On the girth of digraphs},
   journal={Discrete Math.},
   volume={211},
   date={2000},
   number={1-3},
   pages={167--181},
   issn={0012-365X},
   review={\MR{1735347}},
   doi={10.1016/S0012-365X(99)00323-4},
}

\bib{Shen02}{article}{
   author={Shen, Jian},
   title={On the Caccetta-H\"{a}ggkvist conjecture},
   journal={Graphs Combin.},
   volume={18},
   date={2002},
   number={3},
   pages={645--654},
   issn={0911-0119},
   review={\MR{1939082}},
   doi={10.1007/s003730200048},
}

\bib{Sid}{article}{
   author={Sidorenko, Alexander},
   title={A correlation inequality for bipartite graphs},
   journal={Graphs Combin.},
   volume={9},
   date={1993},
   number={2},
   pages={201--204},
   issn={0911-0119},
   review={\MR{1225933}},
   doi={10.1007/BF02988307},
}
	
\bib{Sim}{article}{
   author={Simonovits, Mikl{\'o}s},
   title={Extremal graph problems, degenerate extremal problems, and
   supersaturated graphs},
   conference={
      title={Progress in graph theory},
      address={Waterloo, Ont.},
      date={1982},
   },
   book={
      publisher={Academic Press, Toronto, ON},
   },
   date={1984},
   pages={419--437},
   review={\MR{776819}},
}

\bib{Singer}{article}{
   author={Singer, James},
   title={A theorem in finite projective geometry and some applications to
   number theory},
   journal={Trans. Amer. Math. Soc.},
   volume={43},
   date={1938},
   number={3},
   pages={377--385},
   issn={0002-9947},
   review={\MR{1501951}},
   doi={10.2307/1990067},
}

\bib{Sing62}{book}{
   author={Singleton, Robert},
   title={On minimal graphs of maximum even girth},
   note={Thesis (Ph.D.)--Princeton University},
   publisher={ProQuest LLC, Ann Arbor, MI},
   date={1962},
   pages={43},
   review={\MR{2613719}},
}
		
\bib{Sing66}{article}{
   author={Singleton, Robert},
   title={On minimal graphs of maximum even girth},
   journal={J. Combinatorial Theory},
   volume={1},
   date={1966},
   pages={306--332},
   issn={0021-9800},
   review={\MR{201347}},
}

\bib{Speck}{article}{
   author={Specker, Ernst},
   title={Teilmengen von Mengen mit Relationen},
   language={German},
   journal={Comment. Math. Helv.},
   volume={31},
   date={1957},
   pages={302--314},
   issn={0010-2571},
   review={\MR{0088454 (19,521b)}},
}

\bib{Tho83}{article}{
   author={Thomassen, Carsten},
   title={Cycles in graphs of uncountable chromatic number},
   journal={Combinatorica},
   volume={3},
   date={1983},
   number={1},
   pages={133--134},
   issn={0209-9683},
   review={\MR{716429}},
   doi={10.1007/BF02579349},
}

\bib{Tutte47}{article}{
   author={Tutte, W. T.},
   title={A family of cubical graphs},
   journal={Proc. Cambridge Philos. Soc.},
   volume={43},
   date={1947},
   pages={459--474},
   issn={0008-1981},
   review={\MR{21678}},
   doi={10.1017/s0305004100023720},
}

\bib{UD54}{article}{
   author={Ungar, Peter},
   author={Descartes, Blanche},
   title={Advanced Problems and Solutions: Solutions: 4526},
   journal={Amer. Math. Monthly},
   volume={61},
   date={1954},
   number={5},
   pages={352--353},
   issn={0002-9890},
   review={\MR{1528740}},
   doi={10.2307/2307489},
}

\bib{VY1}{book}{
   author={Veblen, Oswald},
   author={Young, John Wesley},
   title={Projective geometry. Vol. 1},
   publisher={Blaisdell Publishing Co. [Ginn and Co.], New
   York-Toronto-London},
   date={1965},
   pages={x+345},
   review={\MR{179666}},
}

\bib{VY2}{book}{
   author={Veblen, Oswald},
   author={Young, John Wesley},
   title={Projective geometry. Vol. 2 (by Oswald Veblen)},
   publisher={Blaisdell Publishing Co. [Ginn and Co.], New
   York-Toronto-London},
   date={1965},
   pages={x+511},
   review={\MR{179667}},
}

\bib{Wedderburn}{article}{
   author={Wedderburn, J. H. M.},
   title={A theorem on finite algebras},
   journal={Trans. Amer. Math. Soc.},
   volume={6},
   date={1905},
   number={2},
   pages={349--352},
   issn={0002-9947},
   doi={10.2307/1988750},
}

			
\bib{Witt}{article}{
   author={Witt, Ernst},
   title={\"{U}ber die Kommutativit\"{a}t endlicher Schiefk\"{o}rper},
   language={German},
   journal={Abh. Math. Sem. Univ. Hamburg},
   volume={8},
   date={1931},
   number={1},
   pages={413},
   issn={0025-5858},
   review={\MR{3069571}},
   doi={10.1007/BF02941019},
}
	
\bib{Zykov}{article}{
   author={Zykov (\rn{Zykov}), A. A.},
   title={\rn{O nekotorykh svoi0stvakh linei0nykh kompleksov}},
   language={Russian},
   journal={\rn{Mat. sbornik}},
   volume={24(66)},
   date={1949},
   pages={163--188},
   review={\MR{35428}},
}		
\end{biblist}
\end{bibdiv}
\end{document}